\documentclass[11pt]{amsart}

\usepackage{epigamath}

\usepackage[english]{babel}

\numberwithin{equation}{section}

\usepackage{enumitem}

\newtheorem*{theorem*}{Theorem}
\newtheorem*{motto*}{Motto}

\newtheorem{maintheorem}{Theorem}[section]

\newtheorem{theorem}{Theorem}[section]
\newtheorem{corollary}[theorem]{Corollary}
\newtheorem{lemma}[theorem]{Lemma}
\newtheorem{proposition}[theorem]{Proposition}

\theoremstyle{definition}

\newtheorem{definition}[theorem]{Definition}

\newtheorem{remark}[theorem]{Remark}

\usepackage{extarrows}
\usepackage{stmaryrd}
\usepackage[scr=boondoxo]{mathalpha} 

\usepackage{tikz}
\usetikzlibrary{matrix}
\usetikzlibrary{patterns}
\usetikzlibrary{matrix}
\usetikzlibrary{positioning}
\usetikzlibrary{decorations.pathmorphing}
\usetikzlibrary{cd}
\usetikzlibrary{calc, math}
\usetikzlibrary{decorations.markings}

\newcommand{\euflag}[1]{%
  \begin{tikzpicture}[evaluate={\c=#1;}, estar/.pic={%
      \foreach \i in {0, 1, ..., 4}{%
        \path ({\i*72 +90}: \c) coordinate (P\i);
      }
      \path (intersection of P0--P2 and P1--P3) coordinate (A4);
      \path (intersection of P1--P3 and P2--P4) coordinate (A0);
      \path (intersection of P2--P4 and P3--P0) coordinate (A1);
      \path (intersection of P3--P0 and P4--P1) coordinate (A2);
      \path (intersection of P4--P1 and P0--P2) coordinate (A3);
      \filldraw[yellow!85!red] (P0) -- (A3) -- (P1) -- (A4) -- (P2)
      -- (A0) -- (P3) -- (A1) -- (P4) -- (A2) -- cycle;        
    }]
    \filldraw[blue!75!green!80!black] (-15*\c, -10*\c) rectangle (15*\c, 10*\c);
    \foreach \i in {0, 1, ..., 11}{%
      \path (\i*30: 6.5*\c) pic {estar};
    }
  \end{tikzpicture}%
}

\newcommand{\myfatslash}{\mathbin{\mkern-6mu\fatslash}}

\usepackage[all]{xy}

\newcommand{\hooklongrightarrow}{\lhook\joinrel\longrightarrow}
\newcommand{\twoheadlongrightarrow}{\relbar\joinrel\twoheadrightarrow}

\newcommand{\R}{\mathbb{R}}

\newcommand{\Z}{\mathbb{Z}}
\newcommand{\ZZ}{\mathbb{Z}}

\newcommand{\N}{\mathbb{N}}

\newcommand{\NN}{\mathbb{N}}

\newcommand{\RR}{\mathbb{R}}
\newcommand{\A}{\mathbb{A}}
\newcommand{\GG}{\mathbb{G}}

\newcommand{\ov}{\overline}
\newcommand{\un}{\underline}
\newcommand{\wt}{\widetilde}

\newcommand{\cC}{\mathcal C}
\newcommand{\cD}{\mathcal D}
\newcommand{\cX}{\mathcal X}
\newcommand{\cY}{\mathcal Y}
\newcommand{\cU}{\mathcal U}
\newcommand{\cL}{\mathcal L}
\newcommand{\cI}{\mathcal I}
\newcommand{\cO}{\mathcal O}
\newcommand{\cA}{\mathcal A}

\newcommand{\Mbar}{\overline{M}}

\newcommand{\calMbar}{\overline{\mathcal{M}}}

\newcommand{\calE}{\mathcal{E}}

\newcommand{\calF}{\mathcal{F}}
\newcommand{\calG}{\mathcal{G}}

\newcommand{\calJ}{\mathcal{J}}
\newcommand{\calJbar}{\overline{\mathcal{J}}}

\newcommand{\calL}{\mathcal{L}}
\newcommand{\calM}{\mathcal{M}}
\newcommand{\calO}{\mathcal{O}}

\newcommand{\calS}{\mathcal{S}}

\newcommand{\calX}{\mathcal{X}}

\newcommand{\frakm}{\mathfrak{m}}

\newcommand{\frakp}{\mathfrak{p}}

\DeclareMathOperator{\Pic}{Pic}
\DeclareMathOperator{\Spec}{Spec}
\DeclareMathOperator{\Hom}{Hom}

\DeclareMathOperator{\Aut}{Aut}

\DeclareMathOperator{\colim}{colim}

\DeclareMathOperator{\val}{val}
\DeclareMathOperator{\opp}{opp}

\DeclareMathOperator{\id}{id}
\DeclareMathOperator{\HOM}{HOM}

\DeclareMathOperator{\trop}{trop}
\DeclareMathOperator{\Div}{Div}

\DeclareMathOperator{\Gm}{\mathbb{G}_m}
\DeclareMathOperator{\Str}{Str}
\DeclareMathOperator{\red}{red}
\DeclareMathOperator{\calTroPic}{\mathcal{T}ro\mathcal{P}ic}
\DeclareMathOperator{\calLogPic}{\mathcal{L}og\mathcal{P}ic}
\DeclareMathOperator{\TroPic}{TroPic}
\DeclareMathOperator{\LogPic}{LogPic}

\renewcommand{\sp}{sp}
\renewcommand{\Im}{\mathrm{Im}}

\DeclareMathOperator{\exc}{exc}
\DeclareMathOperator{\nex}{nex}
\DeclareMathOperator{\spl}{spl}
\DeclareMathOperator{\st}{st}
\DeclareMathOperator{\qs}{qs}
\DeclareMathOperator{\an}{an}
\DeclareMathOperator{\mon}{mon}

\let\sp\relax
\DeclareMathOperator{\sp}{sp}
\let\Jac\relax
\DeclareMathOperator{\Jac}{Jac}

\newcommand{\et}{\mathrm{\acute{e}t}}

\DeclareMathOperator{\SG}{\mathcal{SG}}
\DeclareMathOperator{\QSG}{\mathcal{QSG}}
\DeclareMathOperator{\QD}{\mathcal{QD}iv}

\newcommand{\Jt}{\mathcal{J}^{\mathrm{trop}}}
\newcommand{\Jtw}{\wt{\mathcal{J}}^{\,\mathrm{trop}}}
\newcommand{\Jts}{\mathcal{J}^{\mathrm{trop}, \,\mathrm{spl}}}
\newcommand{\JtR}{\mathcal{J}^{\mathrm{trop}, \,\RR}}
\newcommand{\JtsR}{\mathcal{J}^{\mathrm{trop}, \,\mathrm{spl}, \,\RR}}
\newcommand{\JJ}{J^{\mathrm{trop}}}
\newcommand{\JJb}{\ov{J}^{\,\mathrm{trop}}}
\newcommand{\JJs}{J^{\mathrm{trop}, \,\mathrm{spl}}}
\newcommand{\Jc}{\mathcal{J}^{\mathrm{comb}}}

\newcommand{\Mt}{\mathcal{M}^{\mathrm{trop}}}
\newcommand{\Mtw}{\wt{\mathcal{M}}^{\,\mathrm{trop}}}
\newcommand{\MtR}{\mathcal{M}^{\mathrm{trop}, \,\RR}}
\newcommand{\MM}{M^{\mathrm{trop}}}
\newcommand{\MMb}{\ov{M}^{\,\mathrm{trop}}}
\renewcommand{\Mc}{\mathcal{M}^{\mathrm{comb}}}

\newcommand{\J}{\mathcal{J}}
\newcommand{\M}{\mathcal{M}}
\newcommand{\Jb}{\overline{\mathcal{J}}}
\newcommand{\Mb}{\overline{\mathcal{M}}}

\DeclareMathOperator{\RPC}{\bf{RPC}}
\DeclareMathOperator{\RPCC}{\bf{RPCC}}
\DeclareMathOperator{\PC}{\bf{PC}}
\DeclareMathOperator{\TopE}{\bf{Top}}

\DeclareMathOperator{\LSch}{\bf{LSch}}
\DeclareMathOperator{\LSta}{\bf{LSta}}

\DeclareMathOperator{\Sta}{\bf{Sta}}

\EpigaVolumeYear{6}{2022} \EpigaArticleNr{15}
\ReceivedOn{August 11, 2021}
\InFinalFormOn{April 7, 2022}
\AcceptedOn{Mai 2, 2022}

\title{Tropicalization of the universal Jacobian}
\titlemark{Tropicalization of the universal Jacobian}

\author{Margarida Melo}
\address{Dipartimento di Matematica Universit\`a Roma Tre,
  Largo San Leonardo Murialdo, I-00146 Roma, Italy}
\email{\href{mailto:melo@mat.uniroma3.it }{melo@mat.uniroma3.it}}

\author{Samouil Molcho}
\address{Departement Mathematik,
  Eidgen\"ossische Technische Hochschule Z\"urich,
  8092 Z\"urich, Switzerland}
\email{\href{mailto:samouil.molcho@math.ethz.ch}{smolho1317@gmail.com}}

\author{Martin Ulirsch}
\address{Institut f\"ur Mathematik, Goethe-Universit\"at Frankfurt,
  60325 Frankfurt am Main, Germany}
\email{\href{mailto:ulirsch@math.uni-frankfurt.de}{ulirsch@math.uni-frankfurt.de}}

\author{Filippo Viviani}
\address{Dipartimento di Matematica Universit\`a Roma Tre,
  Largo San Leonardo Murialdo, I-00146 Roma, Italy}
\email{\href{mailto:viviani@mat.uniroma3.it }{viviani@mat.uniroma3.it}}

\authormark{M. Melo, S. Molcho, M. Ulirsch, and F. Viviani}

\AbstractInEnglish{%
  In this article we provide a stack-theoretic framework to study the
  universal tropical Jacobian over the moduli space of tropical
  curves. We develop two approaches to the process of tropicalization
  of the universal compactified Jacobian over the moduli space of
  curves -- one from a logarithmic and the other from a
  non-Archimedean analytic point of view. The central result from both
  points of view is that the tropicalization of the universal
  compactified Jacobian is the universal tropical Jacobian and that
  the  tropicalization maps in each of the two contexts are compatible
  with the tautological morphisms. In a sequel we will use the
  techniques developed here to provide explicit polyhedral models for
  the logarithmic Picard variety.
}



\acknowledgement{\scriptsize{The first and fourth authors received funds from MIUR via the Excellence Department Project awarded to the Department of Mathematics
and Physics of Roma Tre and by the project  PRIN2017SSNZAW: Advances
in Moduli Theory and Birational Classification and are members of the
Centre for Mathematics of the University of Coimbra --
UIDB/00324/2020), funded by the Portuguese Government through
FCT/MCTES.  The project has received funding from the European Research Council
(ERC) under the European Union Horizon 2020 research and innovation
program (grant agreement No.786580) and ERC Consolidator Grant 770922
- BirNonArchGeom. This project has received funding from the European Union's Horizon
2020 research and innovation program under the
Marie-Sk\l odowska-Curie Grant Agreement No. 793039\, \euflag{0.015}.
The third author also acknowledges support by the LOEWE-Schwerpunkt
``Uniformisierte Strukturen in Arithmetik und Geometrie'', by the
Deutsche Forschungsgemeinschaft (DFG, German Research Foundation),
project number 456557832, and the TRR 326 \textit{Geometry and
Arithmetic of Uniformized Structures}, project number 444845124.
}}

\begin{document}


\removeabove{0.6cm}
\removebetween{0.6cm}
\removebelow{0.6cm}

\maketitle

\begin{prelims}

\DisplayAbstractInEnglish

\bigskip

\DisplayKeyWords

\medskip

\DisplayMSCclass







\end{prelims}


\newpage

\setcounter{tocdepth}{1}

\tableofcontents


\section{Introduction}

The \emph{universal Jacobian} $\calJ_{g,n}$ over $\calM_{g,n}$ is the algebraic stack parametrizing pairs $(X,L)$ consisting of a smooth (projective and irreducible) curve $X$ of genus $g$ with $n$ marked (pairwise distinct) points and a line bundle $L$ on $X$. The stack $\calJ_{g,n}$ is smooth and it has a decomposition into connected components $\calJ_{g,n}=\coprod_{d\in \ZZ} \calJ_{g,n,d}$, according to the degree $d$ of the line bundles we are parametrizing. 
The rigidification $\calJ_{g,n}\myfatslash \Gm$ by the multiplicative group, that acts as scalar multiplication on line bundles, is representable over $\calM_{g,n}$ (and hence a DM=Deligne-Mumford stack). All of its connected components $\calJ_{g,n,d}\myfatslash \Gm$ are proper over $\calM_{g,n}$ with fiber over a geometric point $X\in \calM_{g,n}$ being the degree-$d$ Jacobian $J^d_X$ of $X$. 

The most na\"ive extension of $\calJ_{g,n}$ to the Deligne-Mumford compactification $\calM_{g,n}\subset \calMbar_{g,n}$ as a relative Jacobian is unfortunately neither universally closed (as some line bundles may not have any  limit in a stable degeneration) nor separated (as some line bundles may have different limits in a stable degeneration). 

Following \cite{Cap94, Cap08} (see also \cite{Cor89} for the  moduli space of spin curves), one can resolve the first problem by allowing for extra rational destabilizing components on stable degenerations of $X$. More precisely, we define the \emph{compactified universal Jacobian} $\Jb_{g,n}$ as the algebraic stack parametrizing pairs $(X,L)$ consisting of a quasi-stable curve $X$ of type $(g,n)$, \textit{i.e.} an $n$-pointed  genus-$g$ semistable curve whose destabilizing (or exceptional) components are isolated, and a  line bundle $L$ on $X$ which is admissible, \textit{i.e.} it has degree one on every destabilizing component of $X$ (see Definition \ref{UniCompJac}).  The compactified universal Jacobian $\Jb_{g,n}$ decomposes as well into smooth connected components $\Jb_{g,n}=\coprod_{d\in \ZZ} \Jb_{g,n,d}$, according to the degree of the line bundles, each containing $\calJ_{g,n,d}$ as dense and open substack with normal crossing boundary.
Moreover, the compactified universal Jacobian $\Jb_{g,n}$ is endowed with a  forgetful-stabilization  morphism that satisfies the existence part of the valuative criterion for properness (although it is not universally closed since it is not quasi-compact)
\begin{equation*}
\begin{aligned}
\Phi: \Jb_{g,n} & \longrightarrow \calMbar_{g,n}\\
(X,L) & \longmapsto X^{\st},
\end{aligned}
\end{equation*}
where  $X^{\st}$ is the stabilization of $X$. The DM locus of the rigidification $\Jb_{g,n}\myfatslash \Gm$ is the image of the open and dense subset  $ \Jb_{g,n}^{\,\spl}\subset \Jb_{g,n}$ parametrizing pairs $(X,L)$ such that $X$ is \emph{simple}, \textit{i.e.} such that $X$ remains connected when we remove the destabilizing rational components. 

However, the rigidification $\Jb_{g,n}\myfatslash \Gm$ is not separated nor are its connected components of finite type over $\Mb_{g,n}$. 
A classical idea to resolve this issue (that can be traced back to \cite{OdaSeshadri} in the case of a single curve) is to require the pair $(X,L)$ to be \emph{semistable} with respect to a \emph{universal stability condition} $\phi$ of type $(g,n)$ in the sense of \cite{KP} (also see \cite{Mel15} for an equivalent notion), see Definition \ref{D:stab-con}. Denote by $\Jb_{g,n}(\phi)$ the open substack of $\Jb_{g,n}$ that parametrizes pairs $(X,L)$ for which $L$ is $\phi$-semistable, see \eqref{E:opensub} and Definition \ref{D:admdiv}. The stack $\Jb_{g,n}(\phi)$ is of finite type and  universally closed over $\Mb_{g,n}$ and its $\Gm$-rigidification $\Jb_{g,n}(\phi)\myfatslash \Gm$ is a proper DM stack  if and only if $\phi$ is a \emph{general } stability condition, see Definition \ref{D:admdiv}.  
The space $V_{g,n}$ of all the universal stability conditions of type $(g,n)$ as well as its wall and chamber decomposition that determine the variation of the open substacks $\Jb_{g,n}(\phi)\subset \Jb_{g,n}$ have been explicitly  described  by Kass-Pagani in \cite[Theorems~1 and~2]{KP}.

We refer to Section \ref{S:comp-univ} for a more detailed account on the properties of $\Jb_{g,n}$, as well as for a description of the toroidal stratification of $\calJ_{g,n}\subset \Jb_{g,n}$ (see \S \ref{S:toro-J}) and of its category of strata (see \S\ref{S:catstraJ}).
\vspace{0,2cm}

The first goal of this article is to construct  a tropical analogue $\calJ_{g,n}^{\mathrm{trop}}$ of the compactified universal Jacobian $\Jb_{g,n}$, which we call the \emph{tropical universal Jacobian} (an object that was first envisioned in \cite{Len_tropicalBrillNoether}).
Following the framework of \cite{CCUW}, we define  $\calJ_{g,n}^{\mathrm{trop}}$ as the  category fibered in groupoids over the category $\RPC$ of rational polyhedral cones whose fiber over  $\sigma\in \RPC$ is the groupoid of pairs $(\Gamma/\sigma,D)$, where $\Gamma/\sigma$ is a quasi-stable tropical curve over $\sigma$ of type $(g,n)$ and $D$ is an admissible divisor on the underlying graph $\GG(\Gamma)$ of $\Gamma$, see Definition \ref{UnTrJac}. The category $\calJ_{g,n}^{\mathrm{trop}}$  comes equipped with a (forgetful-stabilization) morphism of categories fibered in groupoids 
\begin{equation*}
\begin{aligned}
\Phi^{\mathrm{trop}}\colon \Jt_{g,n} & \longrightarrow \Mt_{g,n}\\
(\Gamma/\sigma, D) & \longmapsto (\Gamma/\sigma)^{\st},
\end{aligned}
\end{equation*}
where $(\Gamma/\sigma)^{\st}=(\Gamma^{\st}/\sigma)$ is the stabilization of $\Gamma/\sigma$, see Definition \ref{UnTrJac-Mg}. 
In Theorem \ref{TrJacStack}, we show that $\calJ_{g,n}^{\mathrm{trop}}$ is a cone stack and that $\Phi^{\mathrm{trop}}$ is a morphism of cone stacks, in the sense of \cite[\S 2.1]{CCUW}.
Moreover,  $\calJ_{g,n}^{\mathrm{trop}}$ admits the cone substacks $\calJ_{g,n, d}^{\mathrm{trop}}$ (resp. $\calJ_{g,n}^{\mathrm{trop},\,\spl}$, resp. $\calJ_{g,n}^{\mathrm{trop}}(\phi)$ for any $\phi\in V_{g,n}$) parametrizing pairs 
$(\Gamma/\sigma,D)$ such that $\deg D=d$ (resp. $\GG(\Gamma)$ is simple in the sense of Definition \ref{D:spl-gr}, resp. $D$ is $\phi$-semistable  in the sense of Definition \ref{D:admdiv}\eqref{D:admdiv2}).

\vspace{0.1cm}

Our main goal is to justify the following expectation: 

\vspace{0.1cm}

\begin{motto*} The tropical universal  Jacobian $\calJ_{g,n}^{\mathrm{trop}}$ is the \emph{tropicalization} of  $\Jb_{g,n}$ in a way compatible with 
\begin{itemize}
\item the forgetful-stabilization morphisms $\Phi^{\mathrm{trop}}$ and $\Phi$, 
\item the restriction to the connected components $\calJ_{g,n,d}^{\mathrm{trop}}$ and $\Jb_{g,n, d}$ for any $d\in \ZZ$,
\item the restriction to the simple loci $\calJ_{g,n}^{\mathrm{trop}, \,\spl}$ and $\Jb_{g,n}^{\,\spl}$, 
\item the restriction to the $\phi$-semistable loci $\calJ_{g,n}^{\mathrm{trop}}(\phi)$ and  $\Jb_{g,n}(\phi)$ for  $\phi\in V_{g,n}$. 
\end{itemize}
\end{motto*}

There are two frameworks in which the above Motto can be translated into a precise Theorem: the first framework is logarithmic geometry, and the second one is non-Archimedean analytic geometry (\`a la Berkovich). In the non-Archimedean analytic framework, a part of this Motto involving special types of stability conditions has already been realized in \cite{API} (see Theorem \ref{mainthm_skeleton=trop} below and the discussion right after). In the logarithmic framework, and in the non-Archimedean one in the absence of stability conditions, our results appear to be new. 

\vspace{0.1cm}

\subsection{The logarithmic perspective} In the framework of logarithmic geometry, we compare two stacks over the category $\LSch$ of logarithmic schemes (fine, saturated and locally of finite type over a base field $k$) that are constructed from $\Jb_{g,n}$ and $\calJ_{g,n}^{\mathrm{trop}}$.  

The first stack is the  \emph{logarithmic universal Jacobian}  $\calJ_{g,n}^{\log}$ parametrizing pairs $(X\rightarrow S, \calL)$  consisting of a quasi-stable logarithmic curve $X\rightarrow S$ of type $(g,n)$ and an   admissible line bundle $\calL$ on $\underline{X}$, see Definition \ref{D:logJ}.  In Proposition \ref{P:logJ-boun}, we prove that $\J_{g,n}^{\log}$ is representable by the logarithmic algebraic stack $(\Jb_{g,n},M_{\partial \Jb_{g,n}})$, where $M_{\partial \Jb_{g,n}}$ is the divisorial logarithmic structure on $\Jb_{g,n}$ associated to the normal crossing divisor $\partial \Jb_{g,n}$. In particular, the algebraic stack $\un{\calJ_{g,n}^{\log}}$ underlying $\calJ_{g,n}^{\log}$ is  the compactified universal Jacobian $\Jb_{g,n}$ and hence we have a natural morphism of logarithmic algebraic stacks 
$$
\begin{aligned}
\Upsilon_{\calJ_{g,n}^{\log}}\colon \calJ_{g,n}^{\log}& \longrightarrow \Jb_{g,n}=\big(\Jb_{g,n}, \cO_{\Jb_{g,n}}^*\big), \\
(X\rightarrow S, \calL) & \longmapsto (\un X\to \un S, \calL).
\end{aligned}
$$
Moreover, $\calJ_{g,n}^{\log}$ admits open logarithmic algebraic substacks $\calJ_{g,n, d}^{\log}$ (resp. $\calJ_{g,n}^{\log, \,\spl}$, resp. $\calJ_{g,n}^{\log}(\phi)$ for any $\phi\in V_{g,n}$) parametrizing pairs 
 $(X\rightarrow S, \calL)$ such that $\calL$ has relative degree $d$ (resp. $\un X\to \un S$ is a family of simple quasi-stable curves, resp. $\calL$ is relatively $\phi$-semistable).

The second stack is the \emph{tropical universal Jacobian over $\LSch$}, denoted by $\Jtw_{g,n}$, whose fiber over $S\in \LSch$ is the groupoid whose objects  consists of a pair $(\Gamma_s/\ov M_{S,s},D_s)$ for each geometric point $s$ of $\un S$, where $\Gamma_s/\ov M_{S,s}$ is a quasi-stable tropical curve over the monoid $\ov M_{S,s}$ of type $(g,n)$ and $D_s$ is an admissible divisor on $\GG(\Gamma_s)$, subject to a natural compatibility relation with respect to \'etale specializations of geometric points of $\un S$, see Definition \ref{D:tropJnew}. In Proposition \ref{P:2LogJ}, we prove that  $\Jtw_{g,n}$  is the Artin fan (hence in particular a logarithmic algebraic stack) associated to the cone stack $\Jt_{g,n}$ under the natural equivalence of $2$-categories between cone stacks and Artin fans over $k$ (see \cite[Theorem~6.11]{CCUW}).

Recall now that in logarithmic geometry the tropicalization of a logarithmic algebraic stack $\cX$ is meant to be the  Artin fan $\cA_{\cX}$ with faithful monodromy together with the natural strict morphism of 
logarithmic algebraic stacks (that we baptize the \emph{functorial logarithmic tropicalization morphism} of $\cX$)
\begin{equation*}
\trop_{\cX}\colon\cX\longrightarrow \cA_{\cX}
\end{equation*} 
that is initial among all strict morphisms to an Artin fan with faithful monodromy  (see \cite[Proposition~3.2.1]{ACMW}). 

The relation between the above two stacks over $\LSch$ is  summarized in the next Theorem (see  Theorem \ref{T:log-trop} for a more precise version).

\begin{maintheorem}
\noindent 
\begin{enumerate}
\item There is a canonical isomorphism of Artin fans $\cA_{\J_{g,n}^{\log}}\cong \cA_{\Jtw_{g,n}}$ under which we have a factorization 
\begin{equation}\label{E:2trop-log}
\trop_{\J_{g,n}^{\log}}\colon \J_{g,n}^{\log}\xrightarrow{\wt\trop_{\J_{g,n}^{\log}}}\Jtw_{g,n} \xrightarrow{\trop_{\Jtw_{g,n}}}\cA_{\Jtw_{g,n}}  \cong \cA_{\J_{g,n}^{\log}},
\end{equation}
where the map $\wt\trop_{\J_{g,n}^{\log}}$ (that we baptize the \emph{modular logarithmic tropicalization map})  has a modular description $($see Definition \ref{D:log-trop}$)$ and it is  strict, smooth and surjective. 

Moreover,  the diagram \eqref{E:2trop-log} commutes, via the suitable forgetful-stabilization morphisms, with the analogous diagram for $\Mb_{g,n}$ established in \cite{CCUW} and \cite{Uli19}.


 \item The two morphisms of logarithmic algebraic stacks 
 \begin{equation}\label{E:corr-logIN}
\xymatrix{
\Jb_{g,n}&& \J_{g,n}^{\log}  \ar[ll]_{\Upsilon_{\J_{g,n}^{\log}}} \ar[rr]^{\wt{\trop}_{\J_{g,n}^{\log}}} && \Jtw_{g,n}
}
\end{equation}
are compatible with  the toroidal stratification  \eqref{E:strataJ} of $\Jb_{g,n}$  and the stratification \eqref{E:StrataL} of  $ \Jtw_{g,n}$ as an Artin fan. 

 Moreover, the diagram \eqref{E:corr-logIN} is compatible with the restrictions to $\Jb_{g,n, d}$, $\Jb_{g,n}^{\,\spl}$ and $\Jb_{g,n}(\phi)$ for any universal stability condition $\phi\in V_{g,n}$.
\end{enumerate}
\end{maintheorem}

\vspace{0.1cm}

\subsection{The non-Archimedean perspective} In the framework of non-archimedean analytic geometry, we compare two topological spaces that are constructed from $\Jb_{g,n}$ and $\calJ_{g,n}^{\mathrm{trop}}$ (over $k=\ov k$ on which we put the trivial valuation).

The first topological space is the topological space underlying the \emph{beth-analytification} $\Jb_{g,n}^{\,\beth}$ of $\Jb_{g,n}$ and it admits the following description 
\begin{equation*}
\begin{aligned}
& \lvert\Jb_{g,n}^{\,\beth}\rvert:=\big\{\Spec R\to \Jb_{g,n}\big\}/\sim, \\
\end{aligned}
\end{equation*}
where $R$ varies among all the rank-$1$ valuation rings containing $k$ and the equivalence relation $\sim$  is defined as follows: we say that $\Spec R\to \cX$ is equivalent to  $\Spec R'\to \cX$ if there exists another rank-$1$ valuation ring $R''$ containing both $R$ and $R'$, and  such that the two natural morphisms $\Spec R''\to \Spec R\to \cX$ and $\Spec R''\to \Spec R'\to \cX$ coincide. 

The topological space $\big\vert \Jb_{g,n}^{\,\beth} \big\vert$ admits an anticontinuous surjective \emph{reduction map} to the topological space $\big\vert \Jb_{g,n} \big\vert$ underlying the stack $\Jb_{g,n}$ which is defined by 
\begin{equation*}
\begin{aligned}
\red_{\Jb_{g,n}}\colon \big\vert \Jb_{g,n}^{\,\beth}\big\vert & \longrightarrow \big\vert \Jb_{g,n}\big\vert \\
\big[\Spec R\to \Jb_{g,n}\big] & \longmapsto \big[\Spec R/\frakm_R\to \Spec R\to \Jb_{g,n}\big],
\end{aligned}
\end{equation*}
where $\frakm_R$ is the maximal ideal of $R$.

The second topological space is the \emph{generalized cone complex }$\JJ_{g,n}$ associated to $\calJ_{g,n}^{\mathrm{trop}}$ (in the sense of \cite{ACP}). As a topological space $\JJ_{g,n}$ parametrizes isomorphism classes of pairs consisting of a tropical curve $\Gamma$ of type $(g,n)$ and an admissible divisor $D$ on the underlying graph of $\Gamma$. Its \emph{canonical compactification} $\JJb_{g,n}$ parametrizes pairs consisting of an extended tropical curve (see Definition \ref{D:tropcurvR}) and an admissible divisor $D$ on the underlying graph $\GG(\Gamma)$. 
In Definition \ref{TropGCC}, $\JJ_{g,n}$ and $\JJb_{g,n}$ as the colimits
\begin{equation*}
\JJ_{g,n}=\varinjlim_{(G,D)\in \QD_{g,n}} \RR_{\geq 0}^{E(G)}\subset \JJb_{g,n}=\varinjlim_{(G,D)\in \QD_{g,n}} \big(\RR_{\geq 0}\cup\{+\infty\}\big)^{E(G)},
\end{equation*}
where the diagram of topological spaces comes from the description of $\calJ_{g,n}^{\mathrm{trop}}$ as a combinatorial cone stack, see Theorem \ref{TrJacStack}\ref{TrJacStack1}.
Moreover, $\JJ_{g,n}$ admits  generalized cone subcomplexes  $\JJ_{g,n, d}$ (resp. $\JJs_{g,n}$, resp. $\JJ_{g,n}(\phi)$ for any $\phi\in V_{g,n}$) parametrizing pairs 
$(\Gamma,D)$ such that $\deg D=d$ (resp. $\GG(\Gamma))$ is simple, resp. $D$ is $\phi$-semistable). And similarly for their compactifications inside $\JJb_{g,n}$. 

Recall now that in non-archimedean geometry, the tropicalization of a toroidal embedding of Artin stacks $(\cU\subset \cX)$  (or more generally a logarithmic algebraic stack) is meant to be a generalized  cone complex $\Sigma(\cX)$, and its natural compactification $\ov\Sigma(\cX)$ which is a generalized extended cone complex, together with a surjective proper continuous map (that we baptise \emph{functorial analytic tropicalization})  $\trop_{\cX}^{\an}:|\cX^{\beth}|\to \ov\Sigma(\cX)$ which is functorial with respect to locally toric morphisms,  see \cite{ACP} and \cite{Uli19}.

The relation between the above two topological spaces is summarized in the next Theorem (see Theorem \ref{T:ana-trop} for a more precise version).

\begin{maintheorem}
\label{mainthm_skeleton=trop}
\noindent 
\begin{enumerate}
\item There are canonical isomorphisms $\Psi_{\Jb_{g,n}}:\Sigma(\Jb_{g,n})\xrightarrow{\cong} \JJ_{g,n}$ and $\ov\Psi_{\Jb_{g,n}}:\ov\Sigma(\Jb_{g,n})\xrightarrow{\cong} \JJb_{g,n}$ of, respectively, generalized  cone complexes and generalized extended cone complexes, in such  a way that the map $($that we baptise the \emph{modular analytic tropicalization map}$)$
\begin{equation}\label{E:2trop-an}
\wt{\trop}^{\an}_{\Jb_{g,n}}=\ov \Psi_{\Jb_{g,n}}\circ \trop_{\Jb_{g,n}}\colon\big|\Jb_{g,n}^{\,\beth}\big|\xrightarrow{\trop^{\an}_{\Jb_{g,n}}}   \ov\Sigma(\Jb_{g,n}) \xrightarrow[\cong]{\ov\Psi_{\Jb_{g,n}}} \JJb_{g,n}
\end{equation}
has a modular description $($see Definition \ref{D:ana-trop}$\,)$.

Moreover, the diagram \eqref{E:2trop-an} commutes, via the suitable forgetful-stabilization morphisms, with the analogous diagram for $\Mb_{g,n}$ established in \cite{ACP}.
 \item  The two maps 
 \begin{equation}\label{E:corr-tropIN}
\xymatrix{
\big|\Jb_{g,n}\big| && \big|\Jb_{g,n}^{\,\beth}\big|  \ar[ll]_{\red_{\Jb_{g,n}}} \ar[rr]^{\wt{\trop}^{\an}_{\Jb_{g,n}}} && \JJb_{g,n}
}
\end{equation}
 are compatible with  the toroidal stratification  \eqref{E:strataJ} of $\Jb_{g,n}$ and the stratification \eqref{E:strataTr} of $\JJb_{g,n}$ as a generalized extended cone complex. 
 
 Moreover, the diagram \eqref{E:corr-tropIN} is compatible with the restrictions to $\Jb_{g,n, d}$, $\Jb_{g,n}^{\,\spl}$ and $\Jb_{g,n}(\phi)$ for any universal stability condition $\phi\in V_{g,n}$.
\end{enumerate}

\end{maintheorem}

Theorem \ref{mainthm_skeleton=trop} contains as a special case the main result of \cite{API}: when $n=1$ and $\phi\in V_{g,1}$ is a suitable perturbation of the canonical stability condition (which is then general),  $\calJbar_{g,1}(\phi)$ is isomorphic (up to $\Gm$-rigidification) to the Esteves' compactified universal Jacobian stack $\calJbar_{\phi,g}$ of \textit{loc.~cit.} (see Remark~\ref{R:alternJ} below) and so Theorem \ref{mainthm_skeleton=trop} states that the non-Archimedean skeleton of Esteves' compactified universal Jacobian $\calJbar_{\phi,g}$ can be identified with $\JJb_{g,1}(\phi)$ (which can be shown to be isomorphic to the generalized extended cone complex $\JJb_{\phi,g}$ constructed in \textit{loc.~cit.})  making the natural diagram of tropicalization maps commute, which is exactly \cite[Theorem 6.9]{API}\footnote{However, it seems to us that  \cite[Theorem 6.9]{API} should work, with small changes in the proof, for any $n\geq 1$ and any  general universal stability condition $\phi\in V_{g,n}$.}. We believe that one advantage of our approach, with respect to the one of \cite{API}, is that the spaces $\JJb_{g,n}(\phi)$, as $\phi$ varies in $V_{g,n}$, are constructed as  generalized extended cone sub-complexes of $\JJb_{g,n}$, which should  be useful in order to study tropical wall-crossing phenomena (similar to \cite{KP}).


\subsection{Fibers of the forgetful-stabilization morphism}

In the last Section of the paper we study the fibers of the forgetful-stabilization morphism of cone stacks $\Phi^{\mathrm{trop}}:\Jt_{g,n} \rightarrow \Mt_{g,n}$ and of its realization as morphism of topological stacks $\big\vert\Phi^{\mathrm{trop}}
\big\vert\colon \big\vert\Jt_{g,n}\big\vert\to \big\vert\Mt_{g,n}\big\vert$. See Theorems \ref{T:fib-forget}, \ref{T:fibT-forget}, \ref{T:Jac-Pic} in Section \ref{S:top-Jac}.

\begin{maintheorem}
\label{mainthm_fibers}
\noindent 
\begin{enumerate}
\item \label{mainthm_fibers1} The fiber of the forgetful-stabilization morphism of cone stacks $\Phi^{\mathrm{trop}}:\Jt_{g,n} \rightarrow \Mt_{g,n}$ over a stable tropical curve $\ov \Gamma/\sigma\in \Mt_{g,n}(\sigma)$ is the Jacobian cone space $\Jac_{\ov \Gamma/\sigma}$ that we construct in Definition \ref{Jac-fiber}. 

Moreover, a similar result is true for the restriction of $\Phi^{\trop}$ to the cone substacks $\calJ_{g,n, d}^{\mathrm{trop}}$,  $\calJ_{g,n}^{\mathrm{trop}, \,\spl}$ and $\calJ_{g,n}^{\mathrm{trop}}(\phi)$ for any $\phi\in V_{g,n}$.

\item \label{mainthm_fibers2} Let $\ov \Gamma$ be a stable tropical curve of type $(g,n)$ with real edge lengths. 
\begin{enumerate}[label={\small\textrm{(\roman*)}}]
\item The fiber of the forgetful-stabilization morphism of topological stacks $\big\vert\Phi^{\mathrm{trop}}\big\vert\colon \big\vert\Jt_{g,n}\big\vert\to \big\vert\Mt_{g,n}\big\vert$ over  $\ov \Gamma$ is the Jacobian topological space $\Jac_{\ov \Gamma}$ that we construct in Definition \ref{JacTp-fiber}. 

Moreover, a similar result is true for the restriction of $\big\vert\Phi^{\mathrm{trop}}\big\vert$ to the topological substacks $\big\vert\Jt_{g,n,d}\big\vert$, $\big\vert\Jts_{g,n,(d)}\big\vert$ and $\big\vert\Jt_{g,n}(\phi)\big\vert$.

\item There exists a  surjective degree preserving continuous  map 
$$
\alpha_{\ov \Gamma}\colon \Jac_{\ov \Gamma}  \longrightarrow \Pic(\ov \Gamma)
$$
that restricts to an homeomorphism 
$$
\alpha_{\ov \Gamma}(\phi)\colon\Jac_{\ov \Gamma}(\phi) \xlongrightarrow{\cong} \Pic^{|\phi|}(\ov \Gamma).
$$
 for any general universal stability condition $\phi\in V_{g,n}$.
\end{enumerate}
\end{enumerate}
\end{maintheorem}

It follows from the above Theorem that, for any general universal stability condition $\phi\in V_{g,n}$, the fiber of the forgetful-stabilization  morphism of generalized cone complexes (see \eqref{PhiTrop})
$$\Phi^{\mathrm{trop}}\colon\JJ_{g,n}(\phi)\longrightarrow \MM_{g,n}$$ 
over a point  $\ov \Gamma\in \MM_{g,n}$ is homeomorphic to $\Pic^{|\phi|}(\ov \Gamma)/\Aut(\ov \Gamma)$ (see Remark \ref{R:fib-ConeComplex}), thus recovering (and slightly extending to our more general setting) the result of Abreu-Pacini \cite[Theorem~5.14]{API}. The advantage of working with topological stacks as in Theorem \ref{mainthm_fibers}\eqref{mainthm_fibers2}, rather than generalized cone complexes, 
is that we do not have to quotient out by the automorphism group of $\ov \Gamma$ when describing the fiber of a point $\ov \Gamma$.

\subsection{A sequel on $\LogPic$ and $\TroPic$} In \cite{MW} the authors have shown that in the category of logarithmic schemes (and stacks), there is a unique minimal model of the universal logarithmic Jacobian that is not representable by an algebraic stack. The study of this so-called \emph{universal logarithmic Picard variety} $\calLogPic_{g,n,d}$ can be traced back to Illusie  \cite{Illusie} and has subsequently received attention in \cite{Kajiwara, Olsson, Bellardini, FRTU}. The authors of \cite{MW} also introduce the \emph{universal tropical Picard variety}, which universally over $\mathcal{M}_{g,n}^{\log}$ is denoted by $\calTroPic_{g,n,d}$. It parametrizes tropical curves together with a torsor over the sheaf of harmonic or linear functions on $\Gamma$ of degree $d$ and it naturally arises as the tropicalization of $\calLogPic_{g,n,d}$ via a natural tropicalization morphism $\calLogPic_{g,n,d}\to \calTroPic_{g,n,d}$.

In the sequel to this article we carefully study the relationship between this construction and what we have done in this article. The gist of this story is that, for a general universal stability condition $\phi\in V_{g,n}$, the cone stack $\Jtw_{g,n}(\phi)$ defines a proper subdivision of the universal tropical Picard variety $\calTroPic_{g,n,d}$. This induces a logarithmic modification  $\calJ_{g,n}^{\log}(\phi)\myfatslash \mathbb{G}_m \rightarrow \calLogPic_{g,n,d}$ -- effectively a blowup -- by base change along the tropicalization map. As an application, this approach allows one to readily relate the spaces constructed here with Abel-Jacobi theory (see \textit{e.g.} \cite{MarcusWise} and \cite{APII}). The perspective adopted in this paper and the sequel is for instance crucial in ongoing work of Holmes, Pandharipande, Pixton, Schmitt, and the second author regarding the double ramification cycle.

\subsection*{Acknowledgements}
We thank Alex Abreu, Karl Christ, David Holmes, Yoav Len, Marco
Pacini, Nicola Pagani, Aaron Pixton, Rahul Pandharipande, Dhruv
Ranganathan, and Johannes Schmitt for useful discussions regarding
this project. Jonathan Wise in particular helped at the beginning
stages of this project and has shared many of his ideas which we
intend to further explore in a follow-up article.  

Significant progress on this project has been made during the workshop
"Foundations of tropical schemes" at the American Institute of
Mathematics. We thank the organizers, Noah Giansiracusa, David Jensen,
Diane Maclagan, and Steffen Marcus, for providing us with this
opportunity.


\section{The universal tropical Jacobian}

\subsection{Quasi-stable graphs}


We will adopt the graph-theoretic terminology of  \cite[Section~3.1]{CCUW}, which we now briefly recall. A \emph{graph} $G$ is a triple  $(X(G),r_G,i_G)$ such that 
\begin{itemize}
\item $X(G)$ is a finite set;
\item $r_G:X(G)\to X(G)$ is an idempotent map (called the root map);
\item $i_G:X(G)\to X(G)$ is an involution whose fixed set contains the image of $r_G$.
\end{itemize}
We recover the more familiar definition of graphs in the following way. The image of $r_G$ is the vertex set $V(G)$ of $G$. Its complement $F(G):=X(G)\setminus V(G)$ is the set of flags of $G$ and the root map restricts to a map $r_G:F(G)\to V(G)$, which we think of as the map that sends a flag to the root from which it emanates. The involution $i_G$ restricts to an involution on $F(G)$: the fixed points of this involution are  the legs $L(G)$ of $G$, its non fixed points are the half-edges $H(G)$ of $G$. Hence we get that 
$$X(G)=V(G)\sqcup H(G)\sqcup L(G),$$
 the root map $r_G$ is the identity on $V(G)$ and it restricts to a map 
 $$r_G:H(G)\sqcup L(G)\to V(G),$$
  the involution $i_G$ is the identity on $V(G)\sqcup L(G)$ and it is  fixed-point free on $H(G)$. The quotient $E(G):=H(G)/i_G$ is the set of edges of $G$; explicitly, any edge $e$ of $G$ is equal to $e=\{h_1,h_2\}$ with $i_G(h_1)=h_2$ and we say that the half-edges $h_1$ and $h_2$ belong to the edge $e$ and that they are \emph{conjugate} half-edges. 

We will be dealing with \emph{$n$-marked vertex-weighted graphs} $G=(G,h,m)$, where  $h:V(G)\to \N$ is the vertex-weight function (also called the genus function) and $m:\{1,\ldots,n\}\xrightarrow{\cong} L(G)$ is a marking of the legs $L(G)$ of $G$. The genus of $G=(G,h,m)$ is 
$$g(G)=b_1(G)+\sum_{v\in V(G)} h(v).$$

A morphism $\pi:G_1=(G_1,h_1,m_1)\to G_2=(G_2,h_2,m_2)$ of ($n$-marked vertex-weighted) graphs consists of a function $\pi:X(G_1)\to X(G_2)$ with the property that $\pi\circ r_{G_1}=r_{G_2}\circ \pi$ and $\pi\circ i_{G_1}=i_{G_2}\circ \pi$, and which moreover satisfies the following additional properties:
\begin{itemize}
\item For any flag $f\in F(G_2)$, its inverse image $\pi^{-1}(f)$ has one element which is a flag of $G_1$.
\item The morphism restricts to a bijection $\pi_L:L(G_1)\xrightarrow{\cong} L(G_2)$ such that $\pi_L\circ m_1=m_2$.
\item For each vertex $v\in V(G_2)$, the vertex-weighted graph  $\pi^{-1}(v)$ is  connected  of genus $h_2(v)$. 
\end{itemize}
It turns out that a morphism of graphs $\pi:G_1=(G_1,h_1,m_1)\to G_2=(G_2,h_2,m_2)$ is the composition of a weighted edge contraction $G_1=(G_1,h_1,m_1)\to G_1/S=(G_1,h_1,m_1)/S$ for some $S\subseteq E(G_1)$ followed by an isomorphism of graphs $G_1/S=(G_1,h_1,m_1)/S\xrightarrow{\cong}G_2=(G_2,h_2,m_2)$. 
This implies that a morphism of graphs $\pi:G_1=(G_1,h_1,m_1)\to G_2=(G_2,h_2,m_2)$ induces a surjective map $\pi_V:V(G_1)\to V(G_2)$  and a map  $\pi_H:H(G_1)\to H(G_2)\sqcup V(G_2)$ with the property that $\pi_H^{-1}(H(G_2))$ is the set of half-edges of $G_1$ that are not contracted by $\pi$ and that $\pi_H$ restricts to a bijection  $\pi_H: \pi_H^{-1}\big(H(G_2)\big)\xrightarrow{\cong} H(G_2)$. In particular, we get an injection $\pi_H^*: H(G_2)\hookrightarrow H(G_1)$. The map $\pi_H$ induces a map $\pi_E:E(G_1)\to E(G_2)\sqcup V(G_2)$ having similar properties and inducing an injection $\pi_E^*: E(G_2)\hookrightarrow E(G_1)$. Note that a morphism $\pi:G_1\to G_2$ as above is uniquely determined by $\pi_V$ and $\pi_H$. 

Recall that a \emph{stable} graph of type $(g,n)$ is a $n$-marked vertex-weighted connected graph $G$ of total genus $g$ such that for all $v\in V(G)$:
$$2h(v)-2+\val(v)>0$$
where $\val(v)$ is the number of flags emanating from (or incident to) $v$. 
We now define a slight generalization of stable graphs, namely quasi-stable graphs. 

\begin{definition} \label{D:qs-graph}
Fix $(g,n)$ an hyperbolic pair, \textit{i.e.} a pair of integers $g,n\geq 0$ such that $2g-2+n>0$. 
 A \textbf{quasi-stable} graph of type $(g,n)$ is a $n$-marked vertex-weighted connected graph $G=(G,h,m)$ of total genus $g$ such that any vertex $v\in V(G)$ of genus zero has valence at least two and those vertices of genus zero and valence two, called \emph{exceptional vertices}, are such that:
\begin{itemize}
\item every exceptional vertex  has exactly two edges incident to it (in particular, there are no legs rooted at exceptional vertices);
\item two distinct exceptional vertices are not adjacent.
\end{itemize}
\end{definition}

Notice that quasistable graphs with no exceptional vertices correspond exactly to stable graphs.

We will denote the set of exceptional vertices of $G$ by $V_{\exc}(G)$ and the set of remaining vertices, called non-exceptional, by $V_{\nex}(G)$.
For any  $v\in V_{\exc}(G)$, we will denote by $\big\{h_v^1,h_v^2\big\}$ the two half-edges rooted at $v$ and we fix an order of them. The edges to which $h_v^1$ and $h_v^2$ belong are denoted, respectively, by $e_v^1=\big\{h_v^1,i_G(h_v^1)\big\}$ and $e_v^2=\big\{h_v^2,i_G(h_v^2)\big\}$. Moreover, for any $v\in V_{\exc}(G)$ and any $i=1,2$, we set $v^i:=r_G\big(i_G(h_v^i)\big)\in V_{\nex}(G)$, which is the non-exceptional vertex of $G$ which is incident to $e_ v^i$. 
The half-edges (resp. edges) of the form $h_v^1$ and $h_v^2$ (resp. $e_v^1$ and $e_v^2$), for some  $v\in V_{\exc}(G)$, are called exceptional and the set of all exceptional half-edges (resp. edges) is denoted by $H_{\exc}(G)$ (resp. $E_{\exc}(G)$); the remaining half-edges (resp. edges), called non-exceptional, are denoted by $H_{\nex}(G)$ (resp. $E_{\nex}(G)$).

\begin{definition}\label{D:spl-gr}
A quasi-stable graph $G$ is called \textbf{simple} if the graph obtained from $G$ by removing its exceptional vertices is connected. 
\end{definition}

 The \textbf{stabilization} of $G$, denote by $G^{\st}$, is the (connected) graph with 
 \begin{equation}\label{E:stab-gr}
 V(G^{\st}):=V_{\nex}(G)\quad \textrm{ as well as } \quad H(G^{\st}):=H_{\nex}(G)\quad \textrm{ and } \quad L(G^{\st}):=L(G),
 \end{equation}
 the root map $r_{G^{\st}}:H(G^{\st})\sqcup L(G^{\st})\to V(G^{\st})$ is the restriction of the root map $r_G$, the fixed-point free involution $i_{G^{\st}}$ on $H(G^{\st})$ is defined by setting $i_{G^{\st}}\big(i_G(h_v^1)\big):=i_G\big(h_v^2\big)$ for any $v\in V_{\exc}(G)$ and $i_{G^{\st}}(h):=i_G(h)$ if $h\neq i_G\big(h_v^1\big)$ or $i_G(h_v^2)$ for any $v\in V_{\exc}(G)$.
 One can easily check that the graph $G^{\st}$ becomes a stable graph of type  $(g,n)$ with respect to the marking of legs
 $$m^{\st}:\{1,\ldots, n\}\xlongrightarrow{m} L(G)=L(G^{\st}),$$ 
 and the vertex-weight function 
 $$h^{\st}:V(G^{\st})=V_{\nex}(G)\xrightarrow{h_{|V_{\nex}(G)}} \ZZ.$$
  As a notational advice, given $v\in V_{\nex}(G)$ (resp. $h\in H_{\nex}(G)$, resp. $l\in L(G)$), we will denote the corresponding element of $G^{\st}$ by $\ov v$ (resp. $\ov h$, resp. $\ov l$). The edges of $G^{\st}$ come with a natural partition $E(G^{\st}):=E_{\nex}(G^{\st})\sqcup E_{\exc}(G^{\st})$ into exceptional and non-exceptional ones, that  can be described by the following bijections
 \begin{equation}\label{E:edg-stab}
  \begin{aligned}
 E_{\nex}(G)& \xlongrightarrow{\cong} E_{\nex}(G^{\st}) \\
e=\big\{h,i_G(h)\big\} & \longmapsto \ov e:=\big\{\ov h,  i_{G^{\st}}(\ov h)=\ov{i_G(h)}\big\} \\
\end{aligned}
\quad \textrm{ and }\quad
  \begin{aligned}
 V_{\exc}(G) & \xlongrightarrow{\cong}  E_{\exc}(G^{\st}), \\
v & \longmapsto e_v:=\big\{\ov{{i}_G(h_v^1)}, \ov{{i}_G(h_v^2)}\big\}.
\end{aligned}
 \end{equation} 
We will also use this notation: 
 \begin{itemize}
 \item given a non-exceptional edge $e\in E_{\nex}(G^{\st})$ we will denote by $\wt e\in E_{\nex}(G)$ the unique non-exceptional edge of $G$ such that $\ov{\wt e}=e$;
\item   given an exceptional edge $e=e_v\in E_{\exc}(G^{\st})$  we will denote the corresponding (ordered) exceptional edges of $G$ by  $e^i:=e_v^i$ for $i=1,2$.
  \end{itemize}
Note that, by contracting exactly one among the edges   $e^1$ and $e^2$ for any $e\in E_{\exc}(G)$, we get a (non-unique) stabilization morphism $\sigma: G\to G^{\st}$.

 Stable (resp. quasi-stable) graphs of type $(g,n)$ with respect to morphisms of graphs form a category, that we will denote by $\SG_{g,n}$ (resp. $\QSG_{g,n}$). The stabilization procedure gives rise to a \textbf{stabilization functor}
\begin{equation}\label{E:st-gr}
\begin{aligned}
\st:\QSG_{g,n}  \longrightarrow \SG_{g,n},
\end{aligned}
\end{equation}
given by
\begin{equation*}
\begin{aligned}
G & \longmapsto G^{\st} \quad \textrm{ and }\quad
\big(\pi:G_1\to G_2\big) & \longmapsto \big(\pi^{\st}:G_1^{\st}\to G_2^{\st}\big).
\end{aligned}
\end{equation*}
The stabilization  $\pi^{\st}:G_1^{\st}\to G_2^{\st}$ of a morphism $\pi:G_1\to G_2$ is defined as follows: given an exceptional edge $e_v\in E_{\exc}(G_1^{\st})$, if both $e_v^1$ and $e_v^2$ are mapped by $\pi$ to a vertex $w\in V(G_2)$, then $\pi^{\st}(e_v)=\overline{w}$; if instead $\pi$ maps at least one of $e_v^1$ and $e_v^2$ to an edge of $G_2$, then it corresponds to a unique edge $e$ in $G_2^{\st}$ and we set $\pi^{\st}(e_v)=e$.
See Figure \ref{figure_stabilization} for a picture of the stabilization functor.

\begin{figure}[h]\begin{tikzpicture}
\draw (1,0) -- (0,2);
\fill (0.5,1) circle (0.2 em);
\fill (1,0) circle (0.2 em);
\fill (0,2) circle (0.2 em);

\draw (0,2) -- (-0.25,2.25);
\draw (0,2) -- (0,2.25);
\draw (0,2) -- (0.25,2.25);

\draw (1,0) -- (0.75,-0.25);
\draw (1,0) -- (1,-0.25);
\draw (1,0) -- (1.25,-0.25);

\draw (0,3)  node {$G$};

\draw (5,0) -- (4,2);
\fill (5,0) circle (0.2 em);
\fill (4,2) circle (0.2 em);

\draw (4,2) -- (3.75,2.25);
\draw (4,2) -- (4,2.25);
\draw (4,2) -- (4.25,2.25);

\draw (5,0) -- (4.75,-0.25);
\draw (5,0) -- (5,-0.25);
\draw (5,0) -- (5.25,-0.25);

\draw (4,3)  node {$G^{\st}$};

\draw [->] (1.5,1) -- (3.5,1);
\draw (2.5,1.5) node {$\st$};

\end{tikzpicture}\caption{The stabilization functor}\label{figure_stabilization}
\end{figure}
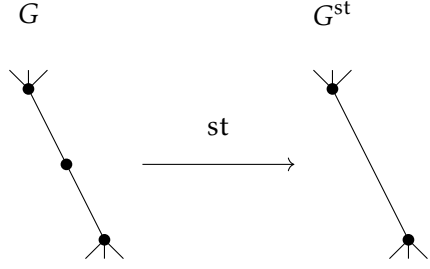

\subsection{Stability conditions} 

Let $G$ be a finite graph. A \emph{divisor} on $G$ is a finite formal sum 
\begin{equation*}
D=\sum_{v\in V(G)} a_v v
\end{equation*}
over the vertices of $G$ with $a_v\in\Z$. Divisors on $G$ form an abelian group, which we denote by $\Div(G)$. The \emph{degree} of a divisor $D=\sum a_v v$ is defined to be $\deg(D)=\sum_{v}a_v$. We write $\Div_d(G)$ for the divisors of degree $d$. The subset $\Div_0(G)$ is a subgroup of $\Div(G)$ and, for any $d\in\Z$, the subset $\Div_d(G)$ is naturally a torsor over $\Div_0(G)$. 

We now define what it means for a divisor on a quasi-stable graphs to be admissible and (semi)stable with respect to a (universal) stability condition. Recall from \cite{KP}:

\begin{definition}\label{D:stab-con}
 A \textbf{universal stability condition $\phi$} of type $(g,n)$ is an assignment of a function for any $G\in \SG_{g,n}$ (called the stability condition induced by $\phi$ on $G$)
$$\phi_G:V(G)\longrightarrow \R \: \text{ with } \: |\phi_G|:=\sum_{v\in V(G)} \phi_G(v)\in \Z$$
such that for any morphism $\pi:G_1\to G_2$ in $\SG_{g,n}$ we have that 
$$\phi_{G_2}(v)=\sum_{w\in \pi^{-1}(v)}\phi_{G_1}(w).$$ 
\end{definition}

We will often denote $\phi_G$ by $\phi$ if there is no danger of confusion. From the definitions, together with the fact that $\SG_{g,n}$ admits a final object (namely the graph with one vertex of genus $g$ and no edges), it follows that the integral number $|\phi_G|$ is independent from $G\in\SG_{g,n}$; it is therefore denoted by $|\phi|$ and called the \emph{degree} of $\phi$. 

\begin{remark}
The space of universal stability conditions of type $(g,n)$, denoted by $V_{g,n}$, is an abelian group with respect to the sum 
$$(\phi+\psi)_G(v):=\phi_G(v) +\psi_G(v) \text{ for any } G\in \SG_{g,n} \text{ and any }v\in V(G). $$
The subset of $V_{g,n}$ formed by the universal stability conditions of degree $d$ is denoted by $V_{g,n}^d$. 
Note that  $V_{g,n}^0$ is a subgroup of $V_{g,n}$ and $V_{g,n}^d$ is a torsor with respect to $V_{g,n}^0$. Moreover, $V_{g,n}^0$ is also a finite dimensional real vector space with 
respect to the scalar multiplication 
$$
 (\lambda\cdot \phi)_G(v):=\lambda\cdot \phi_G(v) \text{ for any } G\in \SG_{g,n} \text{ and any }v\in V(G). 
$$
Hence $V_{g,n}^d$ is a real affine space with respect to the vector space $V_{g,n}^0$. 
\end{remark}

Given $\phi\in V_{g,n}$, we can define a stability condition $\phi_G:V(G)\to \R$ on any graph $G$ that is quasi-stable of type $(g,n)$ by
\begin{equation}\label{E:stab-quasi}
\phi_G(v)=
\begin{cases}
0& \text{ if }v\in V_{\exc}(G),\\
\phi_{G^{\st}}(\ov v) & \text{ if }v\in V_{\nex}(G).  
\end{cases}
\end{equation}
Note that also for a quasi-stable graph $G$ we have that $|\phi_G|=|\phi|$.

\begin{definition}\label{D:admdiv}
Fix an hyperbolic pair $(g,n)$ and let $G$ be a quasi-stable graph of type $(g,n)$. 
\begin{enumerate}
\item \label{D:admdiv1} A divisor $D\in \Div(G)$ is called \textbf{admissible} if $D(v)=1$ for any $v\in V_{\exc}(G)$.
\item \label{D:admdiv2} Let $\phi\in V_{g,n}$ be a universal stability condition. 
\begin{itemize}
\item For a subset $S\subseteq V(G)$ we write $D(S):=\sum_{v\in S} D(v)$ as well as $\phi_G(S):=\sum_{v\in S} \phi_G(v)$. A divisor $D\in \Div(G)$ is called \textbf{$\phi$-semistable} if $D$ is admissible of degree $\deg D$ equal to $\vert\phi\vert$ and  the following inequalities
hold for any subset $S\subseteq V(G)$:
\begin{equation}\label{E:phistab}
\phi_G(S)-\frac{\big|E(S,S^c)\big|}{2}\leq D(S)\leq \phi_G(S)+\frac{\big|E(S,S^c)\big|}{2},
\end{equation}
where $E(S,S^c)$ is the set of edges joining a vertex in $S$ with a vertex in the complementary subset $S^c$. 
\item  A divisor $D\in \Div(G)$ is called \textbf{$\phi$-stable} if $D$ is $\phi$-semistable and both the inequalities \eqref{E:phistab} are strict unless $S$ or $S^c$ is a union (possibly empty) of exceptional vertices.
\end{itemize}
\item \label{D:admdiv3} A universal stability condition $\phi\in V_{g,n}$ is called \textbf{general}\footnote{This is called non-degenerate in \cite{KP}. We prefer to call it general according to the terminology used in \cite{MV12}, \cite{MRV0}, \cite{MRV1}, \cite{MRV2}, \cite{MSV}. Note that in the first two papers, the term non-degenerate is used for a slightly weaker condition.} if for any quasi-stable graph of type $(g,n)$ we have that for any subset $S\subseteq V(G)$ 
$$ \phi_G(S)+\frac{\big|E(S,S^c)\big|}{2}\in \ZZ \Rightarrow S \text{ or } S^c \text{ is a union of exceptional vertices.} $$ 

\end{enumerate}
\end{definition}

Some remarks on the above definition are in order.

\begin{remark}\label{R:stabdiv}
Let us keep the notation of the above Definition \ref{D:admdiv}. 
\begin{enumerate}[label={\small\textrm{(\roman*)}}]
\item \label{R:stabdiv1} The two inequalities in \eqref{E:phistab} for $S$ are equivalent to  the two analogous inequalities (but in reverse order)  for $S^c$. Hence it is enough to require one of the two inequalities in \eqref{E:phistab} for any $S\subseteq V(G)$. 
\item \label{R:stabdiv2} If $S$ (resp. $S^c$) is a union of exceptional vertices then the second (resp. the first) inequality in \eqref{E:phistab} is always an equality for any admissible divisor $D\in \Div(G)$ and for any universal stability condition $\phi$. Hence the definition of $\phi$-stability in \eqref{D:admdiv2} is sharp. 
\item \label{R:stabdiv2b}  If $S$ or $S^c$ is a union of exceptional vertices,  the quantity $\displaystyle \phi_G(S)+\frac{|E(S,S^c)|}{2}$ is always an integer for any universal stability condition $\phi$ (equal to $|S|$ or to $d+|S^c|$, respectively). Hence the definition of general universal stability condition in \eqref{D:admdiv3} is sharp.
\item \label{R:stabdiv3} It follows from the proof of \cite[Theorem~6.1(i), Proposition~7.3]{MV12} that if a quasi-stable graph $G$ of type $(g,n)$ admits a $\phi$-stable divisor for some $\phi\in V_{g,n}$  then $G$ has no separating exceptional vertices, \textit{i.e.}, $G$ is simple.
In particular, there are no exceptional vertices on bridges of $G^{\st}$ on the support of a $\phi$-stable divisor for a general stability condition.

\item \label{R:stabdiv4} It follows from the proofs of \cite[Theorem~6.1(iii) and Proposition~7.3]{MV12}  that $\phi\in V_{g,n}$ is general if and only if for any  quasi-stable graph $G$ of type $(g,n)$ every $\phi$-semistable divisor on $G$ is $\phi$-stable. 
This fact, combined  again with  \cite[Theorem~6.1(iii)]{MV12}, shows that our definition of general universal stability conditions coincides with the one of \cite[Definition~4.1]{KP}.
\end{enumerate}
\end{remark}


\subsection{The category $\QD_{g,n}$}
We now can define the category of quasi-stable graphs together with admissible divisors, and some variants of it, that will play a special role in what follows.

\begin{definition}[\textbf{The category $\QD_{g,n}$}]\label{Cat}
Fix an hyperbolic pair $(g,n)$. 
\begin{enumerate}
\item Let $\QD_{g,n}$ be the category whose objects are pairs $(G,D)$ consisting of a \emph{quasi-stable} graph $G$ of type $(g,n)$  and an \emph{admissible} divisor $D\in \Div(G)$, and whose morphisms $\pi:(G,D)\to (G', D')$ are the morphisms of the underlying graphs $\pi:G\to G'$ such that $\pi_*(D)=D'$, where $\pi_*(D)\in \Div(G')$ is such that 
$\displaystyle \pi_*(D)(v)=\sum_{w\in \pi^{-1}(v)}D(w), \forall v\in V(G')$.
\item We will consider the following full subcategories of $\QD_{g,n}$:
\begin{enumerate}[label={\small\textrm{(\roman*)}}]
\item $\QD_{g,n,d}$ is the full subcategory of $\QD_{g,n}$ whose objects are pairs $(G,D)\in \QD_{g,n}$ such that $\deg D=d$.
\item $\QD_{g,n}^{\spl}$ (resp. $\QD_{g,n, d}^{\spl}$) is the full subcategory of $\QD_{g,n}$ (resp. $\QD_{g,n,d}$) whose objects are pairs $(G,D)\in \QD$ such that $G$ is simple. 
\item For any  stability condition $\phi\in V_{g,n}$, 
$\QD_{g,n}(\phi)$ is the full subcategory of $\QD_{g,n}$ whose objects are pairs $(G,D)\in \QD_{g,n}$ such that $D$ is $\phi$-semistable. Note that, if $\phi$ is general, $\QD_{g,n}(\phi)$ is a full subcategory of $\QD_{g,n,d}^{\spl}$ where $|\phi|=d$.
\end{enumerate}
\end{enumerate}
\end{definition}


\begin{remark}\label{R:subcat}
It is easy to check that:
\begin{itemize}
\item $\QD_{g,n}^{\spl}$ and $\QD_{g,n}(\phi)$ are "under-closed" subcategories of $\QD_{g,n}$ in the following sense: if $\pi:(G,D)\to (G',D')$ is a morphism in $\QD_{g,n}$ and $(G,D)\in \QD_{g,n}^{\spl}$ (resp. $\QD_{g,n}(\phi)$), then also $(G',D')\in \QD_{g,n}^{\spl}$ (resp. $\QD_{g,n}(\phi)$).
\item  $\QD_{g,n,d}$ is a "morphism-closed"  subcategory of
  $\QD_{g,n}$ in the following sense: consider a morphism
  $\pi:(G,D)\to (G',D')$ in $\QD_{g,n}$.  If either $(G,D)$ or
  $(G',D')$  belongs to 
  $\QD_{g,n,d}$ then both of them belong to $\QD_{g,n,d}$. Similarly,
  $\QD_{g,n,d}^{\spl}$ is a morphism-closed  subcategory of
  $\QD_{g,n}^{\spl}$.
\end{itemize}
\end{remark}

There is a forgetful-stabilization functor 
\begin{equation}\label{E:funF}
F:\QD_{g,n}  \longrightarrow \QSG_{g,n}\xlongrightarrow{\st} \SG_{g,n}
\end{equation}
that sends an object $(G,D)\in \QD_{g,n}$ into the stabilization
$G^{\st}\in \SG_{g,n}$ of $G$ and a morphism \linebreak
$\pi:(G,D)\to (G',D')$ into the stabilization $\pi^{\st}:G^{\st}\to
G'^{\st}$ of $\pi:G\to G'$.  
The restrictions of $F$ to the full subcategories $\QD_{g,n,d}$,  $\QD_{g,n, (d)}^{\spl}$ and $\QD_{g,n}(\phi)$ will be denoted again by $F$.  

Throughout the text, and in order to shorten the statements, when we write $(d)$ in the context of any notation we mean that the sentence applies to the object in discussion either indicating $d$ or not. For instance, when we write $\QD_{g,n, (d)}^{\spl}$ we mean that it works for both $\QD_{g,n}^{\spl}$ and for $\QD_{g,n,d}^{\spl}$.

Given a stable graph $\ov G\in\SG_{g,n}$, we will now define  the essential fiber of $F:\QD_{g,n}  \to  \SG_{g,n}$ over $\ov G$ (and of its   
 subcategories $\QD_{g,n,d}$,  $\QD_{g,n, (d)}^{\spl}$ and $\QD_{g,n}(\phi)$). 

\begin{definition}[\textbf{The category $\QD_{\ov G}$}]\label{FibCat} 
For any $\ov G\in \SG_{g,n}$, let   $\QD_{\ov G}$ be the category such that 
\begin{itemize}
\item the objects of $\QD_{\ov G}$ are triples $(G,D,\rho)$ such that $(G,D)\in \QD_{g,n}$ and $\rho:\ov G\xrightarrow{\cong} G^{\st}$ is an  isomorphism in $\SG_{g,n}$;
\item the morphisms $\pi:(G,D,\rho)\to (G',D',\rho')$ are the morphisms $\pi:(G,D)\to (G',D')$ in $\QD_{g,n}$ such that $\pi^{\st}\circ \rho_1=\rho_2$. 
\end{itemize}
By imposing, in the above definition, that $(G,D)$ belongs to $\QD_{g,n,d}$,  $\QD_{g,n, (d)}^{\spl}$ or $\QD_{g,n}(\phi)$, we obtain, respectively, the full subcategories $\QD_{\ov G,d}$,  $\QD_{\ov G, (d)}^{\spl}$ or $\QD_{\ov G}(\phi)$ of $\QD_{\ov G}$.
\end{definition}

\begin{lemma}
The category $\QD_{\ov G}$ $($and hence its full subcategories $\QD_{\ov G,d}$,  $\QD_{\ov G, (d)}^{\spl}$ or $\QD_{\ov G}(\phi))$ does not have nontrivial automorphisms. 
\end{lemma}
\begin{proof}
Let $\pi:(G,D,\rho)\to (G,D,\rho)$ be an automorphism in $\QD_{\ov G}$. The equality   $\pi^{\st}\circ \rho=\rho$ forces $\pi^{\st}:G^{\st}\to G^{\st}$ to be the identity.  
By the definition of the stabilization graph $G^{\st}$ (see \eqref{E:stab-gr}), we get that the restriction of $\pi_V$ to $V_{\nex}(G)$ is the identity and that the restriction of $\pi_H$ to $H_{\nex}(G)$ is the identity. Since any exceptional-half edge of $G$, \textit{i.e.} an half-edge of the form $h_v^i$ for $i=1,2$ and $v\in V_{\exc}(G)$, is conjugate to a non-exceptional half-edge, \textit{i.e.} $i_G(h_v^i)$ is non-exceptional, and using that $\pi_H$ commutes with $i_G$, we deduce that $\pi_H=\id$. Finally, using that $\pi:X(G)\to X(G)$ commutes with the root map $r_G$ and that $\pi_H=\id$, we get for any exceptional vertex $v\in V_{\exc}(G)$ and any $i=1,2$:
$$\pi_V(v)=\pi_V(r_G(h_v^i))=r_G(\pi_H(h_v^i))=r_G(h_v^i)=v,$$
and we deduce that $\pi_V=\id$. Since $\pi$ is uniquely determined by $\pi_V$ and $\pi_H$, we conclude that $\pi$ is the identity. 
\end{proof}


For later use, we will also need a quotient of the above categories in
which two morphisms of (quasi-)stable graphs (or pairs formed by a
quasi-stable graph and an admissible divisor) are identified if they
induce the same map on the edge sets. These categories are important
because they will turn out to be isomorphic to the category of strata
of the corresponding toroidal stacks. We refer the reader to Figure
\ref{figure_automorphisms} for the two types of automorphisms that do
not permute the edges. 

\begin{figure}
  \centering
  aa
  \begin{tikzpicture}[baseline=0pt, main/.style = {draw, circle}]
    \node[main] (1) at (0,0) {$G$}; 
    \draw (1) to [out=135,in=235,looseness=10] (1);
    \draw[thick,<->] (-1.75,-0.5) .. controls (-2,0)
    and (-2,0) .. (-1.75,0.5);
  \end{tikzpicture}
  \hspace{.2\textwidth}
  \begin{tikzpicture}[baseline=0pt, main/.style = {draw, circle}] 
    \node[main] (1) at (0,0) {$G$}; 
    \node[main] (2) at (2,0) {$G$};
    \draw (1) to [out=75,in=105] (2);
    \draw (1) to [out=30,in=150] (2);
    \draw (1) to [out=-75,in=255] (2);
    \node (3) at (1,-0.25) {$\vdots$};
    \draw[thick,<->] (0.5,1.5) .. controls (1,1.75) and (1,1.75) .. (1.5,1.5);
  \end{tikzpicture}
  \caption{The two automorphism types of graphs that do not permute
    the edges: a loop and a ladder of bridges connecting two
    identical graphs.}
  \label{figure_automorphisms}
\end{figure}
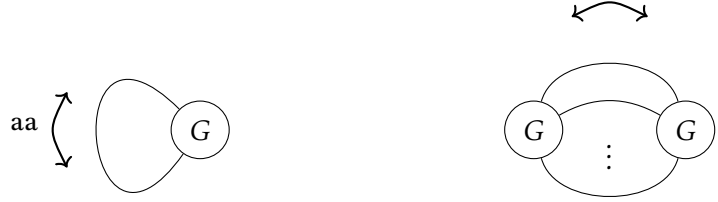

\begin{definition}[\textbf{The categories $\SG_{g,n}^E$, $\QSG_{g,n}^E$, $\QD^E_{g,n}$}]\label{CatE}
Fix an hyperbolic pair $(g,n)$. 
\begin{enumerate}[label={\small\textrm{(\roman*)}}]
\item Let $\QSG^E_{g,n}$ (resp. $\SG^E_{g,n}$)  be the category whose objects are quasi-stable (resp. stable) graphs of type $(g,n)$  and the morphisms from $G$ to $G'$ are the 
equivalence classes of morphisms $\pi:G\to G'$ in $\QSG_{g,n}$ (resp. in $\SG_{g,n}$) with respect to the equivalence relation 
$$\pi\sim \wt{\pi} \Leftrightarrow \pi_E^*=\wt{\pi}_E^*:E(G')\hookrightarrow E(G).$$
The equivalence class of a morphism $\pi:G\to G'$ in $\QSG_{g,n}$ (resp. in $\SG_{g,n}$) will be denoted by $[\pi]: G\to G'$.\footnote{In general, two equivalent morphisms of graphs map each pair of adjacent vertices to the same set. From this observation, one can check that given $\pi_1,\pi_2:G\to G'$, then $[\pi_1]=[\pi_2]$ if and only if there is an automorphism $\tau:G'\to G'$ with $[\tau]=[\rm id]$ such that $\pi_1=\tau\circ \pi_2$. We thank the referee for pointing out this clarification.}
\item Let $\QD^E_{g,n}$  be the category whose objects are pairs $(G,D)$ consisting of a quasi-stable graph $G$ of type $(g,n)$  and an admissible divisor $D\in \Div(G)$, and whose morphisms from $(G,D)$ to $(G', D')$ are the  equivalence classes of morphisms $\pi:(G,D)\to (G',D')$ in $\QD_{g,n}$ with respect to the equivalence relation 
$$\pi\sim \wt{\pi} \Leftrightarrow \pi_E^*=\wt{\pi}_E^*:E(G')\hookrightarrow E(G).$$
The equivalence class of a morphism $\pi:(G,D)\to (G',D')$ in $\QD_{g,n}$  will be denoted by $[\pi]: (G,D)\to (G',D')$. 

Similarly, we can define $\QD^E_{g,n,d}$, $\QD^{E,\spl}_{g,n, (d)}$, $\QD^E_{g,n}(\phi)$.
\end{enumerate}
\end{definition}

The  forgetful-stabilization functor \eqref{E:funF} induces the following diagram of functors
\begin{equation}\label{E:funFE}
\xymatrix{
F:\QD_{g,n}  \ar[r] \ar[d] &  \QSG_{g,n}\ar[r]^{\st} \ar[d] &  \SG_{g,n}\ar[d]\\
F^E:\QD^E_{g,n}  \ar[r]  &  \QSG^E_{g,n}\ar[r]^{\st^E} &  \SG^E_{g,n}\\
}
\end{equation}
where the vertical arrows are the essentially surjective and full (but non faithful) functors that are the identities on objects and send a morphism $\pi$ into $[\pi]$.   


Finally, the above categories induce some natural partially ordered sets (posets) that we now introduce.

\begin{definition}[\textbf{The posets $|\SG_{g,n}|$, $|\QSG_{g,n}|$, $|\QD_{g,n}|$}]\label{Pos}
Fix an hyperbolic pair $(g,n)$. 
\begin{enumerate}[label={\small\textrm{(\roman*)}}]
\item Let $\big|\QSG_{g,n}\big|$ (resp. $\big|\SG_{g,n}\big|$)  be the posets whose objects are isomorphism classes of quasi-stable (resp. stable) graphs of type $(g,n)$  and 
$$G\geq G' \Leftrightarrow \text{there exists a morphism } \pi:G\to G' \text{ in }\QSG_{g,n} (\text{resp. in} \SG_{g,n}).
$$
\item Let $\big|\QD_{g,n}\big|$  be the poset whose objects are isomorphism classes of objects of $\QD_{g,n}$ and
$$(G,D)\geq (G',D') \Leftrightarrow \text{there exists a morphism } \pi:(G,D)\to (G',D') \text{ in }\QD_{g,n}.
$$
Similarly, we can define the posets $\big|\QD_{g,n,d}\big|$, $\big|\QD^{\spl}_{g,n, (d)}\big|$ and $\big|\QD_{g,n}(\phi)\big|$.
\end{enumerate}
\end{definition}

In what follows we will often  pass, implicitly, to skeleton categories of each of the categories defined above, picking one object arbitrarily for each isomorphism class. In particular, we will identify the objects in each category with the elements in the associated poset.

\begin{remark}
Observe that the binary relations $\geq$ in the above definition are clearly reflexive and transitive, while they are antisymmetric because any morphism $\pi:G\to G'$ of ($n$-marked vertex-weighted) graphs is either an isomorphism  or is such that $\pi_E^*:E(G')\hookrightarrow E(G)$ is a proper inclusion. 

This also implies that the posets $|\SG_{g,n}|$, $|\QSG_{g,n}|$, $|\QD_{g,n, (d)}|$ and $|\QD_{g,n, (d)}^{\spl}|$ are graded with respect to the rank function given by the cardinality of the edge set. By following the lines of the proof of Proposition 4.11 and the first part of Theorem 4.15  in \cite{API}, one gets that $|\QD_{g,n}(\phi)|$ is also graded if $\phi$ is general (we do not include a proof here since it is not important for what follows).
The lengths of the above posets with respect to this rank function are equal to 
$$
\begin{aligned}
 l\big(|\SG_{g,n}|\big)&=3g-3+n, \\
 l\big(|\QSG_{g,n}|\big)&=l\big(|\QD_{g,n,(d)}|\big)=2(3g-3+n),\\
 l\big(|\QD^{\spl}_{g,n,(d)}|\big)&=l\big(|\QD_{g,n}(\phi)|\big)=4g-3+n \quad\text{if }\phi\:\text{is general.}
\end{aligned}
$$
\end{remark}
The  forgetful-stabilization functor \eqref{E:funF} induces the following diagram of posets
\begin{equation}\label{E:funFPos}
|F|:|\QD_{g,n}|  \longrightarrow  |\QSG_{g,n}| \xrightarrow{|\st|} |\SG_{g,n}|.
\end{equation}


\subsection{The universal tropical Jacobian as a cone stack}\label{S:conestackJ}


The aim of this subsection is to construct the universal tropical Jacobian as a cone stack, in the sense of \cite[Section~2.1]{CCUW}, endowed with a forgetful morphism to the cone stack 
$\Mt_{g,n}$ of tropical curves, constructed in \cite[Section~3]{CCUW}.

By slight abuse of notation, by a \emph{rational polyhedral cone} we will mean the data of a pair $(\sigma,N)$, consisting of a lattice $N$ and a full-dimensional, strictly convex rational polyhedral cone in $N_{\RR}$: an intersection $\sigma$ of finitely many half spaces in $N_\RR$ which contain no line and span all of $N_{\RR}$. The effect of this definition is that a rational polyhedral cone always comes with an integral structure, namely the intersection $N \cap \sigma$, and a dual monoid 

$$
S_{\sigma} := \Hom_{\textup{Mon}}(\sigma \cap N, \NN). 
$$
Conversely, the rational polyhedral cone is recovered from $S_\sigma$ as 
$$
\sigma = \Hom(S_\sigma, \mathbb{R}_{\ge 0})
\quad \textrm{ and } \quad
N = \Hom(S_{\sigma}^{\textup{gp}}, \ZZ).
$$
A morphism $(\sigma,N) \to (\sigma',N')$ is a homomorphism of lattices $N \to N'$ for which the induced homomorphism $N_\RR \to N'_\RR$ maps $\sigma$ into $\sigma'$. Equivalently, it corresponds to a homomorphism of monoids $S_{\sigma'} \to S_{\sigma}$. Denote the category of rational polyhedral cones by $\mathbf{RPC}$. We usually suppress the reference to $N$ and simply write $\sigma$ for a rational polyhedral cone. 

One may think of a rational polyhedral cone $\sigma$ as a combinatorial analogue of an affine scheme and of $S_\sigma$ as its ring of functions. Quite like a scheme is glued from affine schemes, we may think of a \emph{rational polyhedral cone complex} as a combinatorial object that is glued from rational polyhedral cones along their faces. Rational polyhedral cone complexes form a category that we denote by $\mathbf{RPCC}$. We refer the interested reader to \cite[Definition 2.1]{CCUW} for a precise definition of this notion. The category $\mathbf{RPCC}$ carries a natural Grothendieck topology that is generated by embeddings of rational polyhedral cones as faces, the \emph{face topology}.

Let us begin with the definition of the universal tropical Jacobian as a stack over the category $\RPCC$ of rational polyhedral  cone complexes endowed with the face topology, or equivalently as a category  fibered in groupoids over the category $\RPC$ of rational polyhedral cones (see \cite[Proposition~2.3]{CCUW}). We go through the effort of using this formalism, rather than, say, the formalism of generalized cone complexes, as this is the formalism that will allow us to compare the universal tropical Jacobian with its logarithmic analogue.



\begin{definition}\label{UnTrJac}
Fix an hyperbolic pair $(g,n)$. 
\begin{enumerate}[label={\small\textrm{(\roman*)}}]
\item Let $\Jt_{g,n}(\RPC)$ be the category fibered in groupoids over $\RPC$ such that:
\begin{itemize}
\item the objects are  triples $(\Gamma/\sigma, D)$, where $\Gamma/\sigma$ is a  \emph{quasi-stable tropical curve over $\sigma\in \RPC$} with $n$-markings and of genus $g$ (or simply of type $(g,n)$), \textit{i.e.} a pair consisting of  a  quasi-stable graph $\GG(\Gamma)$ of type $(g,n)$ (called the underlying graph) and a (generalized) metric $d:=d_{\Gamma}:E(\GG(\Gamma))\to S_{\sigma}\setminus\{0\}$, 
and $D$ is an admissible divisor on $\GG(\Gamma)$ (so that $(\GG(\Gamma),D)\in \QD_{g,n}$);


\item the morphisms are pairs $(f,\pi):( \Gamma/\sigma, D)\to (\Gamma'/\sigma',D')$ such that $f:\sigma\to \sigma'$ is a morphism in $\RPC$ and $\pi:(\GG(\Gamma'),D')\to (\GG(\Gamma),D)$ is a morphism in $\QD_{g,n}$
such that $\pi$ is compatible with $f$ as in \cite[Definition~3.2]{CCUW}, i.e. such that for any $e'\in E(\GG(\Gamma'))$ we have that
\begin{itemize}
\item $\pi$ contracts $e'$ if and only if $f^*(d_{\Gamma'}(e'))=0$;
\item if $\pi(e')=e\in E(\GG(\Gamma))$ then $f^*(d_{\Gamma'}(e'))=d_{\Gamma}(e)$.
\end{itemize}
\item the fibration $\Jt_{g,n}(\RPC)\to\RPC$ sends an object
  $(\Gamma/\sigma, D)$ into $\sigma$ and sends a morphism
  $(f,\pi):(\Gamma/\sigma, D)\to (\Gamma'/\sigma',D')$ into
  $f:\sigma\to \sigma'$.
\end{itemize}
\item 
Given $\sigma\in \RPC$, we will denote by $\Jt_{g,n}(\sigma)$ the essential fiber of $\Jt_{g,n}(\RPC)\to \RPC $ over $\sigma$, \textit{i.e.} the groupoid whose objects are $(\Gamma/\sigma, D)\in \Jt_{g,n}(\RPC)$ and whose morphisms are the ones of the form 
$\pi:=(\id_{\sigma}, \pi): (\Gamma/\sigma, D)\to (\Gamma'/\sigma,D')$
in $\Jt_{g,n}(\RPC)$. Note that any morphism $\pi:(\Gamma/\sigma,
D)\to (\Gamma'/\sigma,D')$ in $\Jt_{g,n}(\sigma)$ is such that
$\pi:\GG(\Gamma')\to \GG(\Gamma)$, the underlying morphism of graphs, is an isomorphism. 
\item The \textbf{universal tropical Jacobian} (over $\RPCC$), denoted by  $\Jt_{g,n}$, is the unique stack over $\RPCC$ (by \cite[Proposition~2.3]{CCUW}) whose restriction to $\RPC$ coincides with $\Jt_{g,n}(\RPC)$.
\end{enumerate}
Similarly, we can define the substacks $\Jt_{g,n,d}$, $\Jts_{g,n, (d)}$ and $\Jt_{g,n}(\phi)$ of $\Jt_{g,n}$ by taking only the objects $(\Gamma/\sigma,D)$ such that, respectively, $(\GG(\Gamma),D)\in \QD_{g,n,d}$, $\QD^{\spl}_{g,n,(d)}$ and $\QD_{g,n}(\phi)$. 
\end{definition}

The functor $\sigma \to S_\sigma$ gives an equivalence of categories between rational polyhedral cones and the category of sharp, fine and saturated monoids. We can thus also think of $\Jt_{g,n}$ (and similarly of $\Jt_{g,n,d}$, $\Jts_{g,n, (d)}$ and $\Jt_{g,n}(\phi)$) as a category fibered in groupoids over the category of sharp fine and saturated monoids. We may expand $\Jt_{g,n}$ to  a category fibered in groupoids over the category $\textbf{ShpMon}^{op}$ of sharp integral and saturated monoids, whose fiber over a sharp integral and saturated monoid $P$ is the groupoid of pairs $(\Gamma,D)$ consisting of a quasistable tropical curve $\Gamma$ of genus $g$ with $n$ marked legs, whose edges are metrized by the monoid $P$, and an admissible divisor $D$ on $\mathbb{G}(\Gamma)$. We refer the interested reader to \cite[Section 5.2]{CCUW} for more background on this procedure.

The universal tropical Jacobian comes equipped with a morphism towards the  stack $\Mt_{g,n}$  over $\RPCC$  of tropical curves, defined in \cite[Definition~3.3]{CCUW}. 

\begin{definition}\label{UnTrJac-Mg}
Fix an hyperbolic pair $(g,n)$.  Let 
\begin{equation}\label{Phitrop}
\Phi^{\mathrm{trop}}:\Jt_{g,n}\longrightarrow \Mt_{g,n}
\end{equation}
be the unique morphism of stacks over $\RPCC$ (called \emph{forgetful-stabilization morphism}) induced by the morphism $\Jt_{g,n}(\RPC)\to \Mt_{g,n}(\RPC)$  of categories fibered in groupoids over $\RPC$ that sends:
\begin{itemize}
\item  an object $(\Gamma/\sigma, D)\in \Jt_{g,n}(\RPC)$ to 
the \emph{stabilization} $(\Gamma/\sigma)^{\st}=(\Gamma^{\st}/\sigma)$ of $\Gamma/\sigma$, \textit{i.e.} the  stable tropical curve of type $(g,n)$ such that 
\begin{itemize}
\item the underlying graph is $\GG(\Gamma^{\st}):=\GG(\Gamma)^{\st}$,
\item the metric $d_{\Gamma^{\st}}:E(\GG(\Gamma)^{\st})\to S_{\sigma}$ is defined (using the notation below \eqref{E:edg-stab}) by 
$$d_{\Gamma^{\st}}(e)=
\begin{cases}
d_{\Gamma}(\wt e) & \text{ if } e\in E_{\nex}(\GG(\Gamma)^{\st}), \\
d_{\Gamma}(e^1)+d_{\Gamma}(e^2) & \text{ if } e\in E_{\exc}(\GG(\Gamma)^{\st}).
\end{cases}
$$
\end{itemize}
\item a morphism $(f,\pi):( \Gamma/\sigma, D)\to (\Gamma'/\sigma',D')$
  of $\Jt_{g,n}(\RPC)$ into the morphism $(f,\pi^{\st})$ of
  $\Mt_{g,n}(\RPC)$, where
  $(f,\pi^{\st}):\Gamma^{\st}/\sigma \to (\Gamma')^{\st}/\sigma'$ and
  $\pi^{\st}:\GG(\Gamma'^{\st})=\GG(\Gamma')^{\st}\to
  \GG(\Gamma)^{\st}=\GG(\Gamma^{\st})$ is the stabilization of $\pi$
  defined in \eqref{E:st-gr}.  
\end{itemize}

We will still denote by $\Phi^{\mathrm{trop}}$ the restriction of $\Phi^{\mathrm{trop}}$ to the substacks   $\Jt_{g,n,d}$, $\Jts_{g,n, (d)}$ or $\Jt_{g,n}(\phi)$ of $\Jt_{g,n}$. 
\end{definition}

We will now prove that the universal tropical Jacobian $\Jt_{g,n}$ is a cone stack, \textit{i.e.} a geometric stack over $\mathbf{RPCC}$ (see \cite[Definition 2.7]{CCUW}).  We will achieve this by describing it as a combinatorial cone stack in the sense of \cite[Definition~2.15]{CCUW}, and then use that cone stacks are equivalent to combinatorial cone stacks by \cite[Proposition~2.19]{CCUW}.
The analogous description for $\Mt_{g,n}$ is proved in \cite[Section~3.4]{CCUW} and it realizes $\Mt_{g,n}$ as the  cone stack associated to the combinatorial cone stack 
\begin{equation}\label{E:combMt}
\begin{aligned}
\Gamma_{\SG_{g,n}}: \SG_{g,n}^{\opp}& \longrightarrow \RPC^f\\
G & \longmapsto \RR_{\geq 0}^{E(G)},\\
  \big(\pi:G\to G' \big)& \longmapsto \big(\RR_{\geq 0}^{E(G')}\hookrightarrow \RR_{\geq 0}^{E(G)}\big)
\end{aligned}
\end{equation}
where the inclusion $\RR_{\geq 0}^{E(G')}\hookrightarrow \RR_{\geq 0}^{E(G)}$ is induced by the map $\pi_E^*:E(G')\hookrightarrow E(G)$.
Here $\RPC^f$ denotes the category of rational polyhedral cones with face inclusions as morphisms.

\begin{theorem}\label{TrJacStack}
Fix a hyperbolic pair $(g,n)$. 
\begin{enumerate}[label={\small\textrm{(\roman*)}}]
\item \label{TrJacStack1} $\Jt_{g,n}$ is the cone stack associated to the combinatorial cone stack 
\begin{equation}\label{E:combJt}
\begin{aligned}
\Gamma_{\QD_{g,n}}: \QD_{g,n}^{\opp} & \longrightarrow \RPC^f\\
(G,D) & \longmapsto \RR_{\geq 0}^{E(G)},\\
\big(\pi:(G,D)\to (G',D') \big)& \longmapsto \big(\RR_{\geq 0}^{E(G')}\hookrightarrow \RR_{\geq 0}^{E(G)}\big),
\end{aligned}
\end{equation}
where the inclusion $\RR_{\geq 0}^{E(G')}\hookrightarrow \RR_{\geq 0}^{E(G)}$ is induced by the map $\pi_E^*:E(G')\hookrightarrow E(G)$.
\item \label{TrJacStack2} The forgetful-stabilization morphism  $\Phi^{\trop}:\Jt_{g,n}\to \Mt_{g,n}$ coincides with the morphism of combinatorial cone stacks   from $(\Gamma_{\QD_{g,n}}:\QD_{g,n}^{\opp}\to \RPC^f)$ to $(\Gamma_{\SG_{g,n}}: \SG_{g,n}^{\opp}\to \RPC^f)$ induced by the functor $F:\QD_{g,n}\to \SG_{g,n}$ of \eqref{E:funF} together with the following collection of contravariant and surjective 
morphisms of cones $($for any $(G,D)\in \QD_{g,n})$
$$
\begin{aligned}
\phi_{(G,D)}:\RR_{\geq 0}^{E(G)} & \twoheadlongrightarrow \RR_{\geq 0}^{E(G^{\st})}\\
(x_e)_{e\in E(G)} & \longmapsto 
\left(y_f:=
\begin{cases}
x_{\wt f} & \text{if } f \in E_{\nex}(G^{\st})\\
x_{f^1}+x_{f^2} & \text{if } f\in E_{\exc}(G^{\st})
\end{cases}\right)_{f\in E(G^{\st})}.
\end{aligned}
$$
\end{enumerate}
\end{theorem}
The above Theorem remains true for the cone substacks $\Jt_{g,n,d}$, $\Jts_{g,n, (d)}$ or $\Jt_{g,n}(\phi)$ of $\Jt_{g,n}$ by replacing $\QD_{g,n}$, respectively, with its  full subcategories $\QD_{g,n,d}$,  $\QD_{g,n, (d)}^{\spl}$ or $\QD_{g,n}(\phi)$.

\begin{proof}
Let us prove part \ref{TrJacStack1}. Denote temporarily by $\Jc_{g,n}$ the cone stack associated to the combinatorial cone stack \eqref{E:combJt} and let 
$\Jc_{g,n}(\RPC)$ be its restriction to $\RPC$. 

According to \cite[Definition~2.15,  Proposition~ 2.19]{CCUW}, an object of $\Jc_{g,n}(\RPC)$ is a triple $(G,D,l:\sigma\to \RR_{\geq 0}^{E(G)})$, where $(G,D)\in \QD_{g,n}$ and $l:\sigma\to \RR_{\geq 0}^{E(G)}$ is a morphism in $\RPC$ whose image is not contained in any proper face of $\RR_{\geq 0}^{E(G)}$. By composing with the projections along the coordinates, the morphism $l$ is equivalent to the datum of a collection of surjective morphisms $\{l_e:\sigma\to \RR_{\geq 0}\}_{e\in E(G)}$ in $\RPC$. By passing to the toric monoids, this is equivalent to giving a collection of injective morphisms  of monoids $\{l_e^*:\NN\to S_{\sigma}\}_{e\in E(G)}$, which is indeed determined  by the collection of non-zero elements $\{l_e^*(1)\in S_{\sigma}\}_{e\in E(G)}$. Therefore, the triple $(G,D,l:\sigma\to \RR_{\geq 0}^{E(G)})$ gives rise to (and it is completely determined by)  an object $(\Gamma/\sigma, D)$ of $\Jt_{g,n}(\RPC)$ such that $\Gamma/\sigma$ is the $n$-marked genus $g$ quasi-stable tropical curve whose underlying graph is $\GG(\Gamma):=G$ and whose metric is given by $d_{\Gamma}(e):=l_e^*(1) \in S_{\sigma}$ for any $e\in E(G)$. 

Similarly, a morphism $(G,D,l:\sigma\to \RR_{\geq 0}^{E(G)})\to (G',D',l':\sigma'\to \RR_{\geq 0}^{E(G')})$ in $\Jc_{g,n}(\RPC)$ is the datum of a morphism of cones $f:\sigma\to \sigma'$
and a morphism $\pi:(G',D')\to (G,D)$ in $\QD_{g,n}$ such that the following diagram commutes
$$
\xymatrix{ 
\sigma \ar[r]^{l} \ar[d]_f & \RR_{\geq 0}^{E(G)} \ar@{^{(}->}[d]\\
\sigma' \ar[r]^{l'}  &  \RR_{\geq 0}^{E(G')} 
}
$$
where the left arrow is induced by the inclusion $\pi_E^*:E(G)\hookrightarrow E(G')$ that identifies the edges of $G$ with the edges of $G'$ that are not contracted by $\pi$. 
In terms of the objects $(\Gamma/\sigma, D)$ and $(\Gamma'/\sigma', D')$ of $\Jt_{g,n}(\RPC)$ associated to, respectively, $(G,D,l:\sigma\to \RR_{\geq 0}^{E(G)})$ and $(G',D',l':\sigma'\to \RR_{\geq 0}^{E(G')})$ as explained above, the commutativity of the above diagram means exactly that $\pi$ is compatible with $f$ in the sense of Definition \ref{UnTrJac}, 
and therefore we get a (uniquely defined) morphism $(f,\pi):(\Gamma/\sigma, D)\to (\Gamma'/\sigma', D')$ in $\Jt_{g,n}(\RPC)$. 

The above discussion shows that we have an isomorphism $\Jc_{g,n}(\RPC)\cong \Jt_{g,n}(\RPC)$ of categories fibered in groupoids over $\RPC$, which then gives rise to an isomorphism of stacks $\Jc_{g,n}\cong \Jt_{g,n}$ by \cite[Proposition~2.3]{CCUW}. 

Let us now prove part \ref{TrJacStack2}. We will temporarily denote by $\Mc_{g,n}$ the cone stack associated to the combinatorial cone stack   \eqref{E:combMt},
by $\Mc_{g,n}(\RPC)$ its restriction to $\RPC$ and by $\Phi^{\rm comb}:\Jc_{g,n}\to \Mc_{g,n}$ the morphism of cone stacks induced by the morphism of combinatorial cone stacks 
described in part \ref{TrJacStack2}.  By definition, the morphism $\Phi^{\rm comb}$ is given on objects by 
\begin{equation*}
\begin{split}
\Jc(\RPC)&\longrightarrow \Mc_{g,n}(\RPC)\\ \big(G,D,l\colon\sigma\to \RR_{\geq 0}^{E(G)}\big) &\longmapsto \big(G^{\st}, l^{\st}\colon\sigma\xrightarrow{l} \RR_{\geq 0}^{E(G)}\xrightarrow{\phi_{(G,D)}} \RR_{\geq 0}^{E(G^{\st})}\big).
\end{split}\end{equation*}
 In terms of the isomorphisms $\Jc_{g,n}\cong \Jt_{g,n}$ and
 $\Mc_{g,n}\cong \Mt_{g,n}$, this corresponds to sending an object
 $(\Gamma/\sigma,D)\in \Jt_{g,n}(\RPC)$ into $\Gamma^{\st}/\sigma\in
 \Mt_{g,n}(\RPC)$.  Similarly, a morphism  
\[
  (f,\pi): (G,D,l:\sigma\to \RR_{\geq 0}^{E(G)})\to
  (G',D',l':\sigma'\to \RR_{\geq 0}^{E(G')}) 
\]
of $\Jc_{g,n}(\RPC)$ is sent via $\Phi^{\rm comb}$ into the morphism  
\[
  (f,\pi^{\st}): (G^{\st},l^{\st}:\sigma\to \RR_{\geq
    0}^{E(G^{\st})})\to (G'^{\st},l'^{\st}:\sigma'\to \RR_{\geq
    0}^{E(G'^{\st})}) 
\]
of $\Mc_{g,n}(\RPC)$.
 In terms of the isomorphisms $\Jc_{g,n}\cong \Jt_{g,n}$ and $\Mc_{g,n}\cong \Mt_{g,n}$, this corresponds to sending a morphism $(f,\pi):(\Gamma/\sigma,D)\to (\Gamma'/\sigma',D')$ in $\Jt_{g,n}(\RPC)$ into $(f,\pi^{\st}):\Gamma^{\st}/\sigma\to \Gamma'^{\st}/\sigma'$ in $\Mt_{g,n}$. 
 
 The above discussion shows that, under the isomorphisms  $\Jc_{g,n}\cong \Jt_{g,n}$ and $\Mc_{g,n}\cong \Mt_{g,n}$, the two morphisms $\Phi^{\trop}$ and $\Phi^{\rm comb}$ coincide when restricted to $\RPC$, and hence they coincide everywhere by \cite[Proposition~2.3]{CCUW}. 
\end{proof}



\section{Compactified universal Jacobian}
\label{S:comp-univ}


The aim of this section is to introduce and study the compactified universal Jacobian over the moduli stack of stable curves. In particular, we give a combinatorial description of its toroidal stratification. We refer the reader to \cite{CC} for an alternative description of the toroidal stratification in degrees $g$ and $g-1$ (without marked points). 


\begin{definition}\label{UniCompJac}
Fix a hyperbolic pair $(g,n)$. 
\begin{enumerate}[label={\small\textrm{(\roman*)}}]
\item The \textbf{compactified universal Jacobian} (of type $(g,n)$)  is the algebraic stack $\Jb_{g,n}$ parametrizing pairs $(\cC\to S, \cL)$ consisting of a family $\cC\to S$ of  quasi-stable curves of type $(g,n)$, \textit{i.e.} $n$-pointed nodal projective and connected curves of arithmetic genus $g$ whose dual graph is a quasi-stable graph,  and an admissible line bundle $\cL$ on $\cC$, \textit{i.e.} a line bundle that has degree $1$ on any exceptional component of every geometric fiber of $\cC\to S$.


\item The \emph{universal Jacobian} (of type $(g,n)$) is the open and dense substack $\J_{g,n}\subset \Jb_{g,n}$ parametrizing objects $(\cC\to S, \cL)\in \Jb_{g,n}(S)$ such that $\cC\to S$ is a family of smooth curves. 
\end{enumerate}

\end{definition}

The decomposition of $\Jb_{g,n}$ into connected components is 
\begin{equation}\label{E:connJac}
\Jb_{g,n}=\bigsqcup_{d\in \ZZ} \Jb_{g,n,d}
\end{equation}
where $\Jb_{g,n,d}$ is the algebraic stack parametrizing pairs $(\cC\to S, \cL)\in \Jb_{g,n}(S)$ such that  $\cL$ has relative degree $d$ on $\cC\to S$.

We will be considering the following open substacks (for any universal stability condition $\phi\in V_{g,n}$):
\begin{equation}\label{E:opensub}
 \Jb_{g,n}^{\,\spl} \subset \Jb_{g,n}\supset \Jb_{g,n}(\phi) 
 \end{equation}
where $\Jb_{g,n}(\phi)$ is the open substack parametrizing pairs $(\cC\to S, \cL)\in \Jb_{g,n}(S)$ such that for any geometric point $s$ of $S$ the line bundle $\cL_{s}$ is \emph{$\phi$-semistable} on $\cC_{s}$ (\textit{i.e.} its multidegree $\un{\deg}(\cL_{s})$ is a $\phi$-semistable divisor on the dual graph $G(\cC_{s})$ of the curve $\cC_{s}$),  while $ \Jb_{g,n}^{\,\spl} $ is the open substack parametrizing pairs $(\cC\to S, \cL)\in \Jb_{g,n}(S)$ such that for any geometric point $s$ of $S$ the quasi-stable curve $\cC_{s}$ is \emph{simple}, \textit{i.e.} it remains connected when we remove its exceptional components, or equivalently its dual graph $G(\cC_{s})$ is simple. The stack $\Jb_{g,n}(\phi)$ is connected and it is contained in $\Jb_{g,n,|\phi|}$, while $\Jb_{g,n}^{\,\spl}$ admits the decomposition into connected components 
$$\Jb_{g,n}^{\,\spl}=\bigsqcup_{d\in \ZZ} \Jb_{g,n,d}^{\,\spl}:=\bigsqcup_{d\in \ZZ}  (\Jb_{g,n,d}\cap  \Jb_{g,n}^{\,\spl}).$$

The stack $\Jb_{g,n}$ comes equipped with a forgetful-stabilization morphism to $\Mb_{g,n}$
\begin{equation}\label{PhiAlg}
\Phi\colon \Jb_{g,n}\to \Mb_{g,n},
\end{equation}
that sends an object $(\cC\to S, \cL)\in \Jb_{g,n}(S)$ into the stabilization $(\cC^{\st}\to S)\in \Mb_{g,n}(S)$ of the family $\cC\to S$ of quasi-stable curves. 
If we denote by $\Mb_{g,n}^{\,\qs}$ the stack of quasi-stable curves of type $(g,n)$, then the forgetful morphism $\Phi$ factors as 
 \begin{equation}\label{PhiAlg2}
\Phi\colon \Jb_{g,n}\xlongrightarrow{\Phi^{\qs}} \Mb_{g,n}^{\,\qs}\xlongrightarrow{\st} \Mb_{g,n},
\end{equation}
where the forgetful morphism $\Phi^{\qs}$ sends  $(\cC\to S, \cL)\in \Jb_{g,n}(S)$ into $(\cC\to S)\in \Mb_{g,n}^{\,\qs}(S)$ and the stabilization morphism  $\st$  sends $(\cC\to S)\in \Mb_{g,n}^{\,\qs}(S)$ into $(\cC^{\st}\to S)\in \Mb_{g,n}(S)$. 

The algebraic group $\Gm$ injects functorially into the automorphism group scheme $\Aut_S(\cC/S, \cL)$ of every object $(\cC\to S, \cL)\in \Jb_{g,n}(S)$ as scalar multiplication on the line bundle $\cL$. Hence, we may form the $\Gm$-rigidifications 
\begin{equation}\label{E:Gm-rigid}
 \Jb_{g,n}^{\,\spl}\myfatslash \Gm \subset \Jb_{g,n}\myfatslash \Gm \supset \Jb_{g,n}(\phi)\myfatslash \Gm.
 \end{equation}
Note that the forgetful morphism $\Phi^{\qs}$ (and hence also the forgetful-stabilization morphism $\Phi$), factors through the $\Gm$-rigidification.  By a slight abuse of notation, we will denote by $\Phi: \Jb_{g,n}\myfatslash \Gm\to \Mb_{g,n}$ the induced morphism, as well as its restrictions to  the open substacks $\Jb_{g,n}^{\,\spl}\myfatslash \Gm$ and  $\Jb_{g,n}(\phi)\myfatslash \Gm$.

\begin{remark}\label{R:alternJ}
 It follows from \cite{EP} that $\Jb_{g,n}$ is isomorphic to the algebraic stack parametrizing pairs $(\cX\to S, \cI)$ consisting of a family $\cX\to S$ of $n$-marked genus $g$ stable curves and a coherent sheaf $\cI$ on $\cX$, flat over $S$,  whose geometric fibers are rank-$1$ torsion-free sheaves.
In this alternative description:
\begin{itemize}
\item the open substack $\Jb_{g,n}^{\,\spl}$ parametrizes objects $(\cX\to S, \cI)$ as above such that the geometric fibers of $\cI$ are simple sheaves;
\item for any $\phi\in V_{g,n}$, the open substack $\Jb_{g,n}(\phi)$ corresponds to the stack defined in  \cite[Definition~4.2]{KP} (using \cite[Theorem~6.1]{MV12}).
\item if $\phi\in V_{g,n}$ is general, then $\Jb_{g,n}(\phi)$ is isomorphic to the Esteves' universal compactification $\Jb_{g,n}^{\,\mathcal E, \,\mathrm{ss}}$ (see \cite{Est01}) for a suitable choice of a universal vector bundle $\mathcal E$ on the universal family over $\Mb_{g,n}$ (see \cite[Proposition~4.17]{Mel15}).
\item the morphism $\Phi:\Jb_{g,n}\to \Mb_{g,n}$ is the morphism that sends  objects $(\cX\to S, \cI)$ as above into $(\cX\to S)\in \Mb_{g,n}$. 
\end{itemize}
\end{remark}

In the following Proposition, we collect all the properties of the stack $\Jb_{g,n}$ and of the forgetful morphism $\Phi$, that will be used in the sequel.

\begin{proposition}\label{P:propJ}
\noindent 
\begin{enumerate}

\item \label{P:propJ0} 
  The $\Gm$-gerbe $\Jb_{g,n,d}\to \Jb_{g,n,d}\myfatslash \Gm$ is trivial
  if and only if either $n>0$ or $n=0$ and \linebreak $\gcd(d-g+1,2g-2)=1$. 

In particular, under the above numerical conditions, the universal line bundle over the universal family of $\Jb_{g,n,d}$ descends to the universal family on the  rigidification $ \Jb_{g,n,d}\myfatslash \Gm$. 
\item \label{P:propJ1} 
The algebraic stack $ \Jb_{g,n}^{\,\spl}\myfatslash \Gm$ is the DM-locus of $\Jb_{g,n}\myfatslash \Gm$.
\item \label{P:propJ2}
The morphism $\Phi:\Jb_{g,n}\to \Mb_{g,n}$ of \eqref{PhiAlg}, as well as its restriction to the open subsets $\Jb_{g,n}^{\,\spl}$ and $\Jb_{g,n}(\phi)$  $($for any $\phi \in V_{g,n})$, satisfies the existence part of the valuative criterion for properness.

In particular, the stacks $\Jb_{g,n}$ and  $\Jb_{g,n}^{\,\spl}$  satisfy the existence part of the valuative criterion, and $\Jb_{g,n}(\phi)$ $($for any $\phi \in V_{g,n})$ is universally closed. 
\item \label{P:propJ3}
For a universal stability condition $\phi\in V_{g,n}$, the following conditions are equivalent
\begin{enumerate}[label={\small\textrm{(\roman*)}}]
\item \label{phi-cond1} $\phi$ is general; 
\item  \label{phi-cond2} $\Phi: \Jb_{g,n}(\phi)\myfatslash \Gm\to \Mb_{g,n}$ is separated $($and hence proper$)$;
\item \label{phi-cond3}  $\Jb_{g,n}(\phi)\myfatslash \Gm$ is separated $($and hence proper$)$;
\item \label{phi-cond4} $\Jb_{g,n}(\phi)\subset  \Jb_{g,n}^{\,\spl}$.
\end{enumerate}
\item  \label{P:propJ4} The  morphism $\Phi^{\qs}: \Jb_{g,n}\to \Mb_{g,n}^{\,\qs}$ of \eqref{PhiAlg2} is smooth.
\item \label{P:propJ5} The stack $\Jb_{g,n}$ is smooth and the boundary $\partial \Jb_{g,n}:=\Jb_{g,n}\setminus \J_{g,n}=\Phi^{-1}(\Mb_{g,n}^{\,\qs}\setminus \M_{g,n})$ is a normal crossing divisor.
\end{enumerate}
\end{proposition}
\begin{proof}
Part \eqref{P:propJ0} follows from \cite[Theorem~6.4]{MV14} for $n=0$ and \cite[Proposition~3.2]{Mel19} for $n>0$.
Part \eqref{P:propJ1} follows from \cite[Lemma 2.11]{BFV}. Part \eqref{P:propJ2} follows from \cite[Theorem~32]{Est01} together with Remark \ref{R:alternJ}. 
The proof of part \eqref{P:propJ3} follows by the following cycle of implications:

\ref{phi-cond1} $\Rightarrow$ \ref{phi-cond2} follows from \cite[Cor. 4.4]{KP} and Remark \ref{R:alternJ}.

\ref{phi-cond2} $\Rightarrow$ \ref{phi-cond3} follows from the fact that $\Mb_{g,n}$ is proper.

\ref{phi-cond3} $\Rightarrow$ \ref{phi-cond4} follows from \cite[Lemma 2.11]{BFV} using that a separated algebraic stack has finite stabilizers at every geometric point. 

\ref{phi-cond4} $\Rightarrow$ \ref{phi-cond1} follows from \cite[Proposition~7.3]{MV12} and Remark \ref{R:alternJ}.

Part \eqref{P:propJ4} follows since the obstruction to deform a line
bundle on a quasi-stable curve $C$  lies on $H^2(C,\calO_C)$, which is
zero because $C$ has dimension one. 

Part \eqref{P:propJ5} follows from part \eqref{P:propJ4} and the
well-known fact that $\Mb_{g,n}^{\,\qs}$ is smooth and that 
$\partial \Mb_{g,n}^{\,\qs}$ defined as 
$\Mb_{g,n}^{\,\qs}\setminus \M_{g,n}$ is a normal crossing
divisor. Alternatively, one could use the explicit description of the
completed local rings of $\Jb_{g,n}$ given in \cite[Section~2.3]{BFV} or
\cite[Section~3]{CMKV2}
\end{proof}

\subsection{Toroidal stratification of the compactified universal Jacobian}\label{S:toro-J}

The pair $(\J_{g,n}\subset \Jb_{g,n})$ is a toroidal embeddings of algebraic stacks since $\Jb_{g,n}$ is smooth and $\Jb_{g,n}\setminus \J_{g,n}$ is a normal crossing divisor by Proposition \ref{P:propJ} \eqref{P:propJ5}. 
We now describe the toroidal stratification of $(\J_{g,n}\subset \Jb_{g,n})$. 

\begin{proposition}\label{TorStrata}
\noindent 
\begin{enumerate}[label={\small\textrm{(\roman*)}}]
\item \label{TorStrata1} The toroidal stratification of $(\J_{g,n}\subset \Jb_{g,n})$ is given by 
\begin{equation}\label{E:strataJ}
\Jb_{g,n}=\bigsqcup_{(G,D)\in \QD_{g,n}}  \J_{(G,D)},
\end{equation}
where $\J_{(G,D)}$ is the locally closed $($reduced$)$ substack of $\Jb_{g,n}$ whose $k$-points are $(C,L)\in \Jb_{g,n}(k)$ such that $(G(C),\un{\deg}(L))\cong (G,D)$. 
In particular, each stratum $\J_{(G,D)}$ is smooth, irreducible and of finite type over the base field $k$.
\item \label{TorStrata2} The poset of strata $\{\J_{G,D}\}_{(G,D)\in \QD_{g,n}}$ is anti-isomorphic to the poset $|\QD_{g,n}|$ of Definition \ref{Pos}, \textit{i.e.} 
$$\J_{(G,D)}\subset \ov{\J}_{(G',D')} \Leftrightarrow (G,D)\geq  (G',D') \text{ in } |\QD_{g,n}|.
$$
\item \label{TorStrata3} The forgetful-stabilization morphism $\Phi:\Jb_{g,n}\to \Mb_{g,n}$ of \eqref{PhiAlg} is toroidal and, for any $\ov G\in \SG_{g,n}$, we have that 
\begin{equation}\label{E:pb-stra}
\Phi^{-1}(\M_{\ov G})= \bigsqcup_{\stackrel{(G,D)\in \QD_{g,n} : }{G^{\st} \cong \ov G}} \J_{(G,D)},
\end{equation}
where $\M_{\ov G}$ is the locally closed (reduced) substack of $\Mb_{g,n}$ whose $k$-points consist of curves $C\in \Mb_{g,n}(k)$ such that $G(C)\cong \ov G$. 
\end{enumerate}
\end{proposition}
Parts \ref{TorStrata1} and \ref{TorStrata2} have been proved for $\Jb_{g,n}(\phi)$, in the special case $n=1$ and for specific choices of $\phi\in V_{g,n}$, by Abreu-Pacini in \cite[Proposition~6.4]{API}.
\begin{proof}
Let us first prove part (\ref{TorStrata1}). Since the toroidal embedding $(\J_{g,n}\subset \Jb_{g,n})$ is the pull-back of the toroidal embedding  $(\M_{g,n}\subset \Mb_{g,n}^{\,\qs})$ via the smooth morphism $\Phi^{\qs}:\Jb_{g,n}\to \Mb_{g,n}^{\,\qs}$ (see Proposition \ref{P:propJ}), the toroidal stratification of  $(\J_{g,n}\subset \Jb_{g,n})$ is formed by the connected components of the pull-backs of the strata that form the toroidal stratification of $(\M_{g,n}\subset \Mb_{g,n}^{\,\qs})$. It is well-known that  the toroidal stratification of  $(\M_{g,n}\subset \Mb_{g,n}^{\,\qs})$  is given by 
$$
\Mb_{g,n}^{\,\qs}=\bigsqcup_{G\in \QSG_{g,n}}  \M_{G},
$$
where $\M_{G}$ is the locally closed (reduced) substack of $\Mb_{g,n}^{\,\qs}$ whose $k$-points are $C\in \Mb_{g,n}^{\,\qs}(k)$ such that $G(C)\cong G$. It remains to observe that we have a disjoint union 
$$
(\Phi^{\qs})^{-1}(\M_G)=\bigsqcup_{(G,D)\in \QD_{g,n}} \J_{(G,D)},
$$
and that each locally closed subset $ \J_{(G,D)}$ is irreducible since it is endowed with a surjective smooth morphism $\Phi^{\qs}_{|\J_{(G,D)}}: \J_{(G,D)}\to \M_G$ whose geometric fibers 
$$
(\Phi^{\qs}_{|\J_{(G,D)}})^{-1}(C)=\Pic^D(C):=\{L\in \Pic(C)\: : \underline{\deg} L=D\} \text{ for any } C\in \Mb_{g,n}^{\,\qs}(k)
$$  
are irreducible.

Part \ref{TorStrata2} follows from \cite[Proposition~3.4.1, Proposition~3.4.2]{CC}. 

Let us finally prove part \ref{TorStrata3}. Since the toroidal embedding $(\J_{g,n}\subset \Jb_{g,n})$ is the pull-back of the toroidal embedding  $(\M_{g,n}\subset \Mb_{g,n}^{\,\qs})$ via the smooth morphism $\Phi^{\qs}:\Jb_{g,n}\to \Mb_{g,n}^{\,\qs}$ (see Proposition~\ref{P:propJ}), the fact that $\Phi$ is a toroidal morphism follows from the well-known fact that the stabilization morphism $\st: (\M_{g,n}\subset \Mb_{g,n}^{\,\qs})\to (\M_{g,n}\subset \Mb_{g,n})$ is a toroidal morphism. The equality in \eqref{E:pb-stra} follows from the definition of the morphism $\Phi$.
\end{proof}


\subsection{The category of strata}\label{S:catstraJ}

We now want to show that the (anti-)isomorphism of posets in Proposition \ref{TorStrata} is induced by an (anti-)equivalence of categories. In order to do that, we have to show that the poset of strata $\{\J_{(G,D)}\}$ is induced by a category, namely the category of strata of $(\J_{g,n}\subset \Jb_{g,n})$, that we are now going to define in a more general setting.

Assume  that we have a toroidal embedding of Artin stacks $(\cU\subset \cX)$, locally of finite type over $k$. 
Consider the lisse-\'etale sheaves $D_{\cX}$ and $E_{\cX}$ over $\cX$ such that, for every smooth morphism $V\to \cX$ with $V$ a scheme, $D_{\cX}(V)$ (resp. $E_{\cX}(V)$) is the group of Cartier divisors on $V$ (resp. the submonoid of effective Cartier divisors on $V$) that are supported on $V\setminus U$ where $U:=V\times_{\cX} \cU$. 
Consider the strata $\{W_i\}$ of the toroidal embedding $\cU\subset \cX$ (in particular each $W_i$ is irreducible) and pick a geometric generic point $w_i$ in each stratum $W_i$.
The stalk $D_{\cX,w_i}$  is a finitely generated free abelian group and the stalk $E_{\cX,w_i}$ is a sharp, saturated, and finitely generated submonoid that generates $D_{\cX,w_i}$  as a group.

\begin{definition}
The \textbf{category of strata} of the toroidal embedding $(\cU\subset \cX)$, denoted by $\Str(\cU\subset \cX)$ or simply by $\Str(\cX)$, is the category whose objects are the toroidal strata $\{W_i\}$ and whose morphisms $W_i\to W_j$ are generated by the following two classes of morphisms:
\begin{itemize}
\item the \'etale specializations $w_i\rightsquigarrow w_j$  (see \cite[Appendix A]{CCUW}) whenever $W_i\neq W_j$; 
\item the image $H_{W_i}$ (called the \emph{monodromy group} of the stratum $W_i$) of the natural monodromy representation $\pi_1^{\et}(W_i,w_i)\to \Aut_{\mon}(E_{\cX,w_i})$ whenever $W_i=W_j$.
\end{itemize}
\end{definition}

The category $\Str(\cX)$ comes also equipped with the structure of a category fibered in groupoids over the category $\RPC^f$ of rational polyhedral cones with face inclusions as morphisms. 
Namely, to any stratum $W_i\in \Str(\cX)$, we associate the rational polyhedral cone  
$$\sigma(W_i):=\Hom_{\mon}(E_{\cX,w_i}, \R_{\geq 0}) \subset \Hom_{\mathrm{groups}}(D_{\cX,w_i}, \R).$$
Moreover, any morphism  $W_i\rightarrow W_j$ in $\Str(\cX)$ induces a surjective monoid homomorphism \linebreak$E_{\cX,w_j}\twoheadrightarrow E_{\cX,w_i}$ (and it is completely determined by it), and hence it induces a face morphism $\sigma(W_i)\hookrightarrow \sigma(W_j)$ of rational polyhedral cones. In this way we get a functor 
\begin{equation}\label{E:Gamma-Str}
\begin{aligned}
\Gamma_{\Str(\cX)}: \Str(\cX) & \longrightarrow \RPC^f \\
W_i & \longmapsto \sigma(W_i), \\
W_i\rightarrow W_j & \longmapsto \sigma(W_i)\hookrightarrow \sigma(W_j),
\end{aligned}
\end{equation}
which is easily checked to be a category fibered in groupoids.

In the following Proposition \ref{P:CatStrata}, we are going to relate $\Str(\Jb_{g,n})$ with the category $\QD_{g,n}^E$ introduced in Definition \ref{CatE}. Observe also that the category fibered in groupoids $\Gamma_{\QD_{g,n}}: \QD_{g,n}^{\opp}\to \RPC^f$ of \eqref{E:combJt} factors through a category fibered in groupoids 
\begin{equation}\label{E:Gamma-QDE}
\begin{aligned}
\Gamma_{\QD_{g,n}^E}: (\QD_{g,n}^E)^{\opp}& \longrightarrow \RPC^f\\
(G,D)& \longmapsto \RR_{\geq 0}^{E(G)},\\
[\pi]:(G,D)\to (G',D') & \longmapsto \RR_{\geq 0}^{E(G')}\hookrightarrow \RR_{\geq 0}^{E(G)},
\end{aligned}
\end{equation}
where the face inclusion $\RR_{\geq 0}^{E(G')}\hookrightarrow \RR_{\geq 0}^{E(G)}$ is induced by the injection $\pi_E^*:E(G')\hookrightarrow E(G)$.

The following Proposition \ref{P:CatStrata} is crucial for our understanding of both the analytic and the logarithmic tropicalization map.

\begin{proposition}\label{P:CatStrata}
Let  $\Str(\Jb_{g,n})$ be the category of strata of the toroidal embedding $(\J_{g,n}\subset \Jb_{g,n})$.
There is an equivalence of categories 
\begin{equation}\label{E:equivE}
\begin{aligned}
\calE: (\QD^E_{g,n})^{\opp} & \longrightarrow \Str(\Jb_{g,n})\\
(G,D) & \longmapsto \J_{(G,D)},
\end{aligned}
\end{equation}
such that $\Gamma_{\QD_{g,n}^E}=\Gamma_{\Str(\Jb_{g,n})}\circ \calE$.
\end{proposition}

Before giving a proof, we will need the following Lemma where we describe a rigidification of each stratum $\J_{(G,D)}$ of $\Jb_{g,n}$. For a similar description of the strata of $\Mb_{g,n}$ (and also of their closure), see \cite[Chapter~XII, Proposition~(10.11)]{GAC2} (or \cite[Proposition~4.6]{EF} for a more algebraic proof).

Let $(G,D)\in \QD_{g,n}$. Consider the category fibered in groupoids $\wt\J_{(G,D)}$ over the category of schemes, whose fiber over a scheme $S$ consists of the groupoid of objects $(\cC\to S, \cL)\in \Jb_{g,n}(S)$ together with isomorphisms $\phi_{s}:(G,D)\xrightarrow{\cong} (G(\cC_{s}),\un \deg\cL_{s})$ in $\QD_{g,n}$ for every geometric point $s\to S$, that are compatible with \'etale specializations in the following sense: for any \'etale specialization $f\colon  t\rightsquigarrow s$ of geometric points, with induced morphism  $\pi_f:(G(\cC_{s}),\un \deg\cL_s)\to (G(\cC_t),\un \deg\cL_t)$ in $\QD_{g,n}$, we require that  $\phi_t=\pi_f\circ \phi_s$. 

The group $\Aut(G,D)$ (considered as a constant group scheme over $k$) acts on $\wt\J_{(G,D)}$ by precomposition with the isomorphisms $\phi_{s}$.

\begin{lemma}\label{L:rigid-str}
We have that:
\begin{enumerate}[label={\small\textrm{(\roman*)}}]
\item \label{L:rigid-str1} The category fibered in groupoids $\wt\J_{(G,D)}$ is a smooth and irreducible algebraic stack of finite type over the base field $k$.
\item \label{L:rigid-str2} The natural forgetful morphism $\Phi_{(G,D)}: \wt\J_{(G,D)}\to \J_{(G,D)}$ is a representable, surjective, finite and \'etale, $\Aut(G,D)$-Galois morphism. Hence, it factors through an  isomorphism 
$$\ov\Phi_{(G,D)}:\big[\wt\J_{(G,D)}/\Aut(G,D)\big]\longrightarrow \J_{(G,D)}.$$
\end{enumerate}
\end{lemma}
A slightly different presentation of the strata of $\Jb_{g,n}(\phi)$, in the special case $n=1$ and for specific choices of $\phi$, is given in \cite[Proposition~6.1]{API}.
\begin{proof}
Let us first prove part \ref{L:rigid-str1}. 
Consider the category fibered in groupoids $\wt\M_{G}$ over the category of schemes, whose fiber over a scheme $S$ consists of families of quasi-stable curves $(\cC\to S) \in \Mb^{\,\qs}_{g,n}(S)$ together with an isomorphism $\phi_{s}:G\xrightarrow{\cong} G(\cC_{s})$ in $\QSG_{g,n}$ for every geometric point $s\to S$, that are compatible with \'etale specializations in the following sense: for any \'etale specialization $f\colon  t\rightsquigarrow s$ of geometric points, with induced morphism of graphs $\pi_f:G(\cC_{s})\to G(\cC_t)$, we require that  $\phi_t=\pi_f\circ \phi_s$. 

We claim that  we have an isomorphism of categories fibered in groupoids 
\begin{equation}\label{E:MGtilde}
\wt \M_G\cong \prod_{v\in V(G)} \M_{h(v), \val(v)},
\end{equation}
which implies that $\wt \M_G$ is a smooth and irreducible algebraic stack of finite type over $k$. 

Indeed, we have a functor $\wt \M_G\to  \prod_{v\in V(G)} \M_{h(v), \val(v)}$ that sends  $(\cC\to S, \{\phi_{s}\}_s)\in\wt\M_{G}(S)$ into the relative normalization $(\cC^{\nu}\to S)$ which is seen as an element of $\prod_{v\in V(G)} \M_{h(v), \val(v)}(S)$ via the compatible isomorphisms $\{\phi_s\}_{s\in S}$. Conversely, we have a functor  $\prod_{v\in V(G)} \M_{h(v), \val(v)}\to \wt M_G$ that sends 
$\cD\to S\in \prod_{v\in V(G)} \M_{h(v), \val(v)}(S)$ into the family of quasi-stable curves $\cC\to S$ of type $(g,n)$ that is obtained from $\cD\to S$ by gluing the pairs of marked sections of  $\cD\to S$ that correspond to the pairs of half-edges of $G$ that are exchanged by $i_G$, together with the family of compatible  isomorphisms $\{\phi_s:G \xrightarrow{\cong}G(\cC_s)\}_s$ that is induced by the above gluing procedure. Clearly the two functors are inverse of each other and they define the isomorphism \eqref{E:MGtilde}.  

\vspace{0.1cm}

The algebraic stack $\wt \M_G$ comes with a universal family $\pi_G:\wt \cC_G\to \wt \M_G$ of quasi-stable curves of type $(g,n)$ together with compatible  isomorphisms $\{\phi_s:G\xrightarrow{\cong}G(\wt \cC_{G,s})\}_s$, where $s$ runs over all geometric points of $\wt\M_G$, and $\wt \cC_{G,s}$ is the geometric fiber of $\pi_G$ over the  geometric point $s\to \wt\M_G$. 
The category fibered in groupoids $\wt\J_{(G,D)}$, endowed  with its natural forgetful morphism $\wt\Psi_{(G,D)}: \wt\J_{(G,D)}\to \wt\M_G$, can be naturally identified with the connected component of the Picard stack of the family $\pi_G:\wt \cC_G\to \wt \M_G$ that parametrizes line bundles $\cL$ on $\wt\cC_G$ whose multidegree satisfies $\un \deg(\cL_s)=\phi_s(D)$ for any geometric point $s\to\wt\M_G$, where $\cL_s$ is the restriction of $\cL$ to the geometric fiber $\wt \cC_{G,s}$. Hence, we deduce that $\wt\J_{(G,D)}$  is a smooth and irreducible algebraic stack of finite type over $k$.



We now prove part \ref{L:rigid-str2}. We begin by establishing several properties of the forgetful morphism $\Phi_{(G,D)}:\wt\J_{(G,D)}\to \J_{(G,D)}$. 
\begin{enumerate}[label={\textrm{(\alph*)}}]
\item \label{prop-a} $\Phi_{(G,D)}$ is surjective and the fiber over a geometric point $\Spec K \to \J_{(G,D)}$ is a torsor under  the constant finite group scheme $\Aut(G,D)_K$. 

Indeed, if $(C,L)\in \J_{(G,D)}(K)$ then the fiber of $\Phi_{(G,D)}$ over $(C,L)$ is the set of all the isomorphisms $(G,D)\to (G(C),\un \deg L)$, which is non-empty because of our assumption that 
$(C,L)\in \J_{(G,D)}(K)$ and it is clearly a torsor for $\Aut(G,D)$. 

\item  \label{prop-b}  $\Phi_{(G,D)}$ is representable, finite and unramified. 

In order to prove this, consider the commutative diagram 
\begin{equation}\label{E:wtJ-wtM}
\xymatrix{
\wt\J_{(G,D)}\ar@/^/[rrrd]^{\wt\Psi_{(G,D)}}\ar@/_/[ddr]_{\Phi_{(G,D)}} \ar[dr]^{F_{(G,D)}}& & & \\
& \J_{(G,D)}\times_{\M_G}\wt\M_G\ar[rr]^{\ov \Psi_{(G,D)}}\ar[d]^{\ov\Phi_G}\ar@{}[rrd]|\square& & \wt\M_G \ar[d]^{\Phi_G}\\
& \J_{(G,D)} \ar[rr]_{\Psi_{(G,D)}}& & \M_G
}
\end{equation}
where $\M_G$ is the locally closed (reduced) substack of $\Mb_{g,n}^{\,\qs}$ whose $k$-points are $C\in \Mb_{g,n}^{\,\qs}(k)$ such that $G(C)\cong G$, the morphism $\Phi_G$ forgets the isomorphisms $\{\phi_s\}$, the morphism $\Psi_{(G,D)}$ forgets the line bundle, the square is cartesian  and the morphism $F_{(G,D)}$ is induced by the universal property of the fibered product. 
Now, property \ref{prop-b} will follow from the following two properties:
\begin{enumerate}[label={\small\textrm{(b\arabic*)}}]
\item $\ov\Phi_{G}$ is representable, finite and unramified. 

Indeed, it is enough to show that $\Phi_G$ is representable, finite and unramified. It follows from \eqref{E:MGtilde} (and its proof) that $\Phi_G$ is a composition of clutching morphisms which are known to be representable, finite and unramified by \cite[Cor. 3.9]{Knu2}.
 
\item $F_{(G,D)}$ is a representable closed  embedding. 

Indeed, the fibered product $\J_{(G,D)}\times_{\M_G}\wt\M_G$ is the algebraic stack whose fiber over a scheme $S$ consists of $(\cC\to S, \cL, \{\phi_{s}:G\xrightarrow{\cong} G(\cC_{s})\}_s)$, where 
$(\cC\to S, \{\phi_{s}:G\xrightarrow{\cong} G(\cC_{s})\}_s)\in \wt\M_G$ and $\cL$ is an admissible line bundle on $\cC\to S$ such that $(G,D)\cong (G(\cC_s),\un\deg\cL_s)$. 
This implies that $\ov\Psi_{(G,D)}:\J_{(G,D)}\times_{\M_G}\wt\M_G\to \wt\M_G$ is the union of the connected components of the Picard stack of the family $\pi_G:\wt \cC_G\to \wt \M_G$ that parametrize line bundles $\cL$ on $\wt\cC_G$ whose multidegree is such that there exists an automorphism $\phi$ of $G$ such that  $(\phi_s)^{-1}(\un \deg(\cL_s))=\phi_*(D)$
for every geometric point $s\to\wt\M_G$. By what is proved in part \ref{L:rigid-str1}, it follows that $F_{(G,D)}$ identifies $\wt\J_{(G,D)}$ as one of the connected components of $\J_{(G,D)}\times_{\M_G}\wt\M_G\to \wt\M_G$, namely the one corresponding to $\phi=\id$, and this implies that $F_{(G,D)}$ is a representable closed (and also open) embedding. 

\end{enumerate}

\item  \label{prop-c}  $\Phi_{(G,D)}$ is flat. 

This follows by applying the flatness criterion \cite[Theorem~23.1]{Mat} to the representable morphism $\Phi_{(G,D)}$ and using that $\Phi_{(G,D)}$ is  surjective and finite (by \ref{prop-a} and \ref{prop-b}) from the smooth and irreducible (by part \ref{L:rigid-str1}) algebraic stack $\wt\J_{(G,D)}$ to the smooth and irreducible (by Proposition~\ref{TorStrata}\ref{TorStrata1}) algebraic stack $\J_{(G,D)}$.

\end{enumerate}

By using the above properties of the morphism $\Phi_{(G,D)}$, we can now conclude the proof of part \ref{L:rigid-str2}. Properties \ref{prop-b} and \ref{prop-c} and the first part of property \ref{prop-a} imply that $\Phi_{(G,D)}$ is a representable, surjective, finite and \'etale morphism. The second part of property \ref{prop-a} implies that $\Phi_{(G,D)}$ is a torsor under the constant finite group scheme $\Aut(G,D)_k$, which implies that $\Phi_{(G,D)}$ is a $\Aut(G,D)$-Galois morphism. 
\end{proof}

\begin{proof}[Proof of Proposition \ref{P:CatStrata}]
First of all, we need a local description of the boundary $\partial \Jb_{g,n}\subset \Jb_{g,n}$ around a geometric point $(C,L)\in \Jb_{g,n}(K)$. 
The morphism $\Phi^{\qs}:\Jb_{g,n}\to \Mb_{g,n}^{\,\qs}$ is smooth and the boundary $\partial \Jb_{g,n}$ of $\Jb_{g,n}$ is the pull-back of the boundary $\partial \Mb_{g,n}^{\,\qs}$ of $\Mb_{g,n}$. From the well-known deformation theory of nodal curves it follows that $(C,L)$ admits a smooth chart $V_{(C,L)}\to \Jb_{g,n}$ such that the pull-back $\partial V_{(C,L)}$ of the boundary $\partial \Jb_{g,n}$ is a normal crossing divisor whose irreducible components are $\{D_e\}_{e\in E(G(C))}$, where $D_e$ parametrizes all the partial smoothings of $(C,L)$ for which the node $n_e$ corresponding to the edge $e\in E(G(C))$ persists.   

From this local description of $\partial \Jb_{g,n}$, it follows that the stalk of the sheaf $E_{\Jb_{g,n}}$ at the geometric point $(C,L)\in \Jb_{g,n}(K)$ is canonically isomorphic to
\begin{equation}\label{E:stalkE}
E_{\Jb_{g,n}, (C,L)}=\NN^{E(G(C))}. 
\end{equation}
In particular, the pull-back of the sheaf $E_{\Jb_{g,n}}$ to the finite and \'etale cover $\wt\J_{(G,D)}\to \J_{(G,D)}$ of Lemma \ref{L:rigid-str} is equal to the constant sheaf 
\begin{equation}\label{E:stalkwtJ}
(E_{\Jb_{g,n}})_{|\wt\J_{(G,D)}}=\un{\NN}^{E(G)}. 
\end{equation}
Denote by $\eta_{(G,D)}$ the geometric generic point of the stratum $\J_{(G,D)}$.  Since $\eta_{(G,D)}$ is also the geometric generic point of the finite and \'etale cover $\wt\J_{(G,D)}\to \J_{(G,D)}$, \eqref{E:stalkwtJ} gives a canonical identification 
\begin{equation}\label{E:stalk-gen}
E_{\Jb_{g,n}, \eta_{(G,D)}}=\NN^{E(G)}. 
\end{equation}
This implies that the cone associated to $\J_{(G,D)}$ is canonically isomorphic to
\begin{equation}\label{E:conestra}
\sigma\big(\J_{(G,D)}\big):=\Hom_{\mon}\big(E_{\Jb_{g,n}, \eta_{(G,D)}}, \RR_{\geq 0}\big)=\Hom_{\mon}\big(\NN^{E(G)}, \RR_{\geq 0}\big)=\RR_{\geq 0}^{E(G)}. 
\end{equation}

Recall that our goal is to define an equivalence of categories $\calE: (\QD^E_{g,n})^{\opp}  \longrightarrow \Str(\Jb_{g,n})$ such that we have the equality of functors
\begin{equation}\label{E:commG}
\Gamma_{\QD_{g,n}^E}=\Gamma_{\Str(\Jb_{g,n})}\circ \calE. 
\end{equation}
We define $\calE$ on objects by sending $(G,D)\in \QD_{g,n}^E$ into $\J_{(G,D)}\in \Str(\Jb_{g,n})$, so that $\calE$ will be essentially surjective by Proposition \ref{TorStrata}\ref{TorStrata1}.
Moreover, \eqref{E:conestra} says that equality \eqref{E:commG} holds true at the level of objects. 

It remains to define $\calE$ on morphisms in such a way that it is fully faithful and \eqref{E:commG} is satisfied at the level of morphisms. 
Since any morphism in $\QD_{g,n}$ (and hence also on $\QD_{g,n}^E$) is a composition of edge contractions and automorphisms and any morphism in $\Str(\Jb_{g,n})$ is a composition of \'etale specializations (between the geometric generic points of the strata) and automorphisms, it will be enough to treat these two classes of morphisms separately. 

\textbf{Edge contractions and \'etale specializations}. Fix an element
$(G,D)\in \QD_{g,n}^E$. Given a  subset \linebreak
$S\subseteq E(G)$, denote by
$[\pi_S]:(G,D)\to (G',D')$ the unique morphism in $\QD_{g,n}^E$
induced by the morphism of graphs $\pi_S:G\to G/S:=G'$ which is the
contraction of the edges belonging to $S$ (so that \linebreak
$D':=(\pi_S)_*(D)$). Note that  
$$\Gamma_{\QD_{g,n}^E}\big([\pi_S]\big)\colon\RR_{\geq 0}^{E(G')}\hooklongrightarrow \RR_{\geq 0}^{E(G)}$$
 is induced by the inclusion $(\pi_S)_E^*\colon E(G')=E(G)\setminus S\hookrightarrow E(G)$. 

On the other hand, to a subset $S\subseteq E(G)$ we can associate the \'etale specialization  $\sp_S:\eta_{(G',D')} \rightsquigarrow \eta_{(G,D)}$ such that, on a smooth chart $V_{(G,D)}$ of the geometric point $\eta_{(G,D)}$ (as explained at the beginning of the proof), the pull-back of $\eta_{(G',D')}$ is the geometric generic point of the intersection $\bigcap_{e\in S} D_e$ (that is irreducible if the chart is small enough). 
Using the identification \eqref{E:stalk-gen}, the associated surjective monoid homomorphism 
$$
\sp_S^*\colon \NN^{E(G)}=E_{\Jb_{g,n},\eta_{(G,D)}}\twoheadlongrightarrow E_{\Jb_{g,n},\eta_{(G',D')}}=\NN^{E(G')}
$$
is induced by the inclusion $E(G')=E(G)\setminus S\hookrightarrow
E(G)$. In this way, we get $f_S:\J_{(G',D')}\to \J_{(G,D)}$, a morphism
in $\Str(\Jb_{g,n})$ corresponding to the \'etale specialization
$\sp_S$, with the property that, using the identification
\eqref{E:conestra}, the face inclusion of cones 
$$\Gamma_{\Str(\Jb_{g,n})}(f_S)\colon\RR_{\geq 0}^{E(G')}=\sigma\big(\J_{(G',D')}\big)\hooklongrightarrow \sigma\big(\J_{(G,D)}\big)=\RR_{\geq 0}^{E(G)}$$
 is induced by the inclusion $E(G')=E(G)\setminus S\hookrightarrow E(G)$. 

We now define $\calE([\pi_S])=f_S$ for any $S\subseteq E(G)$ and we are done since any edge contraction in $\QD_{g,n}^E$ with domain $(G,D)$ is equal to $[\pi_S]$ for a unique $S\subseteq E(G)$ and any morphism in $\Str(\Jb_{g,n})$ with codomain $\J_{(G,D)}$ and which is induced by \'etale specializations (between the geometric generic points of the strata) is equal to $f_S$ for a unique $S\subseteq E(G)$. 

\textbf{Automorphisms}. By the definition of the morphisms in $\QD_{g,n}^E$ (see Definition \ref{CatE}), the automorphism group in $\QD_{g,n}^E$ of an object $(G,D)\in \QD_{g,n}^E$ is equal to
\begin{equation}\label{E:autoQD}
\Aut_{\QD^E_{g,n}}(G,D)=\Im\big((-)_E^*\colon\Aut(G,D)\to S_{E(G)}\big)\subseteq S_{E(G)}, 
\end{equation}
where $\Aut(G,D)$ is the automorphism group of $(G,D)$ in $\QD_{g,n}$ and $S_{E(G)}$ is the permutation group on the set $E(G)$. Moreover, by the definition \eqref{E:Gamma-QDE} of the functor $\Gamma_{\QD^E_{g,n}}$,  the homomorphism  
$$\Gamma_{\QD^E_{g,n}}(-)\colon\Aut_{\QD^E_{g,n}}(G,D)\to \Aut_{\RPC^f}(\RR_{\geq 0}^{E(G)}),$$
is given by the restriction of the natural action of $S_{E(G)}$ on $\RR_{\geq 0}^{E(G)}$  by permutation of the extremal rays (note that indeed we have that $\Aut_{\RPC^f}\big(\RR_{\geq 0}^{E(G)}\big)= S_{E(G)}$).  

On the other hand, by the definition of the category $\Str(\Jb_{g,n})$ and using the canonical identification \eqref{E:stalk-gen}, the automorphism group in $\Str(\Jb_{g,n})$ of the stratum $\J_{(G,D)}$  is equal to
$$
\Aut_{\Str(\Jb_{g,n})}\big(\J_{(G,D)}\big)=\Im\big(\pi_1^{\et}(\J_{(G,D)},\eta_{(G,D)}\big)\to \Aut_{\mon}(\NN^{E(G)})=S_{E(G)}\big).
$$
By \eqref{E:stalkwtJ} and the definition of the action of $\Aut(G,D)$ on $\wt\J_{(G,D)}$ (see Lemma \ref{L:rigid-str}), the monodromy representation $\pi_1^{\et}(\J_{(G,D)},\eta_{(G,D)})\to \Aut_{\mon}(\NN^{E(G)})=S_{E(G)}$ factors through the quotient $\pi_1^{\et}(\J_{(G,D)},\eta_{(G,D)})\twoheadrightarrow \Aut(G,D)$ corresponding to the (representable, finite and \'etale, $\Aut(G,D)$-Galois) morphism $\wt\J_{(G,D)}\to \J_{(G,D)}$ followed by the homomorphism $(-)_E^*:\Aut(G,D)\to S_{E(G)}$. Hence, we deduce that 
\begin{equation}\label{E:autoStr}
\Aut_{\Str(\Jb_{g,n})}(\J_{(G,D)})=\Im\big((-)_E^*\colon\Aut(G,D)\to S_{E(G)}\big)\subseteq S_{E(G)}.
\end{equation}
Moreover, using the canonical identification \eqref{E:conestra} and the definition \eqref{E:Gamma-Str} of the functor $\Gamma_{\Str(\Jb_{g,n})}$,  the homomorphism  
$$\Gamma_{\Str(\Jb_{g,n})}(-)\colon\Aut_{\Str(\Jb_{g,n})}(\J_{(G,D)})\longrightarrow \Aut_{\RPC^f}(\RR_{\geq 0}^{E(G)}),$$
is again given by the restriction of the natural action of $S_{E(G)}$ on $\RR_{\geq 0}^{E(G)}$  by permutation of the extremal rays.

Hence, by comparing \eqref{E:autoQD} and \eqref{E:autoStr}, we can define a canonical isomorphism 
$$\calE(-)\colon \Aut_{\QD^E_{g,n}}(G,D)\xlongrightarrow{\cong}\Aut_{\Str(\Jb_{g,n})}(\J_{(G,D)}),
$$
which moreover satisfies \eqref{E:commG} by what we said above.  
\end{proof}

\begin{remark}\label{R:CatStrM}
Let  $\Str(\Mb_{g,n})$ be the category of strata of the toroidal embedding $(\M_{g,n}\subset \Mb_{g,n})$. Analogously to Proposition \ref{P:CatStrata}, there is an equivalence of categories 
\begin{equation}\label{E:equivF}
\begin{aligned}
\calF: (\SG^E_{g,n})^{\opp} & \longrightarrow \Str(\Mb_{g,n})\\
 G & \longmapsto \M_{ G},
\end{aligned}
\end{equation}
such that $\Gamma_{\SG_{g,n}^E}:=\Gamma_{\Str(\Jb_{g,n})}\circ \calF: (\SG_{g,n}^E)^{\opp} \longrightarrow \RPC^f$ is the factorization of the category fibered in groupoids $\Gamma_{\SG_{g,n}}$ of \eqref{E:combMt} through the quotient category $\SG_{g,n}^{\opp}\to (\SG_{g,n}^E)^{\opp}$.


Moreover, the two equivalences of categories \eqref{E:equivE} and \eqref{E:equivF} are compatible with the functor $F^E:\QD_{g,n}^E\to \SG_{g,n}^E$ of \eqref{E:funFE} and the functor 
$$
\begin{aligned}
\Str(\Phi)\colon \Str(\Jb_{g,n})& \longrightarrow \Str(\Mb_{g,n})\\
\J_{(G,D)} & \longmapsto \M_{G^{\st}}
\end{aligned}
$$ 
induced by the toroidal morphism $\Phi:\Jb_{g,n}\to \Mb_{g,n}$ (see Proposition~\ref{TorStrata}\ref{TorStrata3}).

\end{remark}


\section{Logarithmic tropicalization}\label{Sec:log-trop}


The aim of this section is to define and study a logarithmic tropicalization map from a logarithmic universal Jacobian to a universal tropical Jacobian defined over the category of logarithmic schemes. 

\subsection{The logarithmic universal Jacobian}

 Throughout the section, all logarithmic schemes (or stacks) are fine, saturated and locally of finite type over the base field $k$. A logarithmic scheme (or stack) will be denoted by  
 $X=(\un X, \alpha_X:M_X\to \cO_X)$, or simply by $X=(\un X, M_X)$, where $\un X$ is the underlying scheme, $M_X$ is a sheaf of monoids  and  $\alpha_X:M_X\to (\cO_{\un X},\cdot)$ is a morphism of sheaves of monoids which is an isomorphism over the subsheaf  $(\cO_{\un X}^*,\cdot)\subset (\cO_{\un X},\cdot)$ of invertible functions. We will denote by $\ov M_X:=M_X/\alpha_X^{-1}(\cO_{\un X}^*)$
  the characteristic monoid sheaf of $X$.  The category of logarithmic schemes (resp. logarithmic algebraic stacks) will be denoted by $\LSch$ (resp. $\LSta$). 


The forgetful functor from logarithmic algebraic stacks to algebraic stacks 
$$
\begin{aligned}
\LSta & \longrightarrow \Sta\\
X=(\un X,M_X) & \longmapsto \un X,
\end{aligned}
$$
admits a right adjoint 
$$
\begin{aligned}
\Sta & \longrightarrow \LSta\\
\un X& \longmapsto \un X:=(\un X,\cO_{\un X}^*\hookrightarrow \cO_{\un X}),
\end{aligned}
$$
where $\cO_{\un X}^*\hookrightarrow \cO_{\un X}$ is also called the trivial logarithmic structure on $\un X$. The unit of the above adjoint pair of functors gives rise to a morphism of logarithmic algebraic stacks 
\begin{equation}\label{E:morUps}
\Upsilon_X: X=(\un X,M_X)\longrightarrow \un X=(\un X,\cO_{\un X}^*)
\end{equation}
for any logarithmic algebraic stack $X$. 

As usual, the $2$-category $\LSta$ embeds into the $2$-category of categories fibered in groupoids over $\LSch$ by sending $X\in \LSta$ into $ \HOM_{\LSta}( -, X)$. Observe that the category fibered in groupoids over $\LSch$ associated to  $\un X$ is given by $\un X(S)=\un X(\un S)$ for any $S\in \LSch$.

\begin{definition}\label{D:logJ}
The \textbf{logarithmic universal Jacobian}, denoted by $\calJ_{g,n}^{\log}$, is the category fibered in groupoids over $\LSch$ whose fiber over  $S$ is the groupoid whose objects are pairs consisting of
\begin{itemize}
\item a \emph{quasi-stable logarithmic curve} $X\rightarrow S$ of type $(g,n)$, \textit{i.e.} an integral and saturated logarithmically smooth morphism $X\rightarrow S$ such that the underlying morphism $\underline{X}\rightarrow \underline{S}$ is a family of quasi-stable curves of type $(g,n)$;
\item an admissible line bundle $\calL$ on $\underline{X}$. 
 \end{itemize}
\end{definition}


\begin{remark}\label{R:logcurves}
Logarithmic curves (quasi-stable or not) have been characterized in \cite[Theorem~1.3]{Kat}. In particular, given a logarithmic curve $\pi: X\rightarrow S$, the stalk of the characteristic monoid $\ov M_X$  at a geometric point $x$ of $\un X$ has one of the following forms:
\begin{itemize}
\item $\ov M_{X,x}\cong \ov M_{S,\pi(x)}$ if $x$ is a smooth point of $\un X\to \un S$ which is not a marked point;
\item $\ov M_{X,x}\cong \ov M_{S,\pi(x)} \oplus \N v$ if $x$ is a marked point of  $\un X\to \un S$;
\item $\displaystyle \ov M_{X,x}\cong \frac{\ov M_{S,\pi(x)} \oplus \N \alpha\oplus \N \beta}{(\alpha+\beta=\delta_x)}$ if $x$ is a node of  $\un X\to \un S$, where $\delta_x$ is a non-zero element of $\ov M_{S,\pi(x)}$ which is called the smoothing parameter of the node $x$. 
\end{itemize}
\end{remark}

Denote by $\calMbar_{g,n}^{\,\qs}$ the algebraic stack of quasi-stable curves of type $(g,n)$ and by $\calM_{g,n}^{\log,\,\qs}$ the logarithmic algebraic stack of quasi-stable logarithmic curves of type $(g,n)$. If we consider $\Jb_{g,n}$ and $\calMbar_{g,n}^{\,\qs}$ as logarithmic algebraic stacks (by endowing them with the trivial logarithmic structure, as explained above), then it is an immediate consequence of 
Definition \ref{D:logJ} that 
\begin{equation}\label{E:Jlog=prod}
\J_{g,n}^{\log}=\Jb_{g,n}\times_{\Mb_{g,n}^{\,\qs}}\M_{g,n}^{\log,\,\qs} \ .
\end{equation}

In the next proposition, we will denote by $M_{\partial \Jb_{g,n}}$ the divisorial logarithmic structure on $\Jb_{g,n}$ associated to the normal crossing divisor $\partial \Jb_{g,n}$ (see Proposition~\ref{P:propJ}\ref{P:propJ5}).


\begin{proposition}\label{P:logJ-boun}
The logarithmic universal Jacobian $\J_{g,n}^{\log}$ is representable by the logarithmic algebraic stack $(\Jb_{g,n},M_{\partial \Jb_{g,n}})$, \textit{i.e.} we have an equivalence of categories fibered in groupoids
\begin{equation*}
\J_{g,n}^{\log}\simeq \HOM_{\LSta}\big( -, (\Jb_{g,n},M_{\partial \Jb_{g,n}})\big).
\end{equation*}
Moreover, the natural morphism $\Upsilon_{\J_{g,n}^{\log}}$ of \eqref{E:morUps} is given by 
$$
\begin{aligned}
\Upsilon_{\J_{g,n}^{\log}}:\J_{g,n}^{\log}(S) & \longrightarrow \Jb_{g,n}(S)=\Jb_{g,n}(\un S)\\
(X\to S, \cL) & \mapsto (\un X\to \un S, \cL),
\end{aligned}
$$
for any logarithmic scheme $S$. 
\end{proposition}
\begin{proof}
The arguments in \cite{Kat} imply that  $\calM_{g,n}^{\log,\,\qs}$ is representable by the logarithmic stack $(\Mb_{g,n}^{\,\qs}, M_{\partial \Mb_{g,n}^{\,\qs}})$ and that the natural morphism $\Upsilon_{\calM_{g,n}^{\log,\,\qs}}: \calM_{g,n}^{\log,\,\qs} \to \Mb_{g,n}^{\,\qs}$ sends a quasi-stable logarithmic curve $(X\to S)\in \calM_{g,n}^{\log,\,\qs}$ into its underlying family of quasi-stable curves $(\un X\to \un S)\in \Mb_{g,n}^{\,\qs}(\un S)=\Mb_{g,n}^{\,\qs}(S)$.  

The conclusion now follows using  \eqref{E:Jlog=prod}  and the fact that $\partial \Jb_{g,n}$ is the pull-back of $\partial \Mb_{g,n}^{\,\qs}$ via the smooth morphism $\Phi^{\qs}:\Jb_{g,n}\to \Mb_{g,n}^{\,\qs}$ (see Proposition \ref{P:propJ}).
\end{proof}

\begin{remark}\label{R:OtherLog}\begin{enumerate}[label={\small\textrm{(\roman*)}}]
\item The decomposition \eqref{E:connJac} of $\Jb_{g,n}$ into connected components induces the following decomposition of $\J_{g,n}^{\log}$ into connected components 
\begin{equation*}
\J_{g,n}^{\log}=\bigsqcup_{d\in \ZZ} \J_{g,n, d}^{\log}:=\bigsqcup_{d\in \ZZ} \Jb_{g,n, d}\times_{\Mb_{g,n}^{\,\qs}}\M_{g,n}^{\log,\,\qs}.
\end{equation*}
Explicitly, the stack $\J_{g,n,d}^{\log}$ parametrizes those objects $(X\rightarrow S,\calL)$ in $\calJ_{g,n}^{\log}$ such that for every geometric point $s$ of $S$ the line bundle $\calL_s$ has total degree $d$ on $\underline{X}_s$.

\item The open inclusions in \eqref{E:opensub} induce the following open inclusions  (for any universal stability condition $\phi\in V_{g,n}$)
 \begin{equation*}
\calJ_{g,n,(d)}^{\log,\,\spl}:=\Jb_{g,n,(d)}^{\,\spl}\times_{\Mb_{g,n}^{\,\qs}}\M_{g,n}^{\log,\,\qs}\subseteq \calJ_{g,n}^{\log}\supseteq \calJ_{g,n}^{\log}(\phi):=\Jb_{g,n}(\phi)\times_{\Mb_{g,n}^{\,\qs}}\M_{g,n}^{\log,\,\qs}.
\end{equation*}
Explicitly, the stack $\calJ_{g,n}^{\log}(\phi)$ parametrizes those objects $(X\rightarrow S,\calL)$ in $\calJ_{g,n}^{\log}$ such that for every geometric point $s$ of $S$ the line bundle $\calL_s$ on $\underline{X}_s$ is $\phi$-semistable, while the stack $\calJ_{g,n,(d)}^{\log,\,\spl}$ parametrizes those objects $(X\rightarrow S,\calL)$ in  $\calJ_{g,n}^{\log}$ such that for every geometric point $s$ of $S$ the curve $\un X_s$ remains connected when we remove its exceptional components (resp. and the line bundle $\calL_s$ has total degree $d$ on $\underline{X}_s$).

\item Proposition \ref{P:logJ-boun} implies that $\J_{g,n,d}^{\log}$, $\J_{g,n,(d)}^{\log,\,\spl}$ and  $ \calJ_{g,n}^{\log}(\phi)$ are represented, respectively, by $\Jb_{g,n,d}$, $\Jb_{g,n,(d)}^{\,\spl}$ and  $ \Jb_{g,n}(\phi)$, with the divisorial logarithmic structures associated to the normal crossing boundary divisors.
\end{enumerate}
\end{remark}

The forgetful-stabilization morphism \eqref{PhiAlg2} induces a logarithmic forgetful-stabilization morphism 
 \begin{equation}\label{PhiLog}
\Phi^{\log}:\J_{g,n}^{\log}\xrightarrow{\Phi^{\log,\,\qs}} \M_{g,n}^{\log,\,\qs}\xrightarrow{\st^{\log}} \M_{g,n}^{\log},
\end{equation}
where  $\M_{g,n}^{\log}$ is the logarithmic algebraic stack parametrizing stable logarithmic curves of type $(g,n)$, which is representable by $(\Mb_{g,n},M_{\partial \Mb_{g,n}})$ by \cite{Kat}, and the logarithmic stabilization morphism  $\st^{\log}$ is defined in \cite[Section 8.4]{CCUW}.

\subsection{Lifting the universal tropical Jacobian}
We next define a tropical universal Jacobian as a category fibered in groupoids over the category of logarithmic schemes. 

\begin{definition}\label{D:tropJnew}
The \textbf{tropical universal Jacobian} over $\LSch$, denoted by $\Jtw_{g,n}$, is the category fibered in groupoids over the category of logarithmic schemes such that the fiber over $S$ is the groupoid whose objects consists of
\begin{itemize}
\item a pair $(\Gamma_s/\ov M_{S,s},D_s)$ for each geometric point $s$ of $\un S$, where $\Gamma_s/\ov M_{S,s}$ is a \emph{quasi-stable tropical curve over $\ov M_{S,s}$} of type $(g,n)$, \textit{i.e.}  a quasi-stable graph $\GG(\Gamma_s)$ of type $(g,n)$ (called the underlying graph) and a (generalized) metric $d_{\Gamma_s}:E(\GG(\Gamma_s))\to \ov M_{S,s}\setminus \{0\}$, and $D_s$ is an admissible divisor on $\GG(\Gamma_s)$; 
\item for every \'etale specialization $f\colon t\rightsquigarrow s$ of geometric points of $\un S$, a morphism of graphs $\pi_f: \GG(\Gamma_s)\to \GG(\Gamma_t)$ such that $(\pi_f)_*(D_s)=D_t$ (so that  $\pi_f:(\GG(\Gamma_s),D_s)\to (\GG(\Gamma_t),D_t)$ is a morphism in $\QD_{g,n}$) and such that $\pi_f$ is compatible with $f^*:\ov M_{S,s}\to \ov M_{S,t}$, \textit{i.e.} such that for any 
$e\in E(\GG(\Gamma_s))$ we have that
\begin{itemize}
\item $\pi_f$ contracts $e$ if and only if $f^*(d_{\Gamma_s}(e))=0$;
\item if $\pi_f(e)=e'\in E(\GG(\Gamma_t))$ then $f^*(d_{\Gamma_s}(e))=d_{\Gamma_t}(e')$.
\end{itemize}
\end{itemize}
We can define the subcategories $\Jtw_{g,n,d}$, $\wt{\calJ}_{g,n,(d)}^{\mathrm{trop},\,\spl}$, $\Jtw_{g,n}(\phi)$ (for any universal stability condition $\phi\in V_{g,n}$) of $\Jtw_{g,n}$ by requiring that  the fibers over $S$ are the collections  $\big\{(\Gamma_s/\ov M_{S,s},D_s)\big\}_s$ as above with the extra property that $(\GG(\Gamma_s),D_s)$ belongs to, respectively, $\QD_{g,n,d}$, $\QD_{g,n,(d)}^{\spl}$, $\QD_{g,n}(\phi)$.
\end{definition}

The universal tropical Jacobian $\Jtw_{g,n}$ over $\LSch$ comes equipped with a forgetful-stabilization morphism to the stack $\Mtw_{g,n}$ of tropical curves over $\LSch$, as defined in \cite[Definition~7.4]{CCUW}:
\begin{equation}\label{Phitrop2}
\begin{aligned}
\wt{\Phi}^{\mathrm{trop}}:\Jtw_{g,n}& \longrightarrow \Mtw_{g,n}\\
\big\{(\Gamma_s/\ov M_{S,s},D_s)\big\}_s & \longmapsto \big\{\Gamma_s^{st}/\ov{M}_{S,s}\big\}_s,\\
\big\{\pi_f: \GG(\Gamma_s)\to \GG(\Gamma_t)\big\}_{f\colon t\rightsquigarrow s} & \longmapsto \big\{\pi_f^{st}: \GG(\Gamma_s)^{st}\to \GG(\Gamma_t)^{st}\big\}_{f\colon t\rightsquigarrow s} 
\end{aligned}
\end{equation}
where $\Gamma_s^{st}/\ov{M}_{S,s}$ is the stable tropical curve over $\ov{M}_{S,s}$ of type $(g,n)$ such that 
\begin{itemize}
\item the underlying graph is $\GG(\Gamma_s^{\st}):=\GG(\Gamma_s)^{\st}$,
\item the metric $d_{\Gamma_s^{\st}}:E(\GG(\Gamma_s)^{st})\to \ov M_{S,s}\setminus \{0\}$ is defined (using the notation below \eqref{E:edg-stab}) by 
$$d_{\Gamma_s^{\st}}(e)=
\begin{cases}
d_{\Gamma_s}(\wt e) & \text{ if } e\in E_{\nex}(\GG(\Gamma_s)^{\st}), \\
d_{\Gamma_s}(e^1)+d_{\Gamma_s}(e^2) & \text{ if } e\in E_{\exc}(\GG(\Gamma_s)^{\st}).
\end{cases}
$$
\end{itemize}


There is a close relationship between the universal tropical Jacobian $\Jt_{g,n}$ over $\RPCC$ (see Definition \ref{UnTrJac}) and the above defined  universal tropical Jacobian $\Jtw_{g,n}$ over $\LSch$, that we now want to explain. As explained in \cite[Section~6]{CCUW}, there is an equivalence of $2$-categories 
\begin{equation}\label{cone-Artin}
a^*:\big\{\text{Cone stacks}\big\} \xlongrightarrow{\cong}\big\{\text{Artin fans over } k\big\},
\end{equation}
 where an \emph{Artin fan} over $k$ is a logarithmic algebraic stack that is \'etale locally isomorphic to the stack quotient of a toric $k$-variety by its big torus (endowed with its natural  toric logarithmic structure).  Explicitly, the above equivalence $a^*$ sends a rational polyhedral cone $\sigma$ to 
 $$a^\ast \sigma=\big[\Spec k[S_\sigma]\big/\Spec k[S_\sigma^{gp}]\big].$$
More generally, the equivalence $a^*$ sends the cone stack associated to a combinatorial cone stack $\Psi:\cC\to \RPC^f$ (in the sense of \cite[Section~2]{CCUW})  to the Artin fan which is equal to the following colimit in the category of logarithmic algebraic stacks
\begin{equation}\label{colim-Artin}
\varinjlim_{C\in \cC} \big[\Spec k[S_{\Psi(C)}]\big/\Spec k[S_{\Psi(C)}^{gp}]\big],
\end{equation}
where all the morphisms appearing in the above colimit are open embeddings. 

We refer the interested reader to \cite{ACMW, ACMUW, Uli19, CCUW, Uli20} for more details on the theory of Artin fans. 

\begin{proposition}\label{P:2LogJ}
The universal tropical Jacobian $\Jtw_{g,n}$ over $\LSch$ is an Artin fan naturally isomorphic to $a^*\Jt_{g,n}$ and, moreover, we have $\wt{\Phi}^{\mathrm{trop}}\cong a^*\Phi^{\mathrm{trop}}$.  
\end{proposition}
 \begin{proof} 
 Our proof is a generalization
 of the proof of \cite[Lemma 7.9]{CCUW}. 
Let $S$ be a logarithmic scheme. 
Following the discussion after Definition \ref{UnTrJac} we may consider $\Jt_{g,n}$ as a category over $\textbf{ShpMon}^{op}$. Then $\Jt_{g,n}(\Gamma(S,\Mbar_S))$ consists of pairs $(\Gamma,D)$ with $\Gamma$ metrized by the monoid $\Gamma(S,\Mbar_S)$.
We can therefore associate to $S$
a natural map $\Jt_{g,n}(\Gamma(S,\Mbar_S))\rightarrow \Jtw_{g,n}(S)$
that is given by associating to $(\Gamma,D)$ in
$\Jt_{g,n}(\Gamma(S,\Mbar_S))$ the collection $\big\{(\Gamma_s/\ov
M_{S,s}, D_s)\big\}_s$ of contractions of $(\Gamma,D)$ induced by \linebreak
$\Gamma(S,\Mbar_S)\rightarrow\Mbar_{S,s}$ (for each geometric point
$s$ of $\un S$) together with the obvious morphisms
$\pi_f:\GG(\Gamma_s)\to \GG(\Gamma_t)$ induced by the generalization
maps $\ov M_{S,s}\to \ov M_{S,t}$ (for every \'etale specialization
$f\colon t\rightsquigarrow s$ of geometric points of $\un S$). 
 This map is natural in $S$ and therefore, since $a^\ast\Jt_{g,n}$ is a stack, it induces a morphism $a^\ast\Jt_{g,n}\rightarrow \Jtw_{g,n}$. 
 
 We argue that this is an isomorphism, \textit{i.e.} that every $\big\{(\Gamma_s/\ov M_{S,s},D_s)\big\}_s \in\Jtw_{g,n}(S)$ lifts uniquely, up to unique isomorphism, to an object in $a^\ast \Jt_{g,n}$. Since both $a^\ast\Jt_{g,n}$ and $\Jtw_{g,n}$ are stacks, this is an \'etale local assertion on $S$. So we may assume that there is a geometric point $s_0$ of $S$ and a chart $\Mbar_{S,s_0}\xrightarrow{\cong} \Gamma(S,\Mbar_{S})$.
 
 Using this homomorphism, 
 we obtain a natural lift $(\Gamma',D')\in\Jt_{g,n}(\Gamma(S,\Mbar_S))$ of $\big\{(\Gamma_s/\ov M_{S,s},D_s)\big\}_s$. 
 This lift is unique since the  composition 
 \begin{equation*}
 \Jt_{g,n}(\Gamma(S,\Mbar_S))\rightarrow\Jtw_{g,n}(S)\rightarrow\Jtw_{g,n}(s_0)=\Jt_{g,n}(\Mbar_{S,s_0})
 \end{equation*} 
 is an isomorphism. Let $\big\{(\Gamma''_s/\ov M_{S,s},D_s'')\big\}_s$ be the image of $(\Gamma',D')$ in $\Jtw_{g,n}(S)$. Then, by \cite[Lemma 7.8]{CCUW}, the isomorphism $(\Gamma''_{s_0}/\ov M_{S,s_0},D_{s_0}'')\simeq (\Gamma_{s_0}/\ov M_{S,s_0},D_{s_0})$ extends to an \'etale neighborhood of $s_0$ and so the lift $(\Gamma',D')$ is unique up to unique isomorphism. 

The second assertion, namely that $\wt{\Phi}^{\mathrm{trop}}\cong a^*\Phi^{\mathrm{trop}}$, follows from the compatibility of this isomorphism with stabilization. 
 \end{proof}
 
\begin{corollary}\label{C:StrataLog}
The universal tropical Jacobian $\Jtw_{g,n}$ over $\LSch$ admits a stratification into locally closed subsets 
\begin{equation}\label{E:StrataL}
\Jtw_{g,n}=\bigsqcup_{(G,D)\in \QD_{g,n}}  \Jtw_{(G,D)},
\end{equation}
with the property that $\un{\Jtw_{(G,D)}}$ is isomorphic to the classifying stack $B(\Gm^{E(G)}):=\big[\Spec k/\Gm^{E(G)}\big]$.  

Moreover, the poset of strata $\big\{\Jtw_{G,D}\big\}_{(G,D)\in \QD_{g,n}}$ is anti-isomorphic to the poset $|\QD_{g,n}|$ of Definition \ref{Pos}, \textit{i.e.} 
$$\Jtw_{(G,D)}\subset \ov{\Jtw_{(G',D')}} \Leftrightarrow (G,D)\geq  (G',D') \text{ in } |\QD_{g,n}|.
$$
\end{corollary}
\begin{proof}
Using Proposition \ref{P:2LogJ}, it is enough to exhibit the desired stratification for $a^*\Jt_{g,n}$.  Using Theorem~\ref{TrJacStack}\ref{TrJacStack1} and \eqref{colim-Artin}, we get that  
$a^*\Jt_{g,n}$ is equal to the following colimit in $\LSta$
\begin{equation}\label{E:colimitJ}
a^*\Jt_{g,n}=\varinjlim_{(G,D)\in \QD_{g,n}^{opp}} \big[\A^{E(G)}\big/\Gm^{E(G)}\big],
\end{equation}
where the morphisms in $\LSta$ appearing in the above colimits are given as follows: to a morphism $\pi:(G,D)\to (G',D')$ of $\QD_{g,n}$ we associate the morphism in $\LSta$  
$$\ov\pi: \left[\A^{E(G')}\big/\Gm^{E(G')}\right]=\left[(\A^{E(G')} \times \Gm^{E(G)\setminus E(G')})\big/\Gm^{E(G)}\right]\to \left[\A^{E(G)}\big/\Gm^{E(G)}\right]$$
 induced by the map $\pi_E^*:E(G')\hookrightarrow E(G)$ that identifies $E(G')$ with the subset of $E(G)$ consisting of edges of $G$ that are not contracted by $\pi: G\to G'$.

Therefore, we get a stratification of $a^*\Jt_{g,n}$ into locally closed subsets by considering, for any $(G,D)\in \QD_{g,n}$, the image in  $\big[\A^{E(G)}\big/\Gm^{E(G)}\big]$ of the unique closed $\Gm^{E(G)}$-orbit of $\A^{E(G)}$, namely ${\Jtw}_{(G,D)}:=\big[0\big/\Gm^{E(G)}\big]$ where $0$ is the origin of $\A^{E(G)}$.  

Moreover, since the morphisms appearing in the colimit are open embeddings, we get (for any $(G,D), (G',D')\in \QD_{g,n}$):
\begin{equation*}
{\Jtw}_{(G,D)}\subset \ov{{\Jtw}_{(G',D')}} \Leftrightarrow  \text{ there exists a morphism } \pi:(G,D)\to (G',D') \text{ in } \QD_{g,n}.
\end{equation*}
\end{proof}

 
\begin{remark}\label{R:OtherTr2}
It follows from Remark \ref{R:subcat}, together with the fact that the maps appearing in the diagram \eqref{E:colimitJ} are open inclusions, that the subcategories of $\Jtw_{g,n}$ introduced in Definition \ref{D:tropJnew} are Artin subfans, hence open substacks,  of  $\Jtw_{g,n}$ and they can be represented as
$$\Jtw_{g,n,d}=a^*\big(\Jt_{g,n,d}\big), \quad \wt{\calJ}_{g,n,(d)}^{\mathrm{trop}, \spl}=a^*\big(\calJ_{g,n,(d)}^{\mathrm{trop},\,\spl}\big), \quad \textrm{ and }\quad\Jtw_{g,n}(\phi)=a^*\big(\Jt_{g,n}(\phi)\big). $$
Moreover, the colimit description \eqref{E:colimitJ} and the stratification \eqref{E:StrataL} hold true for the above mentioned subcategories provided that we substitute the category $\QD_{g,n}$ by, respectively, $\QD_{g,n,d}$, $\QD_{g,n,(d)}^{\spl}$ or $\QD_{g,n}(\phi)$.
\end{remark}

\subsection{A modular tropicalization morphism}
We can now define a modular logarithmic tropicalization map.

\begin{definition}\label{D:log-trop}
The \textbf{modular logarithmic tropicalization map} (for $\Jb_{g,n}$) is the morphism of categories fibered in groupoids over $\LSch$ (and hence of logarithmic algebraic stacks)
\begin{equation}\label{E:trop-log}
\wt{\trop}_{\J_{g,n}^{\log}}\colon\J_{g,n}^{\log}\longrightarrow \Jtw_{g,n}
\end{equation}
defined (on objects) as follows. Consider an object $(\pi\colon X\to S,\cL)\in \J_{g,n}^{\log}(S)$. For every geometric point $s\to \un S$, we define a metric on the dual graph $\GG(\un X_s)$ of the geometric fiber of $\un X\to \un S$ over $s$ by 
$$
\begin{aligned}
d_{\pi}\colon E(\GG(\un X_s))& \longrightarrow \ov M_{S,s}\setminus \{0\}\\
e & \longmapsto \delta_{n_e},
\end{aligned}
$$
where $\delta_{n_e}$ is the smoothing parameter (in the sense of Remark \ref{R:logcurves}) of the node $n_e$ of $\un X_s$ corresponding to $e$. The pair $(\GG(\un X_s),d_{\pi})$ defines a quasi-stable tropical curve over $\ov M_{S,s}$ of type $(g,n)$, that we denote by $\Gamma(\un X_s)/\ov M_{S,s}$. We denote by $\un \deg(\cL_s)$ the multidegree of the line bundle $\cL_s:=\cL_{|\un X_s}$ on $\un X_s$, which is naturally an admissible divisor on $\GG(\un X_s)$. Moreover, given an \'etale specialization $f\colon t\rightsquigarrow s$ of geometric points of $\un S$, 
the family $\pi:X\to S$ induces a natural morphism of graphs $\pi_f: \GG(\un X_s)\to \GG(\un X_t)$ that is compatible with $f^*:\ov M_{S,s}\to \ov M_{S,t}$ and 
such that $(\pi_f)_*(\un\deg(\cL_s))=\un \deg(\cL_t)$. Then we define
\begin{equation}\label{E:trop-log2}
\wt{\trop}_{\J_{g,n}^{\log}}(\pi:X\to S, \cL):=\left(\big\{\Gamma(\un X_s)/\ov M_{S,s},\un \deg(\cL_s)\big\}_s,\big\{\pi_f:\GG(\un X_s)\to \GG(\un X_t)\big\}_{f\colon t\rightsquigarrow s}\right). 
\end{equation}
\end{definition}

From the above definition of $\wt{\trop}_{\J_{g,n}^{\log}}$, the explicit descriptions of the maps $\Phi^{\log}$ of \eqref{PhiLog} and $\wt{\Phi}^{\trop}$ of \eqref{Phitrop2}, we get the commutativity of the following  diagram 
\begin{equation}\label{diag-logtr}
\xymatrix{
\J_{g,n}^{\log}\ar[rr]^{\wt{\trop}_{\J_{g,n}}^{\log}} \ar[d]^{\Phi^{\log}}&& \Jtw_{g,n}\ar[d]^{\wt\Phi^{\trop}} \\
\M_{g,n}^{\log}\ar[rr]^{\wt{\trop}_{\M_{g,n}^{\log}}}&& \Mtw_{g,n}\\
}
\end{equation}
where $\wt \trop_{\M_{g,n}^{\log}}: \M_{g,n}^{\log}\to \Mtw_{g,n}$ is the modular logarithmic tropicalization map for $\M_{g,n}^{\log}$ defined in \cite[Section~7.2]{CCUW} (where it is called  tropicalization morphism and denoted by $\trop_{g,n}$). 

Before proving the main properties of the map $\wt{\trop}_{\J_{g,n}^{\log}}$ in the following Theorem, recall \cite[Proposition~3.2.1]{ACMW} that for any logarithmic stack $\cX$ (fine, saturated and locally of finite type over a base field $k$) there exists an Artin fan $\cA_{\cX}$ with faithful monodromy together with a natural strict morphism of logarithmic algebraic stacks (that we like to call the \textbf{functorial logarithmic tropicalization morphism} of $\cX$)
\begin{equation}\label{logtropfun}
\trop_{\cX}:\cX\longrightarrow \cA_{\cX}
\end{equation} 
that is initial among all strict morphisms to an Artin fan with faithful monodromy. Clearly, the map $\trop_{\cX}$ is functorial with respect to strict morphisms of logarithmic algebraic stacks, \textit{i.e.} for any strict morphism $\phi:\cX\to \cY$ of logarithmic stacks there exists a strict morphism $\cA(\phi):\cA_{\cX}\to \cA_{\cY}$ of Artin fans that fits into a commutative diagram 
$$
\xymatrix{
\cX \ar[d]_{\trop_{\cX}} \ar[r]^{\phi}&\cY  \ar[d]^{\trop_{\cY}}  \\
 \cA_{\cX}  \ar[r]_{\cA(\phi)}  & \cA_{\cY} 
}
$$

\begin{theorem}\label{T:log-trop}
\noindent 
\begin{enumerate}[label={\small\textrm{(\roman*)}}]
\item \label{T:log-trop1} The modular logarithmic tropicalization map $\wt{\trop}_{\J_{g,n}^{\log}}:\J_{g,n}^{\log}\to \Jtw_{g,n} $ is strict, smooth and surjective. 

\item \label{T:log-trop2} The morphism of Artin fans
$$\cA(\wt{\trop}_{\J_{g,n}^{\log}}):\cA_{\J_{g,n}^{\log}}\to \cA_{\Jtw_{g,n}}$$ 
is an isomorphism. In particular, we have a commutative diagram 
\begin{equation}\label{E:diagtropJ}
\xymatrix{
\J_{g,n}^{\log} \ar[rr]^{\wt\trop_{\J_{g,n}^{\log}}} \ar[d]_{\trop_{\J_{g,n}^{\log}}}&&\Jtw_{g,n}  \ar[d]^{\trop_{\Jtw_{g,n}}}  \\
 \cA_{\J_{g,n}^{\log}}  \ar[rr]_{\cA(\wt{\trop}_{\J_{g,n}^{\log}})}^{\cong}  && \cA_{\Jtw_{g,n}} 
}
\end{equation}

\item \label{T:log-trop2bis}
There exist morphisms of Artin fans $\cA(\wt \Phi^{\mathrm{trop}}): \cA_{\Jtw_{g,n}}\to \cA_{\Mtw_{g,n}}$ and $\cA(\Phi^{\log}):  \cA_{\J_{g,n}^{\log}}\to  \cA_{\M_{g,n}^{\log}}$ sitting into the commutative diagram 
\begin{equation}\label{E:diagtropJM}
\xymatrix{
 \J_{g,n}^{\log}\ar[rr]_{\wt\trop_{\J_{g,n}^{\log}}} \ar@/^2pc/[rrrrrr]^{\trop_{\J_{g,n}^{\log}}} \ar[d]^{\Phi^{\log}} && \Jtw_{g,n} \ar[rr]_{\trop_{\Jtw_{g,n}}}  \ar[d]^{\wt \Phi^{\mathrm{trop}}}&& \cA_{\Jtw_{g,n}}   \ar[d]^{\cA(\wt \Phi^{\mathrm{trop}})}&& \ar[ll]_{\cong}^{\cA(\wt{\trop}_{\J_{g,n}^{\log}})} \cA_{\J_{g,n}^{\log}} \ar[d]^{\cA(\Phi^{\log})},\\ 
  \M_{g,n}^{\log}\ar[rr]^{\wt\trop_{\M_{g,n}^{\log}}} \ar@/_2pc/[rrrrrr]_{\trop_{\J_{g,n}^{\log}}} && \Mtw_{g,n} \ar[rr]^{\trop_{\Mtw_{g,n}}} && \cA_{\Mtw_{g,n}}  && \ar[ll]^{\cong}_{\cA(\wt{\trop}_{\M_{g,n}^{\log}})} \cA_{\M_{g,n}^{\log}},\\ 
}
\end{equation}
where the first row is the diagram \eqref{E:diagtropJ} and the second row is the analogous diagram established in  \cite[Theorem~1.3]{Uli19} $($building upon \cite[\S 7.2]{CCUW}$)$.

 \item \label{T:log-trop3} The two morphisms of logarithmic algebraic stacks 
 \begin{equation}\label{E:corr-log}
\xymatrix{
\Jb_{g,n}&& \J_{g,n}^{\log}  \ar[ll]_{\Upsilon_{\J_{g,n}^{\log}}} \ar[rr]^{\wt{\trop}_{\J_{g,n}^{\log}}} && \Jtw_{g,n}
}
\end{equation}
are compatible with  the stratifications \eqref{E:strataJ}  and \eqref{E:StrataL}, \textit{i.e.}
 $$
 (\wt{\trop}_{\J_{g,n}^{\log}})^{-1}(\Jtw_{(G,D)})=(\Upsilon_{\J_{g,n}^{\log}})^{-1}(\J_{(G,D)})
 $$
 for every $(G,D)\in \QD_{g,n}$.
 In particular, we have that 
 \begin{itemize}
  \item $(\wt{\trop}_{\J_{g,n}^{\log}})^{-1}(\Jtw_{g,n,d})=\J_{g,n,d}^{\log}=(\Upsilon_{\J_{g,n}^{\log}})^{-1}(\Jb_{g,n,d})$ for any $d\in \ZZ$.
 \item $(\wt{\trop}_{\J_{g,n}^{\log}})^{-1}(\wt{\calJ}_{g,n,(d)}^{\,\mathrm{trop},\,\spl})=\J_{g,n,(d)}^{\log,\,\spl}=(\Upsilon_{\J_{g,n}^{\log}})^{-1}(\Jb_{g,n,(d)}^{\,\spl})$. 
  \item $(\wt{\trop}_{\J_{g,n}^{\log}})^{-1}(\Jtw_{g,n}(\phi))= \calJ_{g,n}^{\log}(\phi)=(\Upsilon_{\J_{g,n}^{\log}})^{-1}(\Jb_{g,n}(\phi))$ for any universal stability condition $\phi\in V_{g,n}$.  
 \end{itemize}
\end{enumerate}
\end{theorem}
Note that the morphisms of logarithmic algebraic stacks $\Phi^{\log}$ and $\wt \Phi^{\mathrm{trop}}$ are not strict, so that the existence of the morphisms of Artin fans $\cA(\wt \Phi^{\mathrm{trop}})$ and $\cA(\Phi^{\log})$ in part \ref{T:log-trop2bis} does not come automatically from the functoriality of the logarithmic tropicalization morphism $\trop_{\cX}$. 
\begin{proof}
Part \ref{T:log-trop1} follows along the lines of the proof of \cite[Theorem~7.12]{CCUW}. Given a quasi-stable tropical curve $\Gamma$ over a monoid $P$ together with an admissible divisor $D$ on $\GG(\Gamma)$, we can always find a quasi-stable logarithmic curve $X$ together with an admissible line bundle $L$ such that the dual tropical curve of $X$ is equal to $\Gamma$ and the multidegree of $L$ is equal to $D$. This shows that $\wt{\trop}_{\J_{g,n}^{\log}}$ is surjective. 

In order to show that $\wt{\trop}_{\J_{g,n}^{\log}}$ is strict, we need to check that for every scheme $S$ together with a morphism $M_S\rightarrow M_S'$ of logarithmic structures, any diagram
\begin{equation*}\begin{tikzcd}
(S,M_S)\arrow[rr]\arrow[d] &&\calJ_{g,n}^{\log}\arrow[d,"\wt{\trop}_{\J_{g,n}^{\log}}"]\\
(S,M_S')\arrow[rr]\arrow[rru,dashed]&&\wt{\calJ}_{g,n}^{\,\mathrm{trop}}
\end{tikzcd}\end{equation*}
has a unique lift as pictured. In other words: we have a pair $(\Gamma',D')$ consisting of a quasi-stable tropical curve over $(S,M_S')$ and an admissible divisor $D'$ on $\Gamma'$ and a pair $(X,M_X, L)$ consisting of a quasi-stable logarithmic curve $(X,M_X)$ over $(S,M_S)$ and an admissible line bundle $L$ such that the family $\Gamma_X$ of dual tropical curves of $X$ over $S$ is the family $\Gamma/(S,M_S)$ of tropical curves that is induced from $\Gamma'/(S,M_S)$. We wish to show that there is a quasi-stable logarithmic curve $(X,M_X')$ over $(S,M_S')$ and an admissible line bundle $L'$ on $X$ which induces both $(X,L)$ over $(S,M_S)$ and $(\Gamma',D')$ in $\wt{\calJ}_{g,n}^{\,\mathrm{trop}}(S,M_S
')$. From the proof of \cite[Theorem~7.12]{CCUW} we obtain such an $(X,M_X')$ by modifying the logarithmic structure on $(X,M_X)$ and we set $L'=L$ in order to find an $(X,M_X',L')$ with the desired properties.

Since $\wt{\calJ}_{g,n}^{\,\mathrm{trop}}$ is logarithmically \'etale over the base, we have that $\wt{\trop}_{\J_{g,n}^{\log}}$ is logarithmically smooth, since $\calJ_{g,n}^{\log}$ is logarithmically smooth over $k$. Since $\wt{\trop}_{\J_{g,n}^{\log}}$ is also strict, the map $\wt{\trop}_{\J_{g,n}^{\log}}$ is smooth (in the classical non-logarithmic sense), since strict logarithmically smooth morphisms are automatically smooth.

Let us now prove part \ref{T:log-trop2}. It follows from \cite[Proposition~4.5]{Uli19} that $ \cA_{\J_{g,n}^{\log}}$ is equal to the Artin fan associated, via the equivalence \eqref{cone-Artin}, to the (combinatorial) cone stack  $\Gamma_{\Str(\Jb_{g,n})}:\Str(\Jb_{g,n})\to \RPC^f$ of \eqref{E:Gamma-Str}. On the other hand, since $\Jtw_{g,n}$ is the Artin fan associated to the (combinatorial) cone stack $\Gamma_{\QD_{g,n}}:\QD_{g,n}^{\opp}\to \RPC^f$ of \eqref{E:combJt} (by Proposition \ref{P:2LogJ}), it follows from \cite[Cor. 4.6]{Uli19} that $\cA_{\Jtw_{g,n}}$ is the Artin fan associated to the (combinatorial) cone stack $\Gamma_{\QD_{g,n}^E}:(\QD_{g,n}^E)^{\opp}\to \RPC^f$ of \eqref{E:Gamma-QDE}. Proposition \ref{P:CatStrata}, together with the modular description of $\wt{\trop}_{\J_{g,n}^{\log}}$, implies that the map $\cA(\wt{\trop}_{\J_{g,n}^{\log}}):\cA_{\J_{g,n}^{\log}}\to \cA_{\Jtw_{g,n}}$ is an isomorphism of Artin fans. 

 Let us prove part \ref{T:log-trop2bis}. The commutativity of the left square of \eqref{E:diagtropJM} has been already observed in~\eqref{diag-logtr}. 
 
 We now define a morphism $\cA(\wt \Phi^{\mathrm{trop}}): \cA_{\Jtw_{g,n}}\to \cA_{\Mtw_{g,n}}$ of Artin fans making commutative the central square of \eqref{E:diagtropJM}.
 Consider the commutative diagram of  combinatorial cone stacks 
 \begin{equation}\label{E:com-combst}
 \xymatrix{
 \left(\Gamma_{\QD_{g,n}}: \QD_{g,n}^{\opp}  \longrightarrow \RPC^f \right) \ar[r] \ar[d] & \left(\Gamma_{\QD_{g,n}^E}: (\QD_{g,n}^E)^{\opp} \longrightarrow \RPC^f\right) \ar[d] \\
  \left(\Gamma_{\SG_{g,n}}: \SG_{g,n}^{\opp}  \longrightarrow \RPC^f \right) \ar[r] & \left(\Gamma_{\SG_{g,n}^E}: (\SG_{g,n}^E)^{\opp} \longrightarrow \RPC^f\right) 
 }
 \end{equation}
 where  the two horizontal arrows are the natural factorizations of $\Gamma_{\QD_{g,n}}$ (see \eqref{E:combJt}) and $\Gamma_{\SG_{g,n}}$ (see \eqref{E:combMt}) through the quotient categories in Definition \ref{CatE}, the left vertical arrow is given in Theorem~\ref{TrJacStack}\ref{TrJacStack2} and the right vertical arrow is the induced morphism. 
 Passing to the Artin fans associated to the above combinatorial cone stacks, we get the required commutative diagram 
   \begin{equation}\label{E:com-ArtJM}
 \xymatrix{
\Jtw_{g,n} \ar[rr]^{\trop_{\Jtw_{g,n}}} \ar[d]_{\wt \Phi^{\mathrm{trop}}} && \cA_{\Jtw_{g,n}} \ar[d]^{\cA(\wt \Phi^{\mathrm{trop}})} \\
\Mtw_{g,n} \ar[rr]^{\trop_{\Mtw_{g,n}}} && \cA_{\Mtw_{g,n}}  \\
 }
 \end{equation}
 where we used the description of $\wt \Phi^{\mathrm{trop}}$ in Proposition \ref{P:2LogJ} and Theorem~\ref{TrJacStack}\ref{TrJacStack2}, the description of  $\trop_{\Jtw_{g,n}}$ given in part~\ref{T:log-trop2} and the analogous description of $\trop_{\Mtw_{g,n}}$ given in \cite[\S 4.3]{Uli19}.

 Finally, since the morphisms  $\cA(\wt{\trop}_{\J_{g,n}^{\log}})$ and $\cA(\wt{\trop}_{\M_{g,n}^{\log}})$ are isomorphisms of Artin fans, there exists a unique morphism $\cA(\Phi^{\log}):  \cA_{\J_{g,n}^{\log}}\to  \cA_{\M_{g,n}^{\log}}$ making commutative the right square of \eqref{E:diagtropJM}.

Part~\ref{T:log-trop3}: the first assertion follows from the modular description of $\Upsilon_{\J_{g,n}^{\log}}$ given in Proposition~\ref{P:logJ-boun} and the definition of $\wt{\trop}_{\J_{g,n}^{\log}}$. 
The second assertion follows from the first one together with Remarks~\ref{R:OtherLog} and~\ref{R:OtherTr2}.
\end{proof}


\section{Analytic tropicalization}

The aim of this section is to describe the relation between the universal tropical Jacobian and the  analytification of the universal compactified Jacobian. In this section, we work over a fixed algebraically closed field $k$ on which we put the trivial valuation. 

\subsection{Analytification and functorial analytic tropicalization}\label{S:ana-ske}

In this subsection, we are going to review the definition of the (Berkovich) analytification of an Artin stack and of the functorial analytic tropicalization map of a  toroidal embedding of Artin stacks.
We will mostly follow the presentation and notation of \cite{Uli19} (that generalizes the previous works of \cite{Thu} and \cite{ACP}) although we will describe the tropicalization map in the special case of toroidal embeddings of stacks and not for arbitrary  logarithmic stacks. 


Let $\cX$ be an Artin stack locally of finite type over $k$. It is shown in \cite[Section~5.1]{Uli19} (generalizing the construction of Thuillier \cite{Thu} for schemes) that one can associate to $\cX$ (in a functorial way) a strict $k$-analytic stack $\cX^{\beth}$, \textit{i.e.} a stack in the category of strict analytic $k$-spaces endowed with the  G-\'etale topology having representable diagonal and admitting a G-smooth atlas. We will call $\cX^{\beth}$ the \textbf{beth-analytification} of $\cX$.

We  refrain from giving the definition of $\cX^{\beth}$ (since we will not need it in this paper), but we recall that the   topological space associated to  $\cX^{\beth}$ admits the following explicit description 
\begin{equation}\label{ana-beth}
\begin{aligned}
& \big\lvert\cX^{\beth}\big\rvert:=\big\{\Spec R\to \cX\big\}\big/\sim, \\
\end{aligned}
\end{equation}
where $R$ varies among all the rank-$1$ valuation rings containing $k$ and the equivalence relation $\sim$  is defined as follows: we say that $\Spec R\to \cX$ is equivalent to  $\Spec R'\to \cX$ if there exists another rank-$1$ valuation ring $R''$ containing both $R$ and $R'$, and a $2$-isomorphism between the two natural morphisms $\Spec R''\to \Spec R\to \cX$ and $\Spec R''\to \Spec R'\to \cX$. In particular, note that every point of $\cX^{\beth}$ can be represented by a morphism $\Spec R\to \cX$ where  $R$ is a complete rank-$1$ valuation ring with an algebraically closed residue field.  

The topological space $\vert \cX^{\beth}\vert$ admits an anticontinuous surjective map, called the \textbf{reduction map}, to the topological space $\vert \cX\vert$ underlying the stack $\cX$ which is defined by sending a point 
$[\phi:\Spec R\to \cX]\in \cX^{\beth}$ to the image via $\phi$ of the special point $o\in R$. More precisely, we have 
\begin{equation}\label{E:red-map}
\begin{aligned}
\red_{\cX}: \big\vert \cX^{\beth}\big\vert & \longrightarrow \vert \cX\vert \\
\big[\Spec R\to \cX\big] & \mapsto \big[\Spec R/\frakm_R\to \Spec R\to \cX\big],
\end{aligned}
\end{equation}
where $\frakm_R$ is the maximal ideal of $R$.

\begin{remark}\label{R:beth-comp}
Note that if $\cX$ is proper then, by the valuative criterion of properness, we may describe the topological space underlying $\cX^{\beth}$ as 
\begin{equation}\label{E:beth-comp}
\begin{aligned}
& \big\lvert\cX^{\beth}\big\rvert:=\big\{\Spec K\to \cX\big\}/\sim, \\
\end{aligned}
\end{equation}
where $K$ varies among all the rank-$1$ valuation fields extending $k$ with trivial valuation, and the equivalence relation is defined as above. This is the underlying topological space $\vert \calX^{\an}\vert$ of the non-Archimedean analytification $\calX^{\an}$ associated to $\calX$ in the sense of \cite[Section 2.3]{Uli17}.
\end{remark}


Assume now that we have a toroidal embedding of Artin stacks $(\cU\subset \cX)$, locally of finite type over $k$. Associated to $(\cU\subset \cX)$ there is a canonical topological space $\Sigma(\cX)$, together with a canonical compactification $\ov\Sigma(\cX)$, and a functorial analytic tropicalization map $\trop_{\cX}^{\an}:|\cX^{\beth}|\to \ov\Sigma(\cX)$, that we are now going to review following \cite{Uli19} (which deals with the more general situation of logarithmic Artin stacks) and \cite{ACP} (which deals with the special case of toroidal embeddings of DM-stacks).





As in \S \ref{S:catstraJ}, consider the lisse-\'etale sheaves $D_{\cX}$ and $E_{\cX}$ over $\cX$ such that, for every smooth morphism $V\to \cX$ with $V$ a scheme, $D_{\cX}(V)$ (resp. $E_{\cX}(V)$) is the group of Cartier divisors on $V$ (resp. the submonoid of effective Cartier divisors on $V$) that are supported on $V\setminus U$ where $U:=V\times_{\cX} \cU$. 
Consider the category of strata $\Str(\cX)=\{W_i\}$ of the toroidal embedding $\cU\subset \cX$ (as defined in \S\ref{S:toro-J}) and pick a geometric generic point $w_i$ in each stratum $W_i$.
By composing the functor \eqref{E:Gamma-Str} with the natural inclusion $\RPC^f\subset \TopE$, we get a functor 
\begin{equation}\label{func-str1}
\begin{aligned}
\Gamma_{\cX}:\Str(\cX) & \longrightarrow \TopE \\
W_i & \longmapsto \sigma(W_i):=\Hom_{\mon}(E_{\cX,w_i}, \R_{\geq 0}), \\
(W_i\rightarrow W_j) & \longmapsto \big(\sigma(W_i)\hookrightarrow \sigma(W_j)\big).
\end{aligned}
\end{equation}
We can also upgrade the above functor $\Gamma_{\cX}$ to a functor taking values on compact topological spaces as follows. To any stratum $W_i\in \Str(\cX)$, we associate the  extended cone 
$$\ov{\sigma}(W_i):=\Hom_{\mon}\big(E_{\cX,w_i}, \R_{\geq 0}\cup\{+\infty\}\big),$$ 
which is a canonical compactification of $\sigma(W_i)$.
Moreover, any morphism  $W_i\rightarrow W_j$ in $\Str(\cX)$ induces a surjective monoid homomorphism $E_{\cX,w_j}\twoheadrightarrow E_{\cX,w_i}$, and hence it induces an extended face morphism $\ov{\sigma}(W_i)\hookrightarrow \ov{\sigma}(W_j)$ of extended cones. In this way, we get a functor 
\begin{equation}\label{func-str2}
\begin{aligned}
\ov\Gamma_{\cX}:\Str(\cX) & \longrightarrow \TopE \\
W_i & \longmapsto \ov\sigma(W_i), \\
(W_i\rightarrow W_j) & \longmapsto \big(\ov\sigma(W_i)\hookrightarrow \ov\sigma(W_j)\big).
\end{aligned}
\end{equation}

The \textbf{generalized cone complex} $\Sigma(\cX)$ and the \textbf{generalized extended cone complex} $\ov{\Sigma}(\cX)$ (in the terminology of \cite[Section~2]{ACP}) of the toroidal embedding $\cU\subset \cX$ are given by 
\begin{equation}\label{skele}
\Sigma(\cX):=\colim_{W_i\in \Str(\cX)} \sigma(W_i)\subseteq \ov{\Sigma}(\cX):=\colim_{W_i\in \Str(\cX)} \ov{\sigma}(W_i).
\end{equation}
where the colimit is with respect to the two functors \eqref{func-str1} and \eqref{func-str2}.
We have a  stratification into locally closed subsets (see \cite[Proposition~2.6.2]{ACP})
\begin{equation}\label{skele-deco}
\Sigma(\cX)=\bigsqcup_{W_i\in \Str(\cX)} \sigma(W_i)^o/H_{W_i} \subset \ov{\Sigma}(\cX)=\bigsqcup_{W_i\in \Str(\cX)} \ov{\sigma}(W_i)^o/H_{W_i}
\end{equation}
where $\sigma(W_i)^o:=\Hom_{\mon}(E_{\cX,w_i}, \R_{>0})\subset \sigma(W_i)$,  $\ov{\sigma}(W_i)^o:=\Hom_{\mon}(E_{\cX,w_i}, \R_{>0}\cup\{+\infty\})\subset  \ov{\sigma}(W_i)$ and  $H_{W_i}$ is  the monodromy group of the stratum $W_i$.


The \textbf{functorial analytic tropicalization} map $\trop^{\an}_{\cX}:\left\vert \cX^{\beth} \right\vert\to \ov{\Sigma}(\cX)$  is defined as follows.  Consider a point $[\psi:\Spec R\to \cX]$ of $\vert \cX^{\beth} \vert$, where  $R$ is an integral domain which is complete with respect to a  rank-$1$ valuation $\val_R:R\to \RR_{\geq 0}\cup\{\infty\}$.  Let $x$ be the image via $\psi$ of the closed point of $\Spec R$ and let $W_i$ be the toroidal stratum containing $x$. Since $E_{\cX}$ is \'etale locally constant along each stratum $W_i$ (see \cite[Proposition~6.2.1]{ACP}), any \'etale specialization $s: w_i\rightsquigarrow x$ induces by pull-back an isomorphism  $s^*:E_{\cX,x}\xrightarrow{\cong}E_{\cX,w_i}$, which is well-defined up to the action of the monodromy group $H_{W_i}$ on $E_{\cX,w_i}$. 
Fix an \'etale specialization $s: w_i\rightsquigarrow x$ and consider the chain of monoid homomorphisms 
\begin{equation}\label{mon-hom}
E_{\cX,w_i}\xrightarrow[\cong]{(s^*)^{-1}}E_{\cX,x} \xrightarrow{\widetilde{\psi}} \big((R\setminus R^*)/R^*, \cdot\big) \xrightarrow{\ov{\val_R}} \RR_{>0}\cup\{\infty\}
\end{equation}
where the morphism $\widetilde{\psi}$ sends an effective Cartier divisor $D$ with local equation 
$f\in \widehat{\cO}_{\cX,x}$ at $x$  to the element $\psi^{\sharp}(f)\in R$ (well-defined up to units) which is not a unit of $R$ since $x\in W_i$, and $\ov{\val_R}$  is induced by the valuation $\val_R:R\to \RR_{\geq 0}\cup\{\infty\}$ using that the units $R^*$ are the only elements of $R$ having valuation zero.
Then $\trop^{\an}_{\cX}\big([\psi:\Spec R\to \cX]\big)$ is the image in $\ov{\Sigma}(\cX)$ of the monoid homomorphism \eqref{mon-hom} which is a well-defined element of $\ov{\sigma}(W_i)^o/H_{W_i}$.  
The map $\trop^{\an}_{\cX}$ is continuous, surjective and proper, and it is functorial with respect to toroidal morphisms.

\begin{remark}\label{R:retr}
The functorial analytic tropicalization map $\trop_{\cX}^{\an}$ has a section $J_{\cX}:\ov{\Sigma}(\cX)\to \left\vert \cX^{\beth} \right\vert$ such that the composition 
$$
\frakp_{\cX}=J_{\cX}\circ \trop_{\cX}^{\an}:\left\vert \cX^{\beth} \right\vert\longrightarrow \left\vert \cX^{\beth} \right\vert
$$ 
is a strong deformation retraction onto the non-archimedean skeleton of $\vert \cX^{\beth} \vert$, see \cite{Thu}, \cite{ACP}, and \cite[Proposition~6.3]{Uli19}, as well as \cite[Section 2.6]{Ranganathan_skeletonsofstablemapsII} for a generalization to toroidal Artin stacks.   
\end{remark}



\subsection{Analytic tropicalization of the universal Jacobian}

By applying the construction of the previous section \S \ref{S:ana-ske} to the forgetful morphism $\Phi:\Jb_{g,n}\to \Mb_{g,n}$, which is toroidal by Proposition~\ref{TorStrata}\ref{TorStrata3}, we get the following commutative diagram of continuous maps of topological spaces
\begin{equation}\label{diag-ana}
\xymatrix{
 \big\vert\Jb_{g,n}^{\,\beth}\big\vert \ar[rr]^{\trop^{\an}_{\Jb_{g,n}}} \ar[d]^{|\Phi^{\beth}|}&& \ov{\Sigma}(\Jb_{g,n})\ar[d]^{\Phi^{\Sigma}} \\
 \big|\Mb_{g,n}^{\,\beth} \big|\ar[rr]^{\trop^{\an}_{\Mb_{g,n}}}& &\ov{\Sigma}(\Mb_{g,n})\\
}
\end{equation}

\begin{remark}\label{R:OtherAna}
The decomposition \eqref{E:connJac} of $\Jb_{g,n}$ into connected components induces the following decomposition into connected components 
$$ \trop^{\an}_{\Jb_{g,n}}: \big\vert\Jb_{g,n}^{\,\beth}\big\vert=\bigsqcup_{d\in \ZZ}  \big\vert\Jb_{g,n,d}^{\beth}\big\vert \xrightarrow{\bigsqcup_{d\in \ZZ} \trop^{\an}_{\Jb_{g,n,d}}} \bigsqcup_{d\in \ZZ} \ov{\Sigma}(\Jb_{g,n,d})=\ov{\Sigma}(\Jb_{g,n})
$$
Moreover, we have that 
\begin{itemize}
\item  $\big\vert(\Jb_{g,n}^{\,\spl})^{\beth}\big\vert $ can be identified with the  closed subspace $\red_{\Jb_{g,n}}^{-1}\big(\vert\Jb_{g,n}^{\,\spl}\vert\big)\subset \big\vert\Jb_{g,n}^{\,\beth}\big\vert$;
\item  $\ov{\Sigma}(\Jb_{g,n}^{\,\spl})$ is a generalized extended cone subcomplex of  $\ov{\Sigma}(\Jb_{g,n})$;
\item   $(\trop^{\an}_{\Jb_{g,n}})^{-1}\big(\ov{\Sigma}(\Jb_{g,n}^{\,\spl})\big)=\big\vert(\Jb_{g,n}^{\,\spl})^{\beth}\big\vert$ and $\trop^{\an}_{\Jb_{g,n}^{\,\spl}}$ is the restriction of $\trop^{\an}_{\Jb_{g,n}}$ to $\big\vert(\Jb_{g,n}^{\,\spl})^{\beth}\big\vert$. 
\end{itemize}
Analogous statements hold for $\Jb_{g,n}(\phi)$ for any universal stability condition $\phi\in V_{g,n}$. 
\end{remark}

The aim of this subsection is to describe the  above diagram in terms of  tropical geometry. First of all, let us introduce the relevant tropical objects. 

\begin{definition}[(Quasi-)stable (extended) tropical curves]\label{D:tropcurvR}
\noindent 
\begin{enumerate}
\item A \emph{(quasi-)stable (resp. extended) tropical curve}  $\Gamma$ of type $(g,n)$ is  a (quasi-)stable graph $\GG\Gamma)$ of type $(g,n)$  together with a  metric $d_{\Gamma}:E(\GG(\Gamma))\to \RR_{>0}$ (resp. 
$d_{\Gamma}:E(\GG(\Gamma))\to \RR_{>0}\cup \{\infty\}$).
\item The \emph{stabilization} of a quasi-stable (resp. extended) tropical curve $\Gamma$ of type $(g,n)$ is the stable (resp. extended)  tropical curve $\Gamma^{st}$ of type $(g,n)$ such that 
\begin{itemize}
\item the underlying graph is $\GG(\Gamma^{\st}):=\GG(\Gamma)^{\st}$,
\item the metric $d_{\Gamma^{\st}}:E(\GG(\Gamma))\to \RR_{>0}$ (resp. $d_{\Gamma^{\st}}:E(\GG(\Gamma))\to \RR_{>0}\cup \{\infty\}$) is defined (using the notation below \eqref{E:edg-stab}) by 
$$d_{\Gamma^{\st}}(e)=
\begin{cases}
d_{\Gamma}(\wt e) & \text{ if } e\in E_{\nex}(\GG(\Gamma)^{\st}), \\
d_{\Gamma}(e^1)+d_{\Gamma}(e^2) & \text{ if } e\in E_{\exc}(\GG(\Gamma)^{\st}).
\end{cases}
$$
\end{itemize}
\end{enumerate}
\end{definition}

We now introduce the generalized (extended) cone complex associated to $\calJ_{g,n}^{\mathrm{trop}}$. 


\begin{definition}\label{TropGCC}
\noindent 
\begin{enumerate}
\item \label{TropGCC1} The 
 \textbf{generalized (resp. extended) cone complex} associated to $\calJ_{g,n}^{\mathrm{trop}}$ is defined as the colimit 
\begin{equation}\label{conecompl}
\JJ_{g,n}=\varinjlim_{(G,D)\in \QD_{g,n}} \RR_{\geq 0}^{E(G)}\subset \JJb_{g,n}=\varinjlim_{(G,D)\in \QD_{g,n}} \big(\RR_{\geq 0}\cup\{\infty\}\big)^{E(G)},
\end{equation}
which are obtained as colimits in the category of topological spaces $\TopE$ of, respectively, the diagram
\begin{equation}\label{E:GamQDtop}
\Gamma_{\QD_{g,n}}^{top}:\QD_{g,n}^{\opp}\xrightarrow{\Gamma_{\QD_{g,n}}} \RPC^f\subset \TopE,
\end{equation}
where $\Gamma_{\QD_{g,n}}$ is the functor \eqref{E:combJt},   and its natural extension
\begin{equation}\label{E:GamQDtop2}
\begin{aligned}
\ov \Gamma_{\QD_{g,n}}^{top}:\QD_{g,n}^{\opp}& \longrightarrow \TopE,\\
(G,D) & \longmapsto (\RR_{\geq 0}\cup\{+\infty\})^{E(G)},\\
\big(\pi:(G,D)\to (G',D')\big) & \longmapsto \big((\RR_{\geq 0}\cup\{+\infty\})^{E(G')}\hookrightarrow (\RR_{\geq 0}\cup\{+\infty\})^{E(G)}\big),
\end{aligned}
\end{equation}
using extended cones.

\item \label{TropGCC2} We denote by $\Phi^{\trop}$ the forgetful-stabilization morphism of generalized (resp. extended) cone complexes
\begin{equation}\label{PhiTrop}
\xymatrix{
\JJ_{g,n} \ar@{^{(}->}[r] \ar[d]^{\Phi^{\trop}}& \JJb_{g,n} \ar[d]^{\Phi^{\trop}} \\
 \MM_{g,n}:=\colim_{\ov G\in \SG_{g,n}} \RR_{\geq 0}^{E(\ov G)} \ar@{^{(}->}[r] &  \MMb_{g,n}:=\colim_{\ov G\in \SG_{g,n}} \ov{\RR_{\geq 0}^{E(\ov G)}}
}
\end{equation}
induced by, respectively,  the  morphism of diagrams 
\[
  \big(\Gamma_{\QD_{g,n}}:\QD_{g,n}^{\opp}\to \RPC^f\big)\to
  \big(\Gamma_{\SG_{g,n}}:\SG_{g,n}^{\opp}\to \RPC^f\big) 
\]
described in Theorem~\ref{TrJacStack}\ref{TrJacStack2} composed with the inclusion $\RPC^f\subset \TopE$, and its natural extension using extended cones. 
\end{enumerate}
\end{definition}


\begin{remark}\label{R:conecomE}
From the definition \eqref{E:combJt} of the functor  $\Gamma_{\QD_{g,n}}:\QD_{g,n}^{\opp}\to  \RPC^f$,  it follows that the morphism $\Gamma_{\QD_{g,n}}(\pi):\RR_{\geq 0}^{E(G')}\hookrightarrow \RR_{\geq 0}^{E(G)}$ in $\RPC^f$ induced by a morphism $\pi:(G,D)\to (G',D')$ in $\QD_{g,n}$ depends only from $\pi_E^*:E(G')\hookrightarrow E(G)$.
This implies that the functor $\Gamma_{\QD_{g,n}}$ factors through the functor $\Gamma_{\QD_{g,n}^E}$ of \eqref{E:Gamma-QDE} and hence that $\JJ_{g,n}$ can also be described as the colimit of the diagram 
\begin{equation}\label{E:GamQDEtop}
\Gamma_{\QD_{g,n}^E}^{top}:\QD_{g,n}^{\opp}\xrightarrow{\Gamma_{\QD^E_{g,n}}} \RPC^f\subset \TopE,
\end{equation}
Similarly, the functor $\ov \Gamma_{\QD_{g,n}}^{top}$ of \eqref{E:GamQDtop2} factors through a functor 
\begin{equation}\label{E:GamQDEtop2}
\ov \Gamma_{\QD_{g,n}^E}^{top}:\big(\QD_{g,n}^E\big)^{\opp} \longrightarrow \TopE
\end{equation}
whose limit is $\JJb_{g,n}$. 


\end{remark}

Note that $\JJ_{g,n}$ (resp. $\JJb_{g,n}$) naturally parametrizes pairs $(\Gamma,D)$, where  $\Gamma$ is a quasi-stable (resp. extended) tropical curve of type $(g,n)$  and $D$ is an admissible divisor on the underlying graph $\GG(\Gamma)$. Therefore, using also the presentations \eqref{conecompl}, $\JJ_{g,n}$ and $\JJb_{g,n}$  have a  stratification into locally closed subsets 
\begin{equation}\label{E:strataTr}
\JJ_{g,n}=\bigsqcup_{(G,D)\in \QD_{g,n}}  \JJ_{(G,D)} \subset \JJb_{g,n}=\bigsqcup_{(G,D)\in \QD_{g,n}} \JJb_{(G,D)}
\end{equation}
where $ \JJ_{(G,D)}\cong \RR_{>0}^{E(G)}/\Aut(G,D)$ (resp.  $\JJb_{(G,D)}\cong  (\RR_{>0}\cup\{\infty\})^{E(G)}/\Aut(G,D)$) is the locus parametrizing pairs $(\Gamma,D')\in \JJ_{g,n}$ (resp. $\JJb_{g,n}$) such that $(\GG(\Gamma),D')\cong (G,D)$. Moreover, since the maps appearing in the colimits \eqref{conecompl} are  face inclusions (and hence closed embeddings), we deduce that, given $(G,D), (G',D')\in \QD_{g,n}$, we have that 
\begin{equation}\label{E:strataTr2}
\JJ_{(G',D')}\subset \ov{J^{\mathrm{trop}}_{(G,D)}} \Leftrightarrow \JJb_{(G',D')}\subset \ov{\JJb_{(G,D)}} \Leftrightarrow (G,D)\geq (G',D') \text{ in } |\QD_{g,n}|.
\end{equation}

Similarly,  $\MM_{g,n}$ (resp. $\MMb_{g,n}$) naturally parametrizes stable (resp. extended) tropical curves of type $(g,n)$ (see \cite{BMV11} and \cite{ACP}) and, moreover, the map $\Phi^{\trop}$ sends $(\Gamma,D)\in \JJ_{g,n}$ (resp. $\in \JJb_{g,n}$)  into the stabilization $\Gamma^{\st}\in \MM_{g,n}$ (resp. $\in \MMb_{g,n}$) of $\Gamma$.


\begin{remark}\label{R:OtherTrop}
Similarly to Definition \ref{TropGCC}, we can define
$\JJ_{g,n,d}(\phi)\subset \JJb_{g,n,d}$,  $\JJs_{g,n,(d)}\subset
\ov{J}_{g,n,(d)}^{\mathrm{trop},\,\spl}$ and \linebreak $\JJ_{g,n}(\phi)\subset \JJb_{g,n}(\phi)$ for any universal stability condition $\phi\in V_{g,n}$, by replacing  $\QD_{g,n}$ with, respectively, $\QD_{g,n,d}$, $\QD_{g,n,(d)}^{\spl}$ and $\QD_{g,n}(\phi)$. The stratification \eqref{E:strataTr} and its modular description extends easily to these new generalized (resp. extended) cone complexes. 

From Remark \ref{R:subcat} and Definition \ref{TropGCC}, it follows that:
\begin{itemize}
\item  we have a decomposition into connected components 
$$\JJ_{g,n}=\bigsqcup_{d\in \ZZ} \JJ_{g,n,d} \subset \bigsqcup_{d\in \ZZ} \JJb_{g,n,d}=\JJb_{g,n}.
$$
\item $\JJs_{g,n,(d)}$ (resp. $\ov{J}_{g,n,(d)}^{\,\mathrm{trop},\,\spl}$) is a generalized (resp. extended) cone subcomplex, and hence a closed subspace,  of $\JJ_{g,n,(d)}$ (resp. $\JJb_{g,n,(d)}$).
\item  $\JJ_{g,n}(\phi)$ (resp. $ \JJb_{g,n}(\phi)$) is a connected generalized (resp. extended) cone subcomplex, and hence a closed subspace,  of $\JJ_{g,n,|\phi|}$ (resp. $\JJb_{g,n,|\phi|}$). 

\end{itemize}
\end{remark}

We now define a modular analytic tropicalization map.

\begin{definition}\label{D:ana-trop}
The \textbf{modular analytic tropicalization map} (for $\Jb_{g,n}$) is the map 
\begin{equation}\label{E:trop-ana}
\wt{\trop}^{\an}_{\Jb_{g,n}}\colon \left\vert \Jb_{g,n}^{\,\beth}\right\vert \longrightarrow \JJb_{g,n}
\end{equation}
defined as follows. Consider a point $\big[\Spec R\to \Jb_{g,n}\big]\in \big|\Jb_{g,n}^{\,\beth}\big|$, where $R$ is an integral domain which is complete with respect to a rank-$1$ valuation $\val_R:R\to \RR_{\geq 0}\cup\{\infty\}$ and having an algebraically closed residue field, 
and let $(\cC\to \Spec R, \cL)$ be the induced family of quasi-stable curves endowed with an admissible line bundle. 
On the dual graph $G(\cC_s)$ of the special fiber $\cC_s$ of $\cC\to \Spec R$ we consider the metric 
\begin{equation}\label{E:metricC}
\begin{aligned}
d_{\cC}\colon E(G(\cC_s))& \longrightarrow \big((R\setminus R^*)/R^*,\cdot\big) \xrightarrow{\ov\val_R} \RR_{>0}\cup\{\infty\}\\
e& \longmapsto [f_e] & 
\end{aligned}
\end{equation}
where $f_e$ is an element of $R$, well-defined up to units and which is not a unit, such that  an \'etale local equation for $\cC$ at the node $n_e$ of $\cC_s$ is given by
$xy=f_e$, and $\ov\val_R$ is induced by the valuation $\val_R$ using that the units $R^*$ of $R$ are the only elements of valuation zero. 
The pair $(G(\cC_s),d_{\cC})$ defines a quasi-stable extended tropical curve of type $(g,n)$, that we denote by 
$\Gamma(\cC_s)$, and  we set 
\begin{equation}\label{E:trop-ana2}
\wt{\trop}^{\an}_{\Jb_{g,n}}\big([\Spec R\to \Jb_{g,n}]\big):=\big(\Gamma(\cC_s), \un{\deg}(\cL_s)\big)\in \JJb_{g,n}.
\end{equation}
 \end{definition}
From the above definition of $\wt{\trop}^{\an}_{\Jb_{g,n}}$ and the explicit descriptions of the maps $\Phi$ of \eqref{PhiAlg} and $\Phi^{\trop}$ of \eqref{PhiTrop}, we get the commutativity of the following  diagram (of set-theoretic maps)
\begin{equation}\label{diag-anatr}
\xymatrix{
\big| \Jb_{g,n}^{\,\beth}\big|\ar[rr]^{\wt{\trop}^{\an}_{\Jb_{g,n}}} \ar[d]^{|\Phi^{\beth}|}&& \JJb_{g,n}\ar[d]^{\Phi^{\trop}} \\
\big|\Mb_{g,n}^{\,\beth}\big|\ar[rr]^{\wt{\trop}^{\an}_{\Mb_{g,n}}}&& \MMb_{g,n}\\
}
\end{equation}
where $\wt \trop_{\Mb_{g,n}}^{\an}: \big|\Mb_{g,n}^{\,\beth}\big|\to \MM_{g,n}$ is the modular analytic tropicalization map for $\Mb_{g,n}$ defined in \cite[Section~1.1]{ACP} (and therein called the naive set-theoretic tropicalization map). 

Note that we have similar diagrams with $\Jb_{g,n}$ and $\JJb_{g,n}$ replaced, respectively,  by either $\Jb_{g,n}^{\,\spl}$ and $\ov{J}_{g,n}^{\mathrm{trop},\,\spl}$, or $\Jb_{g,n}(\phi)$ and $\JJb_{g,n}(\phi)$ for any universal stability condition $\phi\in V_{g,n}$.

The main result of this subsection is an identification of the modular analytic tropicalization map $\wt{\trop}^{\an}_{\Jb_{g,n}}$ with the functorial analytic tropicalization map $\trop^{\an}_{\Jb_{g,n}}$. 

\begin{theorem}\label{T:ana-trop}
\noindent 
\begin{enumerate}[label={\small\textrm{(\roman*)}}]
\item \label{T:ana-trop0}  There are canonical isomorphisms $\Psi_{\Jb_{g,n}}:\Sigma(\Jb_{g,n})\xrightarrow{\cong} \JJ_{g,n}$ and $\ov\Psi_{\Jb_{g,n}}:\ov\Sigma(\Jb_{g,n})\xrightarrow{\cong} \JJb_{g,n}$ of, respectively, generalized  cone complexes and generalized extended cone complexes.

\item \label{T:ana-trop1} Using the above isomorphism $\ov\Psi_{\Jb_{g,n}}$, the map $\wt{\trop}^{\an}_{\Jb_{g,n}}:|\Jb_{g,n}^{\,\beth}|\to \JJb_{g,n} $ coincides with the map $\trop^{\an}_{\Jb_{g,n}}:|\Jb_{g,n}^{\,\beth}|\to \ov\Sigma(\Jb_{g,n})$, \textit{i.e.}  $ \wt{\trop}^{\an}_{\Jb_{g,n}}=\ov \Psi_{\Jb_{g,n}}\circ \trop_{\Jb_{g,n}}$. 
 In particular, $\wt{\trop}_{\Jb_{g,n}}^{\an}$ is continuous, surjective and proper. 
 
 \item \label{T:ana-trop2} Diagram \eqref{diag-anatr} coincides with Diagram \eqref{diag-ana}, \textit{i.e.} we have  the following commutative diagram 
\begin{equation}\label{BIG-anatrop}
\xymatrix{
\big| \Jb_{g,n}^{\,\beth}\big|\ar@(ur,ul)[rrrr]^{\wt\trop^{\an}_{\Jb_{g,n}}} \ar[d]^{|\Phi^{\beth}|} \ar[rr]_{\trop^{\an}_{\Jb_{g,n}}} && \ov\Sigma(\Jb_{g,n}) \ar[d]^{\Phi^{\Sigma}} \ar[rr]^{\ov\Psi_{\Jb_{g,n}}}_{\cong}&& \JJb_{g,n}\ar[d]^{\Phi^{\trop}} \\
\big| \Mb_{g,n}^{\,\beth}\big|\ar@(dr,dl)[rrrr]_{\wt\trop^{\an}_{\Mb_{g,n}}}\ar[rr]^{\trop^{\an}_{\Mb_{g,n}}} && \ov\Sigma(\Mb_{g,n}) \ar[rr]^{\ov\Psi_{\Mb_{g,n}}}_{\cong} && \MMb_{g,n}\\
}
\end{equation}
where $\ov\Psi_{\Mb_{g,n}}:\ov\Sigma(\Mb_{g,n})\xrightarrow{\cong} \MMb_{g,n}$ is the isomorphism of generalized extended cone complexes constructed in \cite[Theorem~1.2.1(1)]{ACP}. 
 
 \item \label{T:ana-trop3} The two maps 
 \begin{equation}\label{E:corr-trop}
\xymatrix{
\big|\Jb_{g,n}\big| && \big|\Jb_{g,n}^{\,\beth}\big|  \ar[ll]_{\red_{\Jb_{g,n}}} \ar[rr]^{\wt{\trop}^{\an}_{\Jb_{g,n}}} && \JJb_{g,n}
}
\end{equation}
 are compatible with the stratifications \eqref{E:strataJ}  and \eqref{E:strataTr}, \textit{i.e.}
 $$
 (\wt{\trop}^{\an}_{\Jb_{g,n}})^{-1}\big(\JJb_{(G,D)}\big)=(\red_{\Jb_{g,n}})^{-1}\big(|\J_{(G,D)}|\big)
 $$
 for every $(G,D)\in \QD_{g,n}$. 
  In particular, we have that 
 \begin{itemize}
  \item $(\wt{\trop}^{\an}_{\Jb_{g,n}})^{-1}\big(\JJb_{g,n,d}\big)=\big|\Jb_{g,n,d}^\beth\big|=(\red_{\Jb_{g,n}})^{-1}\big(|\Jb_{g,n,d}|\big)$ for any $d\in \ZZ$.
 \item $(\wt{\trop}^{\an}_{\Jb_{g,n}})^{-1}\big(\ov{J}_{g,n,(d)}^{\mathrm{trop},\,\spl}\big)=\big|(\Jb_{g,n,(d)}^{\,\spl})^\beth\big|=(\red_{\Jb_{g,n}})^{-1}\big(|\Jb_{g,n,(d)}^{\,\spl}|\big)$. 
  \item $(\wt{\trop}^{\an}_{\Jb_{g,n}})^{-1}\big(\JJb_{g,n}(\phi)\big)=\big|\Jb_{g,n}(\phi)^\beth\big|=(\red_{\Jb_{g,n}})^{-1}\big(|\Jb_{g,n}(\phi)|\big)$ for any universal stability condition $\phi\in V_{g,n}$.  
 \end{itemize}
\end{enumerate}
\end{theorem}
Parts \ref{T:ana-trop0}, \ref{T:ana-trop1} and \ref{T:ana-trop2} for $\Jb_{g,n}(\phi)$, in the special case $n=1$ and for specific choices of $\phi\in V_{g,n}$, have been proved  by Abreu-Pacini in \cite[Theorem~6.9]{API}. Both their proof and ours follow the blueprint provided by \cite{ACP} and go by giving an explicit description of the toroidal stratification (see Proposition \ref{P:CatStrata} above).

\begin{proof}
Let us first prove part \ref{T:ana-trop3}: consider a point $\big[\Spec R\to \Jb_{g,n}\big]\in |\Jb_{g,n}^{\,\beth}|$, where $R$ is an integral domain which is complete with respect to a rank-$1$ valuation $\val_R:R\to \RR_{\geq 0}\cup\{\infty\}$ and having an algebraically closed residue field $k$,  and let $(\cC\to \Spec R, \cL)$ be the induced family of quasi-stable curves endowed with an admissible line bundle. By \eqref{E:red-map}, we have that 
\begin{equation}\label{E:mapJ1}
\red_{\Jb_{g,n}}\big([\Spec R\to \Jb_{g,n}]\big)=(\cC_s,\cL_s)\in \Jb_{g,n}(k)
\end{equation}
where $s$ is the special point of $\Spec R$. On the other hand, by Definition \ref{D:ana-trop}, we have that 
 \begin{equation}\label{E:mapJ2}
\wt{\trop}^{\an}_{\Jb_{g,n}}\big([\Spec R\to \Jb_{g,n}]\big):=\big(\Gamma(\cC_s), \un{\deg}(\cL_s)\big)\in \JJb_{g,n},
\end{equation}
where $\Gamma(\cC_s)$ is the quasi-stable tropical curve of type $(g,n)$ whose underlying graph is $G(\cC_s)$ and whose metric is the metric $d_{\cC}$ in \eqref{E:metricC}. 
In particular, we have that 
\begin{equation}\label{E:mapJ3}
\big(G(\cC_s),\un\deg(\cL_s)\big)=\big(\GG(\Gamma(\cC_s)), \un{\deg}(\cL_s)\big). 
\end{equation}
By combining \eqref{E:mapJ1}, \eqref{E:mapJ2} and \eqref{E:mapJ3}, and recalling the definition of $\J_{(G,D)}$ from Proposition~\ref{TorStrata}\ref{TorStrata1} and of $\JJb_{(G,D)}$ from \eqref{E:strataTr}, we conclude that  (for any given $(G,D)\in \QD_{g,n}$)  
$$
\big[\Spec R\to \Jb_{g,n}\big]\in  (\wt{\trop}^{\an}_{\Jb_{g,n}})^{-1}(\JJb_{(G,D)})\Leftrightarrow \big[\Spec R\to \Jb_{g,n}\big]\in (\red_{\Jb_{g,n}})^{-1}\big(|\J_{(G,D)}|\big).
$$
The last assertion follows from what we already proved together with Remarks \ref{R:OtherAna} and \ref{R:OtherTrop}.

Part \ref{T:ana-trop0} follows immediately by combining Remark \ref{R:conecomE}, the definitions \eqref{skele} and Proposition \ref{P:CatStrata}.

Let us now prove part \ref{T:ana-trop1}, using the notation already introduced in the proof of parts~\ref{T:ana-trop3} and~\ref{T:ana-trop0}. Fix $(G,D)\in \QD_{g,n}$ such that 
$(\cC_s,\cL_s)\in \J_{(G,D)}(k)$. From the  the proof of Proposition \ref{P:CatStrata} (see \eqref{E:stalkE} and the deformation-theoretic arguments  that precedes it), it follows that  we have a canonical identification $E_{\J_{(G,D)},(\cC_s,\cL_s)}=\NN^{E(G(\cC_s))}$  and that the monoid homomorphism appearing in \eqref{mon-hom}
$$
\NN^{E(G(\cC_s))}=E_{\J_{(G,D)},(\cC_s,\cL_s)}\xlongrightarrow{\wt\psi} ((R\setminus R^*)/R^*,\cdot)
$$
is induced by the morphism $E(G(\cC_s))\to (R\setminus R^*)/R^*$ that sends $e\in E(G(\cC_s))$ into $[f_e]\in  (R\setminus R^*)/R^*$ where $f_e\in R\setminus R^*$ is such that an \'etale local equation for $\cC$ at the node $n_e$ is given by $xy=f_e$. We deduce that the monoid homomorphism 
$$
\NN^{E(G(\cC_s))}=E_{\J_{(G,D)},(\cC_s,\cL_s)}\xlongrightarrow{\wt\psi} \big((R\setminus R^*)/R^*,\cdot\big)\xlongrightarrow{\ov{\val}_R} \RR_{>0}\cup \{\infty\}
$$
 is induced by the metric $d_{\cC}:E(G(\cC_s))\to \RR_{>0}\cup \{\infty\}$ of \eqref{E:metricC}. From this, it follows that the isomorphism $\ov\Psi_{\Jb_{g,n}}$ constructed in~\ref{T:ana-trop0} sends the element $\trop^{\an}_{\Jb_{g,n}}\big([\Spec R\to \Jb_{g,n}]\big)\in \ov{\sigma}(G,D)^o/H_{(G,D)}\subset \ov\Sigma(\Jb_{g,n})$ to the tropical curve $\wt{\trop}^{\an}_{\Jb_{g,n}}\big([\Spec R\to \Jb_{g,n}]\big)\in \JJb_{g,n}$.

Part \ref{T:ana-trop2}: the commutativity of the upper triangle is proved in part~\ref{T:ana-trop1}, the commutativity of the lower triangle is proved in \cite[Theorem~1.2.1(2)]{ACP}, the commutativity of the left square has already been observed in \eqref{diag-ana}, and the commutativity of the right square follows from the commutativity of the diagram \eqref{diag-anatr} together with the surjectivity of the map  $\trop^{\an}_{\Jb_{g,n}}$. 
\end{proof}


\section{Fibers of the universal tropical Jacobian and examples}


\subsection{Fibers of the universal tropical Jacobian}


In this subsection, we will study  the ''fibers'' of the
forgetful-stabilization morphism 
\[
  \Phi^{\trop}:\Jt_{g,n}\to \Mt_{g,n}.
\]
 More precisely, we will describe the 
fiber product of  $\Phi^{\trop}:\Jt_{g,n}\to \Mt_{g,n}$ with a morphism $\sigma\xrightarrow{\ov \Gamma} \Mt_{g,n}$ induced by a stable tropical curve $\ov \Gamma$ over $\sigma\in \RPC$. 

\begin{definition}\label{Jac-fiber}
Let $\ov \Gamma/\sigma$ be a stable tropical curve over $\sigma\in \RPC$ and let $\ov G:=\GG(\ov \Gamma)\in \SG_{g,n}$ be its underlying stable graph. 
The \emph{Jacobian cone space} of $\ov \Gamma/\sigma$ is the (combinatorial) cone space $\Jac_{\ov\Gamma/\sigma}$ associated to the category fibered in groupoids
$$\QD_{\ov G}^{\opp}\longrightarrow \RPC^f$$
given by the following:
\begin{itemize}
\item To any object $(G,D,\rho)\in \QD_{\ov G}$, we associate the following fibered product over $\sigma$ of rational polyhedral cones over $\sigma$:
$$C(G,D,\rho):=\prod_{e\in E(G^{\st})}C(G,D,\rho)_e:= \prod_{e\in E_{\exc}(G^{\st})} (\sigma\times_{\RR_{\geq 0}} \RR_{\geq 0}^2)_e \times \prod_{e\in E_{\nex}(G^{\st})} \sigma  $$
where the  fibered product  $(\sigma\times_{\RR_{\geq 0}} \RR_{\geq 0}^2)_e $ is with respect to the following morphism of cones:
the morphism $\RR_{\geq 0}^2\to \RR_{\geq 0}$ is the addition map that sends $(a,b)$ into $a+b$, while the morphism $\sigma\to \RR_{\geq 0}$ is dual to the following morphism of  toric monoids 
$$\begin{aligned}
S_{\RR_{\geq 0}}=\NN& \longrightarrow S_{\sigma}, \\
1&\longmapsto d_{\ov \Gamma}(\rho_E^{*}(e)),
\end{aligned}$$
where  $\rho_E^*(e)\in E(\ov G)$ is the inverse image of $e\in E_{\exc}(G^{\st})$ via $\rho:\ov G\xrightarrow{\cong} G^{\st}$ and $d_{\ov \Gamma}:E(\ov G)\to S_{\sigma}$ is the generalized metric corresponding to the tropical curve $\ov \Gamma/\sigma$.

\item Let $\pi:(G,D,\rho)\to (G',D',\rho')$ be a morphism in $\QD_{\ov G}$. For $e\in E(G^{st})$ and $e'\in E(G'^{st})$ with $e=(\pi^{\st})_E^*(e')$ we have non-zero  maps $C(\pi)_{e',e}: C(G',D',\rho')_{e'} \to C(G,D,\rho)_{e} $ such that the following holds:
\begin{itemize}
\item   If $e$ and $e'$ are either both non-exceptional or both exceptional then $C(\pi)_{e',e}=\id$;
 \item Otherwise, we must have that $e'\in E_{\nex}(G'^{\st})$ and $e\in E_{\exc}(G^{\st})$ and we define  
 $$C(\pi)_{e',e}\colon  \sigma \hooklongrightarrow (\sigma\times_{\RR_{\geq 0}} \RR_{\geq 0}^2)_{e}$$ 
 to be induced by the $i$-th face inclusion $\RR_{\geq 0}\hooklongrightarrow \RR_{\geq 0}^2$ where $i=1,2$ is the index such that $e^i$ is not contracted by $\pi$ (while $e^{3-i}$ is necessarily contracted by $\pi$). 
\end{itemize}
We associate to $\pi$ the morphism 
$$C(\pi)=\prod_{\substack{e\in E(G^{\st}) \\ e'\in E(G'^{\st})}} C(\pi)_{e',e}\colon C(G',D',\rho')
\longrightarrow C(G,D, \rho)
$$ 
induced by the $C(\pi)_{e,e'}$.
\end{itemize}

In a similar way, we can define $\Jac_{\ov\Gamma/\sigma, d}$, $\Jac_{\ov\Gamma/\sigma, (d)}^{\spl}$ or $\Jac_{\ov\Gamma/\sigma}(\phi)$ using, respectively, the full subcategories 
$\QD_{\ov G,d}$,  $\QD_{\ov G, (d)}^{\spl}$ or $\QD_{\ov G}(\phi)$ of $\QD_{\ov G}$.
\end{definition}




\begin{theorem}\label{T:fib-forget}
Let $\ov \Gamma$ be a stable tropical curve over $\sigma\in \RPC$  and let $\sigma\xrightarrow{\ov \Gamma} \Mt_{g,n}$ be its modular morphism. Then we have a cartesian diagram 
$$
\xymatrix{ 
\Jac_{\ov \Gamma/\sigma} \ar[r]\ar[d] & \Jt_{g,n}\ar[d]^{\Phi^{\trop}}\\
\sigma\ar[r]^{\ov\Gamma} & \Mt_{g,n}
}
$$ 
\end{theorem}
\begin{proof}
Observe that by Definition \ref{Jac-fiber}, there is a natural morphism $\Jac_{\ov \Gamma/\sigma}\to \sigma$ and moreover, for any morphism $u:\tau \to \sigma$ in $\RPC$, we have that $\tau \times_{\sigma}\Jac_{\ov \Gamma/\sigma}\cong \Jac_{u^*(\ov \Gamma/\sigma)}$. Hence, in order to show the statement, it is enough to show that there is an isomorphism of 
groupoids 
$$\Jac_{\ov \Gamma/\sigma}(\sigma)\xlongrightarrow{\cong} \big(\sigma\times_{\Mt_{g,n}} \Jt_{g,n}\big)(\sigma),$$
between the fibers of $\Jac_{\ov \Gamma/\sigma}\to \sigma$ and of $\sigma\times_{\Mt_{g,n}} \Jt_{g,n}\to \sigma$ over the identity. 

By Definition \ref{Jac-fiber}, an object  of  $\Jac_{\ov \Gamma/\sigma}(\sigma)$ consists of  the datum
$$\big(G,D,\rho, s: \sigma\to C(G,D,\rho)\big),$$
where $(G,D,\rho)\in \QD_{\ov G}$ and $s$ is a section of the morphism $C(G,D,\rho)\to \sigma$  such that its image is not contained in any proper face of $C(G,D,\rho)$. By the definition of $C(G,D,\rho)\to \sigma$,  giving such a section $s: \sigma\to C(G,D,\rho)$ is equivalent, passing to the dual toric monoids, to specifying a collection of ordered pairs  of non-zero elements of $S_{\sigma}$ 
$$\big\{(l_e^1,l_e^2): l_e^1+l_e^2=d_{\ov \Gamma}(\rho_E^{*}(e))\big\}_{e\in E_{\exc}(G^{\st})}.
$$
The above collection of ordered pairs determines a tropical curve $\Gamma/\sigma$ such that $\GG(\Gamma)=G$ and the metric $d_{\Gamma}:E(G)\to S_{\sigma}$ is given by 
$$
\begin{cases}
d_{\Gamma}(e):=d_{\ov \Gamma}(\rho_E^{*}(\ov e)) & \text{ if } e\in E_{\nex}(G),\\
d_{\Gamma}(e^i):=l_{e}^i  &  \text{ if } e\in E_{\exc}(G^{\st}) \: \text{ and } \: i=1,2.
\end{cases}
$$
Moreover, by the above definition of $\Gamma/\sigma$,  the isomorphism of graphs $\rho:\ov G\xrightarrow{\cong} G^{\st}$ induces an isomorphism $\rho:\Gamma^{\st}/\sigma\xrightarrow{\cong} \ov \Gamma/\sigma$ of tropical curves over $\sigma$. Hence, we get an object 
$$\big(\Gamma/\sigma,D,\rho: \Gamma^{\st}/\sigma\xrightarrow{\cong} \ov \Gamma/\sigma\big)\in \big(\sigma\times_{\Mt_{g,n}} \Jt_{g,n}\big)(\sigma).$$ 

Moreover, again from Definition \ref{Jac-fiber}, a morphism 
\begin{equation}\label{E:morJacfib}
\big(G,D,\rho, s: \sigma\to C(G,D,\rho)\big)\longrightarrow \big(G',D',\rho', s': \sigma\to C(G',D',\rho')\big)
\end{equation}
in $\Jac_{\ov \Gamma/\sigma}(\sigma)$ consists of the datum of an isomorphism $\pi:(G',D',\rho')\to (G,D,\rho)$ in $\QD_{\ov G}$ with the property that $s'=C(\pi)\circ s$. In terms of the collections 
of non-zero elements of $S_{\sigma}$ associated to the objects in \eqref{E:morJacfib}
$$\big\{(l_e^1,l_e^2): l_e^1+l_e^2=d_{\ov \Gamma}(\rho_E^{*}(e))\big\}_{e\in E_{\exc}(G^{\st})} \quad \text{ and } \quad 
\big\{(l_{e'}'^1,l_{e'}'^2): l_{e'}'^1+l_{e'}'^2=d_{\ov \Gamma}((\rho')_E^{*}(e'))\big\}_{e'\in E_{\exc}(G'^{\st})},
$$
the  compatibility condition $s'=C(\pi)\circ s$ on sections translates into the equalities in $S_{\sigma}$
$$l'^i_{(\pi^{\st})_E^*(e)}=l_e^i \quad \text{ for any } i=1,2.$$
This condition then implies that  $\pi:(G',D',\rho')\to (G,D,\rho)$ gives rise to a morphism in $(\sigma\times_{\Mt_{g,n}} \Jt_{g,n})(\sigma)$ 
$$
\big(\Gamma/\sigma,D,\rho: \Gamma^{\st}/\sigma\xrightarrow{\cong} \ov \Gamma/\sigma\big) \longrightarrow \big(\Gamma'/\sigma,D',\rho': \Gamma'^{\st}/\sigma\xrightarrow{\cong} \ov \Gamma'/\sigma\big)
$$
between the objects associated to the two objects in \eqref{E:morJacfib}.

Putting everything together, we obtain a morphism of groupoids 
\begin{equation}\label{E:mor-grps}
\begin{aligned}
\Jac_{\ov \Gamma/\sigma}(\sigma)& \xlongrightarrow{} (\sigma\times_{\Mt_{g,n}} \Jt_{g,n})(\sigma),\\
\big(G,D,\rho, s: \sigma\to C(G,D,\rho)\big) & \longmapsto \big(\Gamma/\sigma,D,\rho: \Gamma^{\st}/\sigma\xrightarrow{\cong} \ov \Gamma/\sigma\big). 
\end{aligned}
\end{equation}
Since the above construction can be reversed, it follows that the morphism \eqref{E:mor-grps} is an isomorphism and we are done. 
\end{proof}

\subsection{The universal tropical Jacobian as a topological stack}\label{S:top-Jac}

The aim of this subsection is to construct a topological realization of $\Phi^{\trop}:\Jt_{g,n}\to \Mt_{g,n}$ and to study its fibers. 

As explained in \cite[Section~5.3]{CCUW}, we can extend \eqref{Phitrop} to a morphism of \emph{real cone stacks} (\textit{i.e.} geometric stacks over the category $\PC$ of (non necessarily rational) polyhedral cones): 
\begin{equation}\label{PhitropR}
\Phi^{\mathrm{trop},\,\RR}:\JtR_{g,n}\longrightarrow \MtR_{g,n},
\end{equation}
As in \cite[Proposition~5.9]{CCUW}, the fiber of the above morphism over a given  $\sigma\in \PC$ can be described as follows:
\begin{itemize}
\item $\JtR_{g,n}(\sigma)$ is the groupoid of pairs $(\Gamma,D)$ consisting of a quasi-stable tropical curve $\Gamma$ of type $(g,n)$ with edge lengths over the dual cone $\sigma^{{\vee}}$, \textit{i.e.} a quasi-stable graph $\GG(\Gamma)$ of type $(g,n)$ endowed with a generalized metric $d_{\Gamma}:E(\GG(\Gamma))\to \sigma^{{\vee}}\setminus \{0\}$,  and $D$ is a divisor on  $\GG(\Gamma)$ such that $(\GG(\Gamma),D)\in \QD_{g,n}$;
\item $\MtR_{g,n}(\sigma)$ is the groupoid of stable tropical curves $\ov \Gamma$ of type $(g,n)$ with edge lengths over the dual cone $\sigma^{{\vee}}$, \textit{i.e.} a stable graph $\GG(\ov \Gamma)$ of type $(g,n)$ 
endowed with a generalized metric $d_{\ov \Gamma}:E(\GG(\ov \Gamma))\to \sigma^{{\vee}}\setminus \{0\}$;
\item $\Phi^{\mathrm{trop},\,\RR}(\sigma)$ sends $(\Gamma, D)\in \JtR_{g,n}(\sigma)$ into the stabilization $\Gamma^{{\st}}\in \MtR_{g,n}(\sigma)$ of $\Gamma$, which is 
defined by  setting $\GG(\Gamma^{\st}):=\GG(\Gamma)^{\st}$ and  $d_{\Gamma^{\st}}:E(\GG(\Gamma))\to \sigma^{{\vee}}\setminus \{0\}$ equal to  (using the notation below \eqref{E:edg-stab})
$$
d_{\Gamma^{\st}}(e)=
\begin{cases}
d_{\Gamma}(\wt e) & \text{ if } e\in E_{\nex}(\GG(\Gamma)^{\st}), \\
d_{\Gamma}(e^1)+d_{\Gamma}(e^2) &  \text{ if } e\in E_{\exc}(\GG(\Gamma)^{\st}). 
\end{cases}
$$
\end{itemize}
In a similar way, we can define $\Phi^{\trop,\,\RR}:\JtR_{g,n,d}\to
\MtR_{g,n}$, $\Phi^{\trop,\,\RR}:\JtsR_{g,n,(d)}\to \MtR_{g,n}$ and
$\Phi^{\trop,\,\RR}:\JtR_{g,n}(\phi)\to \MtR_{g,n}$.

Going one step further as explained in \cite[Section~5.4]{CCUW}, we can take the topological realizations of the above real cone stacks obtaining a morphism of \emph{topological stacks} (\textit{i.e.} geometric stacks over the category $\TopE$ of topological spaces):
\begin{equation}\label{PhitropTop}
\big\vert\Phi^{\mathrm{trop}}\big\vert\colon \big\vert\Jt_{g,n}\big\vert\longrightarrow \big\vert\Mt_{g,n}\big\vert.
\end{equation}
It follows from \cite[Lemma 5.1]{CCUW} that at the level of points, \textit{i.e.} taking fibers over the topological space $\star$ with one point, we have that:
\begin{itemize}
\item $\big|\Jt_{g,n}\big|(\star)=\JtR_{g,n}(\RR_{\geq 0})$ is the groupoid of pairs $(\Gamma,D)$ consisting of a quasi-stable tropical curve $\Gamma$ of type $(g,n)$ with real edge lengths, \textit{i.e.} a quasi-stable graph $\GG(\Gamma)$ of type $(g,n)$ 
endowed with a  metric $d_{\Gamma}:E(\GG(\Gamma))\to \RR_{>0}$, and a divisor $D$  on  $\GG(\Gamma)$ such that $(\GG(\Gamma),D)\in \QD_{g,n}$;
\item $\big|\Mt_{g,n}\big|(\star)=\MtR_{g,n}(\RR_{\geq 0})$ is the groupoid of stable tropical curves $\ov \Gamma$ of type $(g,n)$ with real edge lengths, \textit{i.e.} a stable graph $\GG(\ov \Gamma)$ of type $(g,n)$ 
endowed with a metric $d_{\ov \Gamma}:E(\GG(\ov \Gamma))\to \RR_{>0}$;
\item $\big|\Phi^{\mathrm{trop}}\big|(\star)=\Phi^{\mathrm{trop},\,\RR}(\RR_{\geq 0})$ sends $(\Gamma, D)\in \JtR_{g,n}(\star)$ into the stabilization $\Gamma^{{\st}}\in \MtR_{g,n}(\star)$ of $\Gamma$.
\end{itemize}
In a similar way, we can define $|\Phi^{\mathrm{trop}}|:
\big|\Jt_{g,n,d}\big|\to \big|\Mt_{g,n}\big|$,
$\big|\Phi^{\trop}\big|: \big|\Jts_{g,n,(d)}\big|\to
\big|\Mt_{g,n}\big|$ and \linebreak $\big|\Phi^{\trop}\big|:
\big|\Jt_{g,n}(\phi)\big|\to \big|\Mt_{g,n}\big|$.

We now want to describe the fiber of $\big|\Phi^{\mathrm{trop}}\big|:\big|\Jt_{g,n}\big|\to \big|\Mt_{g,n}\big|$ over a tropical curve $\ov \Gamma\in \MtR_{g,n}(\star)$. 

\begin{definition}\label{JacTp-fiber}
Let $\ov \Gamma\in \MtR_{g,n}(\star) $ be a stable tropical curve with real edge lengths and let $\ov G:=\GG(\ov \Gamma)\in \SG_{g,n}$ be its underlying stable graph. 
The \emph{Jacobian topological space} of $\ov \Gamma$ is the topological space $\Jac_{\ov\Gamma}$ obtained as the colimit associated to the  functor 
$$\QD_{\ov G}^{\opp}\longrightarrow \TopE$$
defined by 
\begin{itemize}
\item to any object $(G,D,\rho)\in \QD_{\ov G}$, we associate the following polytope (seen as a topological space)
$$P(G,D,\rho):=\prod_{e\in E(G^{\st})}P(G,D,\rho)_e:= \prod_{e\in E_{\exc}(G^{\st})} \big[0,d(\rho_E^{*}(e))\big] \times \prod_{e\in E_{\nex}(G^{\st})}\star, $$
where  $\rho_E^*(e)\in E(\ov G)$ is the inverse image of $e$ via $\rho$ and  $d:E(\ov G)\to \RR_{>0}$ is the  metric corresponding to the tropical curve $\ov \Gamma$.
\item to any morphism $\pi:(G,D,\rho)\to (G',D',\rho')$ in $\QD_{\ov G}$, we associate the morphism 
$$P(\pi)=\prod_{\substack{e\in E(G^{\st}) \\ e'\in E(G'^{\st})}} P(\pi)_{e',e}:P(G',D',\rho')
\longrightarrow P(G,D, \rho)
$$ 
where the possibly non-zero  maps $P(\pi)_{e',e}: P(G',D',\rho')_{e'} \to P(G,D,\rho)_{e} $ are those for which  $e=(\pi^{\st})_E^*(e')$. In this case we have that 
\begin{itemize}
\item   if $e$ and $e'$ are either both non-exceptional or both exceptional then $P(\pi)_{e',e}=\id$;
 \item otherwise, we must have that $e'\in E_{\nex}(G'^{\st})$ and $e=e_v\in E_{\exc}(G^{\st})$ for $v\in V_{\exc}(G)$ and we define 
 $$\begin{aligned}
P(\pi)_{e',e}:  \star & \hooklongrightarrow [0,d(\rho_E^{*}(e)] \\
\star & \longmapsto 
\begin{cases} 
0 & \text{ if } e_{v}^1 \text{ is contracted by } \pi,\\
d(\rho_E^{*}(e)) & \text{ if } e_{v}^2 \text{ is contracted by } \pi.
\end{cases}
\end{aligned}$$ 
 \end{itemize}
\end{itemize}
In a similar way, we can define $\Jac_{\ov\Gamma, d}$, $\Jac_{\ov\Gamma, (d)}^{\spl}$ or $\Jac_{\ov\Gamma}(\phi)$ using, respectively, the full subcategories 
$\QD_{\ov G,d}$,  $\QD_{\ov G, (d)}^{\spl}$ or $\QD_{\ov G}(\phi)$ of $\QD_{\ov G}$.
\end{definition}

\begin{theorem}\label{T:fibT-forget}
Let $\ov \Gamma\in \big|\MtR_{g,n}\big|(\star) $ be a stable tropical curve with real edge lengths and let $\star\xrightarrow{\ov \Gamma} |\Mt_{g,n}|$ be its associated modular morphism. 
Then we have a cartesian diagram 
$$
\xymatrix{ 
\Jac_{\ov \Gamma} \ar[r]\ar[d] & \big|\Jt_{g,n}\big|\ar[d]^{|\Phi^{\trop}|}\\
\star \ar[r]^{\ov\Gamma} & \big|\Mt_{g,n}\big|
}
$$ 
\end{theorem}
\begin{proof}
The proof is similar to the proof of \cite[Proposition~5.12]{CCUW}.

We can choose a rational polyhedral cone $\sigma$, a map of cones $u: \RR_{\geq 0}\to \sigma$ and a stable tropical curve $\wt \Gamma/\sigma$ over $\sigma$ such that $\GG(\wt \Gamma)=\GG(\ov \Gamma)$ and the metric $d_{\ov \Gamma}$ of $\ov \Gamma$ is equal to the 
following composition 
$$d_{\ov \Gamma}\colon E(\GG(\ov \Gamma))=E(\GG(\wt \Gamma)) \xlongrightarrow{d_{\wt \Gamma}} S_{\sigma}\hooklongrightarrow \sigma^{\vee}\xlongrightarrow{u^{\vee}} (\RR_{\geq 0})^{\vee}\cong \RR_{\geq 0}.$$ 
This is equivalent to saying that the morphism of real cone stacks $\RR_{\geq 0}\xrightarrow{\ov \Gamma} \MtR_{g,n}$ corresponding to the tropical curve $\ov \Gamma\in  \big|\MtR_{g,n}\big|(\star)=\MtR_{g,n}(\RR_{\geq 0})$ factors as 
$$\RR_{\geq 0}\xlongrightarrow{u} \sigma \xlongrightarrow{\wt \Gamma} \MtR_{g,n}
$$
where the last morphism $\sigma \xrightarrow{\wt \Gamma} \MtR_{g,n}$ corresponds to the tropical curve $\wt \Gamma$ seen as an element of $\MtR_{g,n}(\sigma)$. 

Since the fiber of $\Phi^{\trop}$ over $\sigma \xrightarrow{\wt \Gamma} \Mt_{g,n}$ is equal to $\Jac_{\wt \Gamma/\sigma}\to \sigma$ by Theorem \ref{T:fib-forget}, we get that the fiber of $\Phi^{\trop,\,\RR}$ over $\RR_{\geq 0}\xrightarrow{\ov \Gamma} \MtR_{g,n}$ is equal to 
$$\RR_{\geq 0}\times_{\sigma} \left(\sigma\times_{\MtR_{g,n}} \JtR_{g,n}\right)=\RR_{\geq 0}\times_{\sigma} (\Jac_{\wt{\Gamma}/\sigma})^{\RR}\longrightarrow \RR_{\geq 0},$$
where $(\Jac_{\wt{\Gamma}/\sigma})^{\RR}$ is the real cone space associated to $\Jac_{\wt{\Gamma}/\sigma}$. 

Passing to the associated  topological stacks, we get that the desired fiber $|\Phi^{\trop}|^{-1}(\ov \Gamma)$ is equal to the fiber over $1\in |\RR_{\geq 0}|\xrightarrow{|u|} |\sigma|$ of the morphism of topological spaces $|\Jac_{\wt \Gamma/\sigma}|\to |\sigma|$, which is equal to 
$\Jac_{\ov \Gamma}$ as it follows by comparing Definitions \ref{Jac-fiber} and \ref{JacTp-fiber} and the fact that $u^*(\wt \Gamma)=\ov \Gamma$ by construction. 
 \end{proof}


We now want to relate the Jacobian topological space $\Jac_{\ov \Gamma}$ of $\ov \Gamma\in \big|\MtR_{g,n}\big|(\star) $ with the tropical Picard variety $\Pic(\ov\Gamma)=\bigsqcup_{d\in \Z} \Pic^d(\ov\Gamma)$ associated to $\ov \Gamma$. Recall (see \textit{e.g.} \cite{MZ08}) that $\Pic^d(\ov\Gamma)$ parametrizes linear equivalence classes of divisors of degree $d$ on $\ov \Gamma$ and that, by the  tropical Abel-Jacobi theorem, we have an isomorphism 
$$A^d_{\ov \Gamma}\colon \Pic^d(\ov\Gamma) \xlongrightarrow{\cong}\Jac(\ov \Gamma):=\frac{H_1(\ov \Gamma,\RR)}{H_1(\ov \Gamma, \ZZ)}.$$
In particular, $\Pic^d(\ov \Gamma)$ is a real torus of dimension equal to $b_1(\ov \Gamma):=\dim H_1(\ov \Gamma,\RR)$. 

\begin{theorem}\label{T:Jac-Pic}
Let $\ov \Gamma\in \big|\MtR_{g,n}\big|(\star) $ be a stable tropical curve with real edge lengths and let $\ov G:=\GG(\Gamma)\in \SG_{g,n}$ be its underlying stable graph. 
\begin{enumerate}[label={\small\textrm{(\roman*)}}]
\item \label{T:Jac-Pic1} There exists a  surjective degree preserving continuous  map 
$$
\begin{aligned}
\alpha_{\ov \Gamma}\colon \Jac_{\ov \Gamma} & \longrightarrow \Pic(\ov \Gamma) \\
P(G,D, \rho)\ni (x_e)_{e\in E_{\exc}(G^{\st})} &\longmapsto \sum_{v\in V_{\exc}(G)} p_{x_{e_v}} +  \sum_{v\in V_{\nex}(G)} D(v),
\end{aligned}
$$
where $p_{x_{e_v}}$ is the point of $\ov \Gamma$ at distance $x_{e_v}$ from $\rho_V^{-1}(v)$. 
\item \label{T:Jac-Pic2} For any general universal stability condition $\phi\in V_{g,n}$, the above map restricts to an homeomorphism 
$$
\alpha_{\ov \Gamma}(\phi)\colon\Jac_{\ov \Gamma}(\phi) \xlongrightarrow{\cong} \Pic^{|\phi|}(\ov \Gamma).
$$
\end{enumerate}
\end{theorem}
Part \ref{T:Jac-Pic2} has been proved by Abreu-Pacini \cite[Theorem~5.10]{API} using a slightly different language. In fact, in a sequel to this paper we show that the natural map from $\calJ_{g,n}^{\mathrm{trop}}$ the universal tropical Picard variety $\calTroPic_{g,n,d}$ is a proper subdivision of cone stacks. This will, in particular, imply that the map $\alpha_{\ov \Gamma}(\phi)$ is a polytopal subdivision. 


\begin{proof}
Let us prove part \ref{T:Jac-Pic1}.  Arguing as in \cite[Theorem~5.10]{API}, %
it can be proved that, for any $(G,D,\rho)\in \QD_{\ov G}$, the restriction map 
$$
(\alpha_{\ov \Gamma})_{|P(G,D, \rho)}\colon P(G,D,\rho)\longrightarrow \Pic^{\deg(D)}(\ov \Gamma)\xlongrightarrow[A^{\deg(D)}_{\ov \Gamma}]{\cong}\Jac(\ov \Gamma).
$$
is affine, and hence continuous. Since $\Jac_{\ov \Gamma}$ is defined as the colimit of the polytopes $P(G,D, \rho)$, as $(G,D,\rho)$ varies in $\QD_{\ov G}$, it follows that $\alpha_{\ov \Gamma}$ is continuous. The map $\alpha_{\ov \Gamma}$ preserves the degree since  $D(v)=1$ for every $v\in V_{\exc}(G)$. Finally, the fact that $\alpha_{\ov \Gamma}$ is surjective follows from the fact that every divisor $D$ on $\ov \Gamma$ is linearly equivalent to an admissible divisor on a quasistable model of $\Gamma$. 

Part \ref{T:Jac-Pic2}: the map $\alpha_{\ov \Gamma}(\phi)$ is continuous by part~\ref{T:Jac-Pic1} and it is bijective by \cite[Theorem 5.6]{API} (also see the sequel, where we generalize this), and it is a homeomorphism since $\Jac_{\ov \Gamma}(\phi)$ is compact (being a colimit of polytopes) and $\Pic^{|\phi|}(\ov \Gamma)$ is Hausdorff (being a real torus). 
\end{proof}

\begin{remark}\label{C:Namikawa}
The isomorphism of Theorem \ref{T:Jac-Pic}\ref{T:Jac-Pic2} induce an admissible polytopal decomposition of the real torus $\Pic^{|\phi|}(\ov \Gamma)$ (depending on the chosen $\phi$) which coincides with the Namikawa decomposition studied in \cite[\S 5.3]{ChristPayneShen} (generalizing the decompositions studied in \cite[\S 6]{OdaSeshadri}.
\end{remark}

By combining Theorem \ref{T:fibT-forget} and Theorem \ref{T:Jac-Pic}\ref{T:Jac-Pic2}, we obtain the following

\begin{corollary}\label{C:fibT-forget}
For any general universal stability condition $\phi\in V_{g,n}$, the fiber of the forgetfull-stabilization morphism  \eqref{PhitropTop} of topological stacks \eqref{PhitropTop}
\begin{equation*}
\big\vert\Phi^{\mathrm{trop}}\big\vert\colon \big\vert\Jt_{g,n}(\phi)\big\vert\longrightarrow \big\vert\Mt_{g,n}\big\vert.
\end{equation*}
over a point  $\ov \Gamma\in \big|\MtR_{g,n}\big|(\star) $ is canonically homeomorphic to $\Pic^{|\phi|}(\ov \Gamma)$.
\end{corollary}

\begin{remark}\label{R:fib-ConeComplex} 
Using Theorem \ref{T:fibT-forget} and arguing as in \cite[Theorem~5.14]{API}, it can be proved that the fiber of the forgetful-stabilization morphism \eqref{PhiTrop} of generalized cone complexes 
$$\Phi^{\mathrm{trop}}\colon\JJ_{g,n}\longrightarrow \MM_{g,n}$$ 
over a point  $\ov \Gamma\in \MM_{g,n}$ is homeomorphic to  $\Jac_{\ov \Gamma}/\Aut(\ov \Gamma)$.

In particular, using Theorem \ref{T:Jac-Pic}\ref{T:Jac-Pic2}, it follows that, for any general universal stability condition $\phi\in V_{g,n}$, the fiber of the forgetfull-stabilization morphism 
\eqref{PhiTrop} of generalized cone complexes 
$$\Phi^{\mathrm{trop}}\colon\JJ_{g,n}(\phi)\longrightarrow \MM_{g,n}$$ 
is canonically homeomorphic to $\Pic^{|\phi|}(\ov \Gamma)/\Aut(\ov \Gamma)$. 
\end{remark}


\subsection{Canonical stability condition and break divisors}

For every hyperbolic pair $(g,n)$ and every degree $d\in \Z$, we can  define the \emph{log-canonical stability condition} $\phi^{can}\in V_{g,n}^d$ by 
\begin{equation*}
\phi^{can}_G(v)=\frac{d}{2g-2+n}\big(2h(v)-2+\val(v)\big)
\end{equation*}
for $G\in\calG\calS_{g,n}$ and $v\in V(G)$.

 
When $n=0$, as explained in \cite[Remark 5.14]{KP}, the stack $\calJbar_{g,0}(\phi^{can})$ is isomorphic to Caporaso's universal compactified Jacobian \cite{Cap94, Cap08}, which is studied in detail in \cite{Mel09}, \cite{BFV}, \cite{BMV}, \cite{MV14}, \cite{BMV}, \cite{BFMV}, \cite{CMKV2}, \cite{CMKV3} and is, in turn, isomorphic to Pandharipande's universal compactified Jacobian  \cite{Pandharipande} (see \cite{EP}). 
Note that in this case $\phi^{can}$ is general precisely when $d-g+1$ and $2g-2$ are coprime (see \cite[Proposition~6.2]{Cap94} and \cite[Remark~5.12]{KP}).

In the case $d=g$ and $n=0$, the tropical shadow of these constructions has been discovered independently in \cite{ABKS} under the name \emph{break divisor}. An admissible divisor $D$ on a quasi-stable graph $G$ of type $(g,0)$ is said to be a \emph{break divisor}, if there is a spanning tree $T$ of $G$ as well as a map $\phi\colon E(G)-E(T)\rightarrow V(G)$ such that $\phi(e)$ is adjacent to $e$ and
\begin{equation*}
D=\sum_{v\in V(G)}h(v)+\sum_{e\in E(G)-E(T)}\phi(e) \ .
\end{equation*}
Such a break divisor is automatically of degree $g$, since $g=b_1(G)+\sum_{v\in V(G)}h(v)$. By \cite[Lemma~5.13]{ChristPayneShen}, an admissible divisor is $\phi_G^{can}$-stable if and only it is a break divisor. So Theorem \ref{T:Jac-Pic} (and thus \cite{API}) also recovers the main result of \cite{ABKS} stating that every divisor of degree $g$ on a metric graph $\Gamma$ has a unique break divisor representative. In Figure \ref{figure_breakdivisorsthetagraph} we show the induced polyhedral decomposition of $\Pic_g(\Gamma)$ in the fiber over a stable tropical curve of genus whose underlying graph is a theta-graph (without vertex weights). 

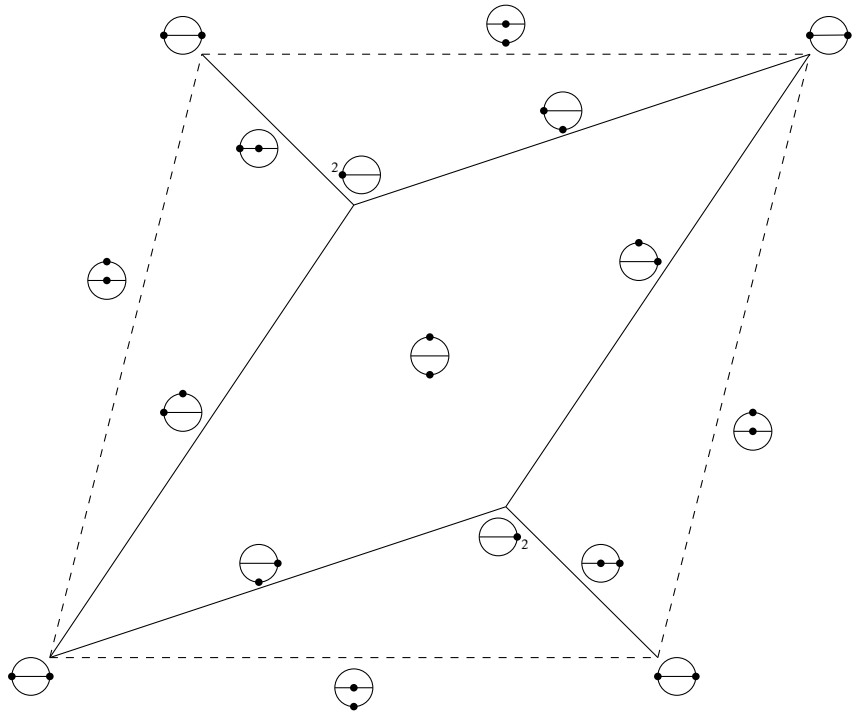
\begin{figure}[h]
\begin{tikzpicture}

\draw[dashed] (0,0) -- (2,8);
\draw(0,0) -- (4,6);
\draw[dashed] (0,0) -- (8,0);
\draw (0,0) -- (6,2);
\draw[dashed] (8,0) -- (10,8);
\draw (6,2) -- (10,8);
\draw (4,6) -- (10,8);
\draw[dashed] (2,8) -- (10,8);
\draw (2,8) -- (4,6);
\draw (8,0) -- (6,2);

\draw (-0.25,-0.25) circle (0.25); 
\draw (-0.5,-0.25) -- (0,-0.25);
\fill (-0.5,-0.25) circle (0.05);
\fill (0,-0.25) circle (0.05);

\draw (5,4) circle (0.25); 
\draw (4.75,4) -- (5.25,4);
\fill (5,4.25) circle (0.05);
\fill (5,3.75) circle (0.05);

\draw (10.25,8.25) circle (0.25); 
\draw (10,8.25) -- (10.5,8.255);
\fill (10,8.25) circle (0.05);
\fill (10.5,8.25) circle (0.05);

\draw (1.75,8.25) circle (0.25); 
\draw (1.5,8.25) -- (2,8.25);
\fill (1.5,8.25) circle (0.05);
\fill (2,8.25) circle (0.05);

\draw (8.25,-0.25) circle (0.25); 
\draw (8,-0.25) -- (8.5,-0.25);
\fill (8,-0.25) circle (0.05);
\fill (8.5,-0.25) circle (0.05);

\draw (0.75,5) circle (0.25); 
\draw (0.5,5) -- (1,5);
\fill (0.75,5) circle (0.05);
\fill (0.75,5.25) circle (0.05);

\draw (9.25,3) circle (0.25); 
\draw (9,3) -- (9.5,3);
\fill (9.25,3) circle (0.05);
\fill (9.25,3.25) circle (0.05);

\draw (4,-0.4) circle (0.25); 
\draw (3.75,-0.4) -- (4.25,-0.4);
\fill (4,-0.4) circle (0.05);
\fill (4,-0.65) circle (0.05);

\draw (6,8.4) circle (0.25); 
\draw (5.75,8.4) -- (6.25,8.4);
\fill (6,8.4) circle (0.05);
\fill (6,8.15) circle (0.05);

\draw (4.1,6.4) circle (0.25); 
\draw (3.85,6.4) -- (4.35,6.4);
\fill (3.85,6.4) circle (0.05);
\node[scale=0.5] (2) at (3.75,6.5) {$2$};

\draw (5.9,1.6) circle (0.25); 
\draw (5.65,1.6) -- (6.15,1.6);
\fill (6.15,1.6) circle (0.05);
\node[scale=0.5] (2) at (6.25,1.5) {$2$};

\draw (1.75,3.25) circle (0.25); 
\draw (1.5,3.25) -- (2,3.25);
\fill (1.5,3.25) circle (0.05);
\fill (1.75,3.5) circle (0.05);

\draw (7.75,5.25) circle (0.25); 
\draw (7.5,5.25) -- (8,5.25);
\fill (8,5.25) circle (0.05);
\fill (7.75,5.5) circle (0.05);

\draw (2.75,1.25) circle (0.25);
\draw (2.5,1.25) -- (3,1.25);
\fill (3,1.25) circle (0.05);
\fill (2.75,1) circle (0.05);

\draw (6.75,7.25) circle (0.25); 
\draw (6.5,7.25) -- (7,7.25);
\fill (6.5,7.25) circle (0.05);
\fill (6.75,7) circle (0.05);

\draw (2.75,6.75) circle (0.25); 
\draw (2.5,6.75) -- (3,6.75);
\fill (2.5,6.75) circle (0.05);
\fill (2.75,6.75) circle (0.05);

\draw (7.25,1.25) circle (0.25); 
\draw (7,1.25) -- (7.5,1.25);
\fill (7.5,1.25) circle (0.05);
\fill (7.25,1.25) circle (0.05);
\end{tikzpicture}
\caption{The break divisor decomposition of the Jacobian of the theta
  graph.}
\label{figure_breakdivisorsthetagraph}
\end{figure}

\begin{remark}
Let $X$ be a smooth projective curve over a non-Archimedean field $K$
and let $\calX$ be its stable model over the valuation ring $R$. The
fiber product $\calX\times_{\calMbar_g}\J_g(\phi^{can})$ is a proper
polystable model for the degree $g$ Jacobian $J_g(X)$ of $X$. One may
now argue as in \cite{BrandtUlirsch_sympow} and use the combinatorial
description of the strata from Section \ref{S:comp-univ} above to
identify the non-Archimedean skeleton associated to this polystable
model (in the sense of \cite{Berkovich_smooth=contractible}) with the
tropical Jacobian $J_g^{\mathrm{trop}}(\Gamma_X)$ of the tropicalization
$\Gamma_X$ of $X$. This constitutes another path towards the main
result of \cite{BR15} saying that \emph{the skeleton of the Jacobian
  is the Jacobian of the skeleton}.  
\end{remark}


\end{document}